\documentclass[hidelinks,onefignum,onetabnum]{siamart251216}

\usepackage{lipsum}
\usepackage{amsfonts}
\usepackage{graphicx}
\usepackage{epstopdf}
\usepackage{algorithmic}
\ifpdf
  \DeclareGraphicsExtensions{.eps,.pdf,.png,.jpg}
\else
  \DeclareGraphicsExtensions{.eps}
\fi

\usepackage{amsmath} 
\usepackage{amssymb} 
\usepackage{mathrsfs} 
\usepackage{mathtools} 
\usepackage{blkarray}
\usepackage{bm}
\usepackage{cite}

\usepackage{float} 
\usepackage{algorithm}
\usepackage{algorithmic}

\usepackage{enumerate} 
\usepackage{booktabs} 
\usepackage{multirow} 
\usepackage{diagbox} 
\usepackage{subcaption} 

\newsiamremark{remark}{Remark}
\newsiamremark{hypothesis}{Hypothesis}
\crefname{hypothesis}{Hypothesis}{Hypotheses}
\newsiamthm{claim}{Claim}
\newsiamremark{fact}{Fact}
\crefname{fact}{Fact}{Facts}

\headers{A Superfast Direct Solver for 2D Inverse NUDFT}{YINGZHOU LI AND JINGYU LIU}

\title{A Superfast Direct Solver for the 2D Type-II Inverse Nonuniform Discrete Fourier Transform Based on Hierarchically Semiseparable Matrices
}

\author{Yingzhou Li\thanks{School of Mathematical Sciences, Fudan University; Shanghai Key Laboratory for Contemporary Applied Mathematics, Fudan University 
  (\email{yingzhouli@fudan.edu.cn}).}
\and Jingyu Liu\thanks{School of Mathematical Sciences, Fudan University
  (\email{jyliu22@m.fudan.edu.cn}).}
}

\DeclareFontFamily{U}{mathx}{}
\DeclareFontShape{U}{mathx}{m}{n}{<-> mathx10}{}
\DeclareSymbolFont{mathx}{U}{mathx}{m}{n}
\DeclareMathAccent{\widecheck}{0}{mathx}{"71}
\makeatother

\newcommand*{\trans}{\top}
\newcommand*{\herm}{*}
\newcommand*{\pinv}{\dagger}

\newcommand*{\complex}{\mathbb{C}}
\newcommand*{\unitcircle}{\mathbb{S}}

\newcommand*{\Mat}[1]{\bm{#1}}
\newcommand*{\tMat}[1]{\widetilde{\Mat{#1}}}
\newcommand*{\hMat}[1]{\widehat{\Mat{#1}}}
\newcommand*{\cMat}[1]{\widecheck{\Mat{#1}}}
\newcommand*{\rank}{\operatorname{rank}}

\newcommand*{\bigO}{\mathscr{O}}
\newcommand*{\e}{\mathrm{e}}
\newcommand*{\imag}{\mathrm{i}}
\newcommand*{\compl}{\mathrm{c}}
\newcommand*{\fro}{\mathrm{F}}

\newcommand*{\indexset}[1]{\mathcal{#1}}

\newcommand*{\mylog}[2][]{\log^{#1} #2}

\newcommand*{\idmat}{\Mat{I}}
\newcommand*{\ranktol}{\varepsilon}

\newcommand*{\h}{\mathrm{h}}

\newcommand*{\construct}{\mathrm{c}}
\newcommand*{\factor}{\mathrm{f}}
\newcommand*{\solve}{\mathrm{s}}
\newcommand*{\pre}{\mathrm{pre}}
\newcommand*{\iter}{\mathrm{iter}}

\newcommand*{\exact}{\mathrm{ex}}

\newcommand*{\faceprod}{\bullet}

\newcommand*{\tree}[1]{\mathsf{#1}}
\newcommand*{\sib}{\operatorname{sib}}
\newcommand*{\ch}{\operatorname{ch}}
\newcommand*{\level}{\operatorname{level}}

\newcommand*{\dirx}{x}
\newcommand*{\diry}{y}

\newcommand*{\numcoldir}{n}
\newcommand*{\numcolx}{\numcoldir^{[\dirx]}}
\newcommand*{\numcoly}{\numcoldir^{[\diry]}}
\newcommand*{\numcol}{N}

\newcommand*{\numrowdir}{m}
\newcommand*{\numrowx}{\numrowdir^{[\dirx]}}
\newcommand*{\numrowy}{\numrowdir^{[\diry]}}
\newcommand*{\numrow}{M}

\newcommand*{\colind}{k}
\newcommand*{\colindpair}{\ell}
\newcommand*{\colinddir}{k}
\newcommand*{\colindx}{\colinddir^{[\dirx]}}
\newcommand*{\colindy}{\colinddir^{[\diry]}}

\newcommand*{\rowind}{j}
\newcommand*{\rowinddir}{j}
\newcommand*{\rowindx}{\rowinddir^{[\dirx]}}
\newcommand*{\rowindy}{\rowinddir^{[\diry]}}

\newcommand*{\coeff}{c}
\newcommand*{\targetv}{f}

\newcommand*{\pt}{x}
\newcommand*{\ptx}{x}
\newcommand*{\pty}{y}

\newcommand*{\nudftmat}{\Mat{A}}
\newcommand*{\nudftmatx}{\nudftmat^{[\dirx]}}
\newcommand*{\nudftmaty}{\nudftmat^{[\diry]}}

\newcommand*{\tnudftmat}{\Mat{G}}
\newcommand*{\tnudftmatx}{\tnudftmat^{[\dirx]}}
\newcommand*{\tnudftmaty}{\tnudftmat^{[\diry]}}

\newcommand*{\dftmat}{\Mat{F}}
\newcommand*{\dftmatx}{\dftmat^{[\dirx]}}
\newcommand*{\dftmaty}{\dftmat^{[\diry]}}
\newcommand*{\dftmatinv}{\dftmat^{-1}}
\newcommand*{\dftmatxinv}{\dftmat^{[\dirx], -1}}
\newcommand*{\dftmatyinv}{\dftmat^{[\diry], -1}}

\newcommand*{\nudfthss}{\tMat{G}}
\newcommand*{\nudfthssx}{\nudfthss^{[\dirx]}}
\newcommand*{\nudfthssy}{\nudfthss^{[\diry]}}

\newcommand*{\coeffvec}{\Mat{c}}
\newcommand*{\targetvvec}{\Mat{f}}

\newcommand*{\colindset}{\indexset{K}}
\newcommand*{\colindsetx}{\colindset^{[\dirx]}}
\newcommand*{\colindsety}{\colindset^{[\diry]}}
\newcommand*{\rowindset}{\indexset{J}}

\newcommand*{\rowptS}{\gamma}
\newcommand*{\colptS}{\xi}
\newcommand*{\rowptSx}{\rowptS^{[\dirx]}}
\newcommand*{\rowptSy}{\rowptS^{[\diry]}}
\newcommand*{\colptSx}{\colptS^{[\dirx]}}
\newcommand*{\colptSy}{\colptS^{[\diry]}}

\newcommand*{\nudftkernel}{\Psi}
\newcommand*{\nudftkernelx}{\nudftkernel^{[\dirx]}}
\newcommand*{\nudftkernely}{\nudftkernel^{[\diry]}}

\newcommand*{\rowptSmat}{\Mat{\Gamma}}
\newcommand*{\circulantC}{\Mat{C}}
\newcommand*{\circulantD}{\Mat{D}}
\newcommand*{\unitroot}{\zeta}

\newcommand*{\nudftumat}{\Mat{u}}
\newcommand*{\nudftusca}{u}
\newcommand*{\nudftvmat}{\Mat{v}}
\newcommand*{\nudftvsca}{v}

\newcommand*{\colsubindset}{\indexset{C}}
\newcommand*{\colsubindsetx}{\colsubindset^{[\dirx]}}
\newcommand*{\colsubindsetxcompl}{\colsubindset^{[\dirx], \compl}}
\newcommand*{\colsubindsety}{\colsubindset^{[\diry]}}
\newcommand*{\colsubindsetycompl}{\colsubindset^{[\diry], \compl}}
\newcommand*{\rowsubindset}{\indexset{R}}

\newcommand*{\myblue}{\mathrm{b}}
\newcommand*{\mygreen}{\mathrm{g}}
\newcommand*{\myyellow}{\mathrm{y}}
\newcommand*{\fast}{\mathrm{fast}}
\newcommand*{\col}{\mathrm{c}}

\newcommand*{\nudftrownorm}{Q}
\newcommand*{\nudftrownormx}{\nudftrownorm^{[\dirx]}}
\newcommand*{\nudftrownormy}{\nudftrownorm^{[\diry]}}
\newcommand*{\nudftconst}{C}

\newcommand*{\nudftintervalx}{[a, b)}
\newcommand*{\nudftintervaly}{[c, d)}

\newcommand*{\nudftrowpt}{\Mat{\gamma}}
\newcommand*{\nudftcolpt}{\Mat{\xi}}

\newcommand*{\factorcomplexity}{\bigO\bigl(\numrow + \numcol^{3 / 2} \mylog[3]{\numcol}\bigr)}
\newcommand*{\solvecomplexity}{\bigO\bigl(\numrow + \numcol \mylog[3]{\numcol}\bigr)}
\newcommand*{\offlinecomplexity}{\bigO\bigl(\numrow \mylog[2]{\numcol} + \numcol^{3 / 2} \mylog[3]{\numcol}\bigr)}
\newcommand*{\onlinecomplexity}{\bigO\bigl(\numrow + \numcol \mylog[3]{\numcol}\bigr)}

\newcommand*{\HSStree}{\tree{T}}
\newcommand*{\HSStreex}{\HSStree^{[\dirx]}}
\newcommand*{\HSStreey}{\HSStree^{[\diry]}}

\newcommand*{\HSSrowindset}{\indexset{J}}
\newcommand*{\HSScolindset}{\indexset{K}}
\newcommand*{\HSScolindsetx}{\indexset{K}^{[\dirx]}}
\newcommand*{\HSScolindsety}{\indexset{K}^{[\diry]}}

\newcommand*{\HSSH}{\Mat{H}}
\newcommand*{\HSSHx}{\HSSH^{[\dirx]}}
\newcommand*{\HSSHy}{\HSSH^{[\diry]}}

\newcommand*{\HSSroot}{\rho}
\newcommand*{\HSSrootx}{\HSSroot^{[\dirx]}}
\newcommand*{\HSSrooty}{\HSSroot^{[\diry]}}

\newcommand*{\HSSmaxlevel}{L}
\newcommand*{\HSSmylevel}{\ell}

\newcommand*{\HSSnode}{\tau}
\newcommand*{\HSSnodex}{\HSSnode^{[\dirx]}}
\newcommand*{\HSSnodey}{\HSSnode^{[\diry]}}

\newcommand*{\HSSrownode}{\tau}

\newcommand*{\HSScolnode}{\sigma}

\newcommand*{\HSSch}{\alpha}
\newcommand*{\HSSchch}{\chi}
\newcommand*{\HSSchx}{\HSSch^{[\dirx]}}
\newcommand*{\HSSchy}{\HSSch^{[\diry]}}

\newcommand*{\HSSchsib}{\beta}

\newcommand*{\HSSchind}{p}
\newcommand*{\HSSchrowind}{p}
\newcommand*{\HSSchcolind}{q}
\newcommand*{\HSSchnum}{b}

\newcommand*{\HSSD}{\Mat{D}}
\newcommand*{\HSSU}{\Mat{U}}
\newcommand*{\HSSV}{\Mat{V}}
\newcommand*{\HSSR}{\Mat{R}}
\newcommand*{\HSSW}{\Mat{W}}
\newcommand*{\HSSB}{\Mat{B}}

\newcommand*{\HSSrank}{r}
\newcommand*{\HSSrankx}{\HSSrank^{[\dirx]}}
\newcommand*{\HSSranky}{\HSSrank^{[\diry]}}

\newcommand*{\HSSdiagsize}{d}
\newcommand*{\HSSdiagsizex}{\HSSdiagsize^{[\dirx]}}
\newcommand*{\HSSdiagsizey}{\HSSdiagsize^{[\diry]}}

\newcommand*{\HSSc}{\Mat{c}}
\newcommand*{\HSSf}{\Mat{f}}

\newcommand*{\HSStU}{\tMat{U}}
\newcommand*{\HSStR}{\tMat{R}}
\newcommand*{\HSShR}{\hMat{R}}
\newcommand*{\HSSF}{\Mat{F}}

\newcommand*{\HSStV}{\tMat{V}}
\newcommand*{\HSStW}{\tMat{W}}
\newcommand*{\HSShW}{\hMat{W}}
\newcommand*{\HSSG}{\Mat{G}}

\newcommand*{\HSStB}{\tMat{B}}

\newcommand*{\HSSC}{\Mat{C}}
\newcommand*{\HSSS}{\Mat{S}}

\newcommand*{\HSSX}{\Mat{X}}
\newcommand*{\HSSY}{\Mat{Y}}

\newcommand*{\HSSOm}{\Mat{\Omega}}
\newcommand*{\HSScD}{\cMat{D}}
\newcommand*{\HSScU}{\cMat{U}}

\newcommand*{\HSSnoderowsize}{m}
\newcommand*{\HSSnodecolsize}{n}
\newcommand*{\HSSnodereducedrowsize}{\check{m}}

\newcommand*{\HSSQ}{\Mat{Q}}
\newcommand*{\HSShV}{\hMat{V}}
\newcommand*{\HSStD}{\tMat{D}}

\newcommand*{\HSShD}{\hMat{D}}
\newcommand*{\HSSP}{\Mat{P}}
\newcommand*{\HSShU}{\hMat{U}}

\newcommand*{\HSScf}{\cMat{f}}
\newcommand*{\HSShf}{\hMat{f}}

\newcommand*{\HSShc}{\hMat{c}}

\newcommand*{\HSShb}{\hMat{b}}

\newcommand*{\indset}{\indexset{J}}

\newcommand*{\GENmat}{\Mat{A}}
\newcommand*{\GENmatapprox}{\tMat{A}}
\newcommand*{\GENrowsize}{m}
\newcommand*{\GENcolsize}{n}

\newcommand*{\GENLRleft}{\Mat{U}}
\newcommand*{\GENLRleftA}{\GENLRleft_{1}}
\newcommand*{\GENLRleftB}{\GENLRleft_{2}}

\newcommand*{\GENLRright}{\Mat{V}}
\newcommand*{\GENLRrightA}{\GENLRright_{1}}
\newcommand*{\GENLRrightB}{\GENLRright_{2}}

\newcommand*{\GENrank}{r}
\newcommand*{\GENrankA}{\GENrank_{1}}
\newcommand*{\GENrankB}{\GENrank_{2}}

\newcommand*{\GENmatA}{\Mat{A}_{1}}
\newcommand*{\GENmatB}{\Mat{A}_{2}}
\newcommand*{\GENcolsizeA}{\GENcolsize_{1}}
\newcommand*{\GENcolsizeB}{\GENcolsize_{2}}

\newcommand*{\GENindAB}{t}

\newcommand*{\GENzmat}{\Mat{z}}
\newcommand*{\GENzsca}{z}
\newcommand*{\GENwmat}{\Mat{w}}
\newcommand*{\GENwsca}{w}

\newcommand*{\tsca}{t}

\newcommand*{\myJ}{\indexset{J}}

\newcommand*{\myKx}{\indexset{K}^{[x]}}
\newcommand*{\myKy}{\indexset{K}^{[y]}}
\newcommand*{\myK}{\indexset{K}}

\newcommand*{\myTx}{\tree{T}^{[x]}}
\newcommand*{\myTy}{\tree{T}^{[y]}}
\newcommand*{\myT}{\tree{T}}

\newcommand*{\myHx}{\Mat{H}^{[x]}}
\newcommand*{\myHy}{\Mat{H}^{[y]}}
\newcommand*{\myH}{\Mat{H}}

\newcommand*{\myDx}{\Mat{D}^{[x]}}
\newcommand*{\myDy}{\Mat{D}^{[y]}}
\newcommand*{\myD}{\Mat{D}}

\newcommand*{\myUx}{\Mat{U}^{[x]}}
\newcommand*{\myUy}{\Mat{U}^{[y]}}
\newcommand*{\myU}{\Mat{U}}

\newcommand*{\myVx}{\Mat{V}^{[x]}}
\newcommand*{\myVy}{\Mat{V}^{[y]}}
\newcommand*{\myV}{\Mat{V}}

\newcommand*{\myVxh}{\Mat{V}^{[x], \herm}}
\newcommand*{\myVyh}{\Mat{V}^{[y], \herm}}
\newcommand*{\myVh}{\Mat{V}^{\herm}}

\newcommand*{\myRx}{\Mat{R}^{[x]}}
\newcommand*{\myRy}{\Mat{R}^{[y]}}
\newcommand*{\myR}{\Mat{R}}

\newcommand*{\myWx}{\Mat{W}^{[x]}}
\newcommand*{\myWy}{\Mat{W}^{[y]}}
\newcommand*{\myW}{\Mat{W}}

\newcommand*{\myBx}{\Mat{B}^{[x]}}
\newcommand*{\myBy}{\Mat{B}^{[y]}}
\newcommand*{\myB}{\Mat{B}}

\newcommand*{\mytx}{\tau}
\newcommand*{\myty}{\sigma}
\newcommand*{\myt}{\lambda}

\newcommand*{\mychx}{\alpha}
\newcommand*{\mychy}{\beta}
\newcommand*{\mych}{\xi}

\newcommand*{\myrowchindx}{p}
\newcommand*{\myrowchindy}{q}
\newcommand*{\mycolchindx}{r}
\newcommand*{\mycolchindy}{s}

\usepackage{amsopn}
\DeclareMathOperator{\diag}{diag}

\newcommand{\STAGE}[1]{%
  \item[] \textit{#1}%
}

\begin{document}

\maketitle

\begin{abstract}
Reconstructing images on Cartesian grids from nonuniform Fourier samples is a basic computational task in Fourier reconstruction methods for computed tomography.
This paper proposes a direct solver for the 2D type-II nonuniform discrete Fourier transform~(NUDFT) based on the face-splitting product of hierarchically semiseparable~(HSS) matrices.
We define the face-splitting product of two HSS trees and prove that the face-splitting product of two HSS matrices admits an HSS representation on the resulting tree.
The proof is constructive, providing explicit formulas for the HSS generators and bounds on the off-diagonal ranks at each level.
For the NUDFT, a Fourier change of basis expresses the transformed 2D matrix as a face-splitting product of two transformed 1D NUDFT matrices.
This identity yields an algebraic construction of a 2D HSS approximation from two 1D HSS approximations, for which we establish a Frobenius-norm approximation error bound.
Combining this construction with recompression, URV factorization, and a two-dimensional inverse fast Fourier transform gives a direct solver for the associated least-squares problem.
For fixed accuracy, square frequency grids, and balanced HSS trees with fixed leaf column sizes, the solver has an offline complexity of \(\offlinecomplexity\) and an online complexity of \(\onlinecomplexity\) per right-hand side, where \(\numrow\) and \(\numcol\) are the numbers of sample points and unknown coefficients, respectively.
Numerical experiments on random and polar grids illustrate the computational scaling and reconstruction accuracy of the method, and demonstrate a substantial reduction in lsqr iteration counts when the direct solver is used as a preconditioner.
\end{abstract}

\begin{keywords}
nonuniform discrete Fourier transform, hierarchically semiseparable matrices, fast direct solvers, least-squares problems
\end{keywords}

\begin{MSCcodes}
65T50, 65F55, 65F20
\end{MSCcodes}

\section{Introduction} \label{sec:introduction}

Fourier transforms are basic computational tools in signal processing and image reconstruction~\cite{Bagchi_Mitra_2012, Hansen_Jorgensen_Lionheart_2021}.
On a uniform Cartesian grid, the fast Fourier transform~(FFT) makes both the evaluation and inversion of discrete Fourier transforms efficient.
In imaging applications, the measurement geometry can instead produce Fourier samples at nonuniform locations, while the image is represented on a Cartesian grid.
Reconstruction then requires the solution of an inverse nonuniform Fourier problem.

This paper considers the two-dimensional type-II \emph{nonuniform discrete Fourier transform}~(NUDFT) of the following form:
\begin{equation} \label{eq:typeII_2d_nudft}
    \targetv_{\rowind} = \sum_{\colindx = 0}^{\numcolx - 1} \sum_{\colindy = 0}^{\numcoly - 1} \coeff_{\colindx, \colindy} \e^{-2 \pi \imag (\colindx \ptx_{\rowind} + \colindy \pty_{\rowind})}, \quad 0 \leq \rowind \leq \numrow - 1,
\end{equation}
where the \emph{sample points} \(\{(\ptx_{\rowind}, \pty_{\rowind})\}\) are distributed arbitrarily in \([0, 1)^2\) and the \emph{frequencies} \(\{(\colindx, \colindy)\}\) are distributed on a Cartesian grid of contiguous integers.
Throughout the paper, the number of sample points \(\numrow\) is assumed to be larger than or equal to the number of frequencies \(\numcol = \numcolx \numcoly\).
When the \emph{coefficients} \(\{\coeff_{\colindx, \colindy}\}\) are given and one aims to compute the \emph{target values} \(\{\targetv_{\rowind}\}\), the problem is called the \emph{forward} NUDFT. 
When the target values \(\{\targetv_{\rowind}\}\) are given and one aims to determine the coefficients \(\{\coeff_{\colindx, \colindy}\}\), the problem is called the \emph{inverse} NUDFT.

Both the forward and inverse NUDFT have wide applications in scientific computing, such as signal processing~\cite{Bagchi_Mitra_2012}, image reconstruction~\cite{Bourgeois_Wajer_vanOrmondt_GraveronDemilly_2001}, and fast convolution~\cite{Jiang_Greengard_2025}.
A motivating application is Fourier reconstruction in computed tomography~(CT).
For parallel-beam projections, the Fourier slice theorem~\cite{Hansen_Jorgensen_Lionheart_2021} states that
\begin{equation*}
    \widehat{p}_{\theta}(\omega)
    =
    \widehat{u}(\omega \cos \theta, \omega \sin \theta),
\end{equation*}
where \(u\) is the image and \(p_{\theta}\) is its projection at angle \(\theta\).
Thus, taking one-dimensional Fourier transforms of projections at different angles provides samples of the image Fourier transform along radial lines.
The resulting Fourier sampling locations are on a nonuniform polar grid, whereas the image is typically represented on a Cartesian grid.
Approximating the Fourier integrals by sums over a Cartesian image grid yields a discrete model of the form~\eqref{eq:typeII_2d_nudft}, after normalizing the sampling coordinates and incorporating known scaling and phase factors.
In this interpretation, \(\{\coeff_{\colindx, \colindy}\}\) represent image coefficients on the Cartesian grid, and \(\{(\ptx_{\rowind}, \pty_{\rowind})\}\) encode the nonuniform Fourier sampling locations.
Recovering the image coefficients therefore leads to an inverse type-II 2D NUDFT problem.

Let the NUDFT matrix \(\nudftmat\) be defined as\footnote{In this paper, the starting index of a matrix can be \(0\) or \(1\), depending on the context.}
\begin{equation}
\label{eq:typeII_2d_nudft_matrix}
\begin{gathered}
  \nudftmat\bigl(\rowind,(\colindx,\colindy)\bigr)
  =
  \e^{-2\pi\imag
  \left(\colindx\ptx_{\rowind}+\colindy\pty_{\rowind}\right)},
  \\[2pt]
  0\leq \rowind\leq \numrow-1,
  \qquad
  0\leq \colindx\leq \numcolx-1,
  \qquad
  0\leq \colindy\leq \numcoly-1.
\end{gathered}
\end{equation}
where the column index \((\colindx, \colindy)\) is ordered in a lexicographical manner.
That is, we have an underlying \(1\)-to-\(1\) map \((\colindx, \colindy) \leftrightarrow \colindy + \colindx \numcoly\).
The forward NUDFT can be expressed as a matrix-vector multiplication \(\targetvvec = \nudftmat \coeffvec\) and the inverse NUDFT can be formulated as a linear least-squares problem \(\min_{\coeffvec} \|\nudftmat \coeffvec - \targetvvec\|_{2}\), where \(\coeffvec \in \complex^{\numcol}\) is the vectorized form of the coefficients \(\{\coeff_{\colindx, \colindy}\}\) and \(\targetvvec \in \complex^{\numrow}\) is the vector corresponding to the target values \(\{\targetv_{\rowind}\}\).
Direct evaluation costs \(\bigO(\numrow \numcol)\), while solving the least-squares problem by a dense QR factorization costs \(\bigO(\numrow \numcol^{2})\).

\subsection{Related Work} \label{subsec:related_work}

On the uniform Cartesian grid \((\ptx_{\rowindx}, \pty_{\rowindy}) = (\rowindx / \numcolx, \rowindy / \numcoly)\), the NUDFT reduces to the two-dimensional \emph{discrete Fourier transform}~(DFT), whose matrix is a Kronecker product of two one-dimensional DFT matrices.
The forward and inverse transforms can then be computed by the FFT~\cite{Cooley_Tukey_1965} and inverse FFT~(iFFT), respectively, in \(\bigO\bigl(\numcol \mylog{\numcol}\bigr)\) time.
For an oversampled uniform Cartesian grid with \(\numrowx \geq \numcolx\) and \(\numrowy \geq \numcoly\), the NUDFT matrix consists of columns of a larger DFT matrix, and both problems can still be solved using FFTs in \(\bigO\bigl(\numrow \mylog{\numrow}\bigr)\) time.

For nonuniform sampling, the FFT cannot be applied directly.
The forward transform can nevertheless be evaluated efficiently using \emph{nonuniform fast Fourier transform}~(NUFFT) algorithms~\cite{Anderson_Dahleh_1996, Barnett_Magland_Klinteberg_2019, Dutt_Rokhlin_1993, Dutt_Rokhlin_1995, Greengard_Lee_2004, Lee_Greengard_2005, Potts_Steidl_Tasche_2001, Ruiz-Antolín_Townsend_2018}.
For fixed accuracy, these algorithms typically require \(\bigO\bigl(\numrow + \numcol \mylog{\numcol}\bigr)\) operations.

For nonuniform sampling, the pseudoinverse of the NUDFT matrix generally differs from a scaled adjoint.
Iterative least-squares solvers such as lsqr~\cite{Paige_Saunders_1982} use NUFFT evaluations to apply the matrix and its adjoint efficiently.
Their convergence depends on the conditioning of the NUDFT matrix, which can deteriorate for highly nonuniform sampling.
Preconditioning can therefore be important for obtaining accurate solutions within a reasonable number of iterations.

Direct inversion methods have also been developed.
Kircheis and Potts proposed methods for the one-dimensional~\cite{Kircheis_Potts_2019} and multivariate~\cite{Kircheis_Potts_2023} NUFFT.
Their multivariate methods include approaches based on density compensation and optimization of sparse interpolation matrices.
Recently, in~\cite{Wilber_Epperly_Barnett_2025}, the authors proposed a direct method to solve the type-II NUDFT in 1D based on the \emph{hierarchically semiseparable}~(HSS) matrix, which has a complexity of \(\bigO\bigl((\numrow + \numcol) \mylog[2]{\numcol}\bigr)\).
Further, the authors in~\cite{Li_Liu_2025} extended this method to the case of type-III NUDFT in 1D, by factoring the type-III NUDFT matrix into a product of a type-II NUDFT matrix and an HSS matrix.
For the 2D problem, if the sample points form a tensor grid, the NUDFT matrix has a Kronecker product structure, and one can apply a 1D solver in each coordinate direction.
Nonuniform sampling does not generally have this tensor-product structure in Cartesian coordinates, so the inverse problem cannot be reduced to two independent 1D inversions.
This prevents a direct application of the 1D solver in~\cite{Wilber_Epperly_Barnett_2025} to the general 2D problem.

\subsection{Contributions} \label{subsec:contributions}

Our main contribution is an algebraic HSS construction for the transformed 2D NUDFT matrix \(\tnudftmat = \nudftmat \dftmatinv\), where \(\dftmat\) is the 2D DFT matrix.
We derive the separable kernel representation
\(\tnudftmat = \tnudftmatx \faceprod \tnudftmaty\),
where \(\tnudftmatx\) and \(\tnudftmaty\) are the transformed 1D NUDFT matrices in the \(\dirx\)- and \(\diry\)-directions, respectively.
This representation motivates the off-diagonal low-rank structure and provides the basis for our construction.

We define the face-splitting product of two HSS trees and prove that the face-splitting product of two HSS matrices admits an HSS representation on the resulting tree.
The constructive proof gives explicit formulas for the generators and bounds on the off-diagonal ranks at each level.
Applied to the two transformed 1D NUDFT matrices, this yields a 2D HSS approximation for which we establish a Frobenius-norm error bound.

After the algebraic face-splitting construction, the HSS approximation of \(\tnudftmat\) is recompressed by the algorithms in~\cite{Xia_2012}.
Then, the URV factorization~\cite{Xi_Xia_Cauley_Balakrishnan_2014} is used as a direct solver for the least-squares problem associated with \(\nudfthss\).
The construction and factorization of the HSS matrix constitute the \emph{offline} stage, whose complexity is \(\offlinecomplexity\).
Once the solver is built, the inverse NUDFT problem can be solved by applying the HSS solver followed by the iFFT.
This is referred to as the \emph{online} stage, and its complexity is \(\onlinecomplexity\).

The proposed direct solver can also be used as a preconditioner for lsqr.
Numerical experiments reconstruct a modified Shepp-Logan phantom from nonuniform Fourier samples on random and polar grids, examining the computational scaling and reconstruction accuracy.
The preconditioned iterations obtain accurate solutions using substantially fewer iterations than unpreconditioned lsqr.

\subsection{Organization} \label{subsec:organization}

Section~\ref{sec:preliminaries} introduces the notation and reviews HSS matrix algorithms and the one-dimensional direct inverse method.
Section~\ref{sec:low_rank_property_nudft_matrix} develops the face-splitting construction and its approximation error analysis.
Section~\ref{sec:typeII_2d_nudft_solver} presents the resulting two-dimensional direct solver and analyzes its complexity.
Numerical experiments are reported in Section~\ref{sec:numerical_results}, followed by conclusions and future directions in Section~\ref{sec:conclusion_future_directions}.

\section{Preliminaries} \label{sec:preliminaries}

\subsection{Notations and Preliminaries} \label{subsec:notations}

We use \(\unitcircle \coloneq \{\GENzsca \in \complex : |\GENzsca| = 1\}\) to denote the unit circle.
For a set \(\indset\), we use \(|\indset|\) to denote the cardinality of \(\indset\).
We use MATLAB notation to denote the submatrix of \(\GENmat\) with row and column indices in \(\rowsubindset\) and \(\colsubindset\), i.e., \(\GENmat(\rowsubindset, \colsubindset)\).
A colon in the index means all the indices in that dimension, e.g., \(\GENmat(:, \colsubindset)\) means the submatrix of \(\GENmat\) with all rows and column indices in \(\colsubindset\).

Let \(\|\cdot\|\) be a unitarily invariant matrix norm.
For a given tolerance \(0 < \ranktol < 1\), the numerical rank of a matrix \(\GENmat \in \complex^{\GENrowsize \times \GENcolsize}\), denoted by \(\rank_{\ranktol}(\GENmat)\), is defined as the smallest integer \(\GENrank\) such that there exists a matrix \(\GENmatapprox \in \complex^{\GENrowsize \times \GENcolsize}\) with \(\|\GENmat - \GENmatapprox\| < \ranktol \|\GENmat\|\) and \(\rank(\GENmatapprox) = \GENrank\).

For two matrices \(\GENmatA \in \complex^{\GENrowsize \times \GENcolsizeA}\) and \(\GENmatB \in \complex^{\GENrowsize \times \GENcolsizeB}\) of the same size of rows, the \emph{face-splitting product} of \(\GENmatA\) and \(\GENmatB\) is defined as the Kronecker product of the corresponding rows of \(\GENmatA\) and \(\GENmatB\), i.e.,
\begin{equation*}
  \GENmatA \faceprod \GENmatB \coloneq
  \begin{bmatrix}
    \GENmatA(1, :) \otimes \GENmatB(1, :)\\
    \GENmatA(2, :) \otimes \GENmatB(2, :)\\
    \vdots\\
    \GENmatA(\GENrowsize, :) \otimes \GENmatB(\GENrowsize, :)
  \end{bmatrix} \in \complex^{\GENrowsize \times (\GENcolsizeA \GENcolsizeB)}.
\end{equation*}
The Frobenius norm of the face-splitting product can be estimated by the following lemma, which can be proved through direct calculation.
\begin{lemma} \label{lem:faceprod_F_norm}
  Let \(\GENmatA \in \complex^{\GENrowsize \times \GENcolsizeA}\) and \(\GENmatB \in \complex^{\GENrowsize \times \GENcolsizeB}\), then
  \begin{equation*}
    \|\GENmatA \faceprod \GENmatB\|_{\fro}
    \leq \min \Bigl\{
      \max_{1 \leq \rowind \leq \GENrowsize} \bigl\{\|\GENmatA(\rowind, :) \|_{2}\bigr\} \|\GENmatB\|_{\fro}, 
      \max_{1 \leq \rowind \leq \GENrowsize} \bigl\{\|\GENmatB(\rowind, :) \|_{2}\bigr\} \|\GENmatA\|_{\fro}
      \Bigr\}.
  \end{equation*}
\end{lemma}

\subsection{HSS Matrices} \label{subsec:hss}

HSS matrices are a class of hierarchical matrices~\cite{Borm_Grasedyck_Hackbusch_2003, Hackbusch_1999, Xia_Chandrasekaran_Gu_Li_2010} widely used in scientific computing~\cite{Corona_Martinsson_Zorin_2015, Martinsson_Rokhlin_2005}.
We first introduce the definition of HSS tree, which is the underlying data structure for HSS matrices.

\begin{definition}[HSS tree,~\cite{Xi_Xia_Cauley_Balakrishnan_2014}] \label{def:hss_tree}
    A tree \(\HSStree\) is called an \emph{HSS tree} with row and column indices \(\HSSrowindset\) and \(\HSScolindset\) if each node \(\HSSnode\) of \(\HSStree\) is associated with two index sets \(\HSSrowindset_{\HSSnode}\) and \(\HSScolindset_{\HSSnode}\), satisfying the following conditions:
    \begin{enumerate}[(1)]
        \item \(\HSSrowindset_{\HSSroot} = \HSSrowindset\), \(\HSScolindset_{\HSSroot} = \HSScolindset\) for the root node \(\HSSroot\).
        \item For a nonleaf node \(\HSSnode\), we have \(\HSSrowindset_{\HSSnode} = \sqcup_{\HSSch \in \ch(\HSSnode)} \HSSrowindset_{\HSSch}\) and \(\HSScolindset_{\HSSnode} = \sqcup_{\HSSch \in \ch(\HSSnode)} \HSScolindset_{\HSSch}\), where the notation \(\sqcup\) means the disjoint union.
    \end{enumerate}
    For a node \(\HSSnode\), we use \(\ch(\HSSnode)\) and \(\sib(\HSSnode)\) to denote the set of children of \(\HSSnode\) and the set of siblings of \(\HSSnode\), respectively.
    For an HSS tree \(\HSStree\), we define the level of each node as follows:
    For the root node \(\HSSroot\), \(\level(\HSSroot) = 0\).
    If \(\HSSnode\) is a nonleaf node and \(\HSSch \in \ch(\HSSnode)\), then \(\level(\HSSch) = \level(\HSSnode) + 1\).
    The maximum level of the tree is denoted by \(\HSSmaxlevel\).
\end{definition}

\begin{definition}[HSS matrix,~\cite{Xi_Xia_Cauley_Balakrishnan_2014}] \label{def:hss_matrix}
  Let \(\HSStree\) be an HSS tree of row and column indices \(\HSSrowindset\) and \(\HSScolindset\).  
  A matrix \(\HSSH \in \complex^{|\HSSrowindset| \times |\HSScolindset|}\) is called an \emph{HSS matrix} about \(\HSStree\) if there are matrices \(\HSSD_{\HSSnode}\), \(\HSSU_{\HSSnode}\), \(\HSSV_{\HSSnode}\), \(\HSSR_{\HSSnode}\), \(\HSSW_{\HSSnode}\) and \(\HSSB_{\HSSrownode, \HSScolnode}\)~(called HSS generators) associated with each node \(\HSSnode\), which satisfy the following conditions:
  \begin{enumerate}[(1)]
    \item
      For a leaf node \(\HSSnode\), \(\HSSD_{\HSSnode} = \HSSH({\HSSrowindset_{\HSSnode}, \HSScolindset_{\HSSnode}})\) is a dense matrix.
    \item
      For a nonleaf node \(\HSSnode\) with children \(\{\HSSch_{\HSSchind}\}_{\HSSchind = 1}^{\HSSchnum}\),
      \(\HSSD_{\HSSnode} = \HSSH({\HSSrowindset_{\HSSnode}, \HSScolindset_{\HSSnode}})\) is a \(\HSSchnum \times \HSSchnum\) block matrix, with the \((\HSSchrowind, \HSSchcolind)\)-th block \(\HSSD_{\HSSnode; \HSSchrowind, \HSSchcolind}\) being \(\HSSD_{\HSSch_{\HSSchind}}\) for \(\HSSchrowind = \HSSchcolind\) and \(\HSSU_{\HSSch_{\HSSchrowind}} \HSSB_{\HSSch_{\HSSchrowind}, \HSSch_{\HSSchcolind}} \HSSV_{\HSSch_{\HSSchcolind}}^{\herm}\) for \(\HSSchrowind \neq \HSSchcolind\).
      The matrices \(\HSSU_{\HSSnode}\) and \(\HSSV_{\HSSnode}\) are \(\HSSchnum \times 1\) block matrices, with the \(\HSSchind\)-th block being \(\HSSU_{\HSSch_{\HSSchind}} \HSSR_{\HSSch_{\HSSchind}}\) and \(\HSSV_{\HSSch_{\HSSchind}} \HSSW_{\HSSch_{\HSSchind}}\), respectively.
  \end{enumerate}
  We call \(\HSSU_{\HSSnode}\) and \(\HSSV_{\HSSnode}^{\herm}\) the \emph{basis matrices}, \(\HSSR_{\HSSnode}\) and \(\HSSW_{\HSSnode}\) the \emph{transfer matrices}, and \(\HSSB_{\HSSrownode, \HSScolnode}\) the \emph{interaction matrices}.
\end{definition}

One important property of HSS matrices, called the \emph{shared basis property}, is that the basis matrices \(\HSSU_{\HSSnode}\) and \(\HSSV_{\HSSnode}^{\herm}\) are bases for the column and row spaces of the corresponding off-diagonal blocks~(called the \emph{HSS blocks}) \(\HSSH(\HSSrowindset_{\HSSnode}, \HSScolindset^{\compl}_{\HSSnode})\) and \(\HSSH(\HSSrowindset^{\compl}_{\HSSnode}, \HSScolindset_{\HSSnode})\) of \(\HSSH\), respectively.
Here \(\HSSrowindset^{\compl}_{\HSSnode} = \HSSrowindset \setminus \HSSrowindset_{\HSSnode}\) and \(\HSScolindset^{\compl}_{\HSSnode} = \HSScolindset \setminus \HSScolindset_{\HSSnode}\).
We use \(\HSSrank_{\HSSnode}\) to denote the maximum rank of the HSS block \(\HSSH(\HSSrowindset_{\HSSnode}, \HSScolindset^{\compl}_{\HSSnode})\) and \(\HSSH(\HSSrowindset^{\compl}_{\HSSnode}, \HSScolindset_{\HSSnode})\).
The \emph{HSS rank} of \(\HSSH\) is defined as the maximum of \(\HSSrank_{\HSSnode}\) over all nodes \(\HSSnode\) in \(\HSStree\).

Throughout this paper, we assume that the HSS matrix is associated with a given HSS tree \(\HSStree\) with maximum level \(\HSSmaxlevel\).
For simplicity, we also assume that \(\HSStree\) is a full tree, i.e., all the leaf nodes are on the same level \(\HSSmaxlevel\), and that the HSS row and column blocks of a node \(\HSSnode\) have a same rank \(\HSSrank_{\HSSnode}\).

\subsection{Recompression of HSS Matrices} \label{subsec:hss_recompression}

An HSS representation may use more basis vectors than needed to approximate its HSS blocks to a prescribed tolerance, increasing storage and subsequent factorization costs.
In this section, we review the recompression algorithm in~\cite[Section~5]{Xia_2012}, adapted to accommodate the quadtree structure.
The algorithm reduces the sizes of the HSS generators through a bottom-up orthogonalization followed by a top-down compression.
Let \(0 < \ranktol < 1\) be the tolerance for the recompression.

The first stage is to orthogonalize the basis matrices, which is done by a bottom-up traversal of the HSS tree.
For every leaf node \(\HSSnode\), compute the thin QR factorizations
\begin{equation*}
  \HSSU_{\HSSnode} = \HSStU_{\HSSnode} \HSSF_{\HSSnode}, \quad
  \HSSV_{\HSSnode} = \HSStV_{\HSSnode} \HSSG_{\HSSnode},
\end{equation*}
where \(\HSStU_{\HSSnode}\) and \(\HSStV_{\HSSnode}\) are orthonormal.
For a nonleaf node \(\HSSnode\) with children \(\{\HSSch_{\HSSchind}\}_{\HSSchind = 1}^{\HSSchnum}\), we first update the interaction matrices by
\begin{equation*}
\HSStB_{\HSSch_{\HSSchrowind}, \HSSch_{\HSSchcolind}}
= \HSSF_{\HSSch_{\HSSchrowind}} \HSSB_{\HSSch_{\HSSchrowind}, \HSSch_{\HSSchcolind}} \HSSG_{\HSSch_{\HSSchcolind}}^{\herm}
\end{equation*}
for every pair of distinct children of \(\HSSnode\).
If \(\HSSnode\) is also nonroot, the factors \(\HSSF_{\HSSch_{\HSSchind}}\) and \(\HSSG_{\HSSch_{\HSSchind}}\) are available after processing its children.
By the nested basis property, it suffices to compute the thin QR factorizations
\begin{equation*}
    \begin{bmatrix}
      \HSSF_{\HSSch_{1}} \HSSR_{\HSSch_{1}} \\
      \vdots \\
      \HSSF_{\HSSch_{\HSSchnum}} \HSSR_{\HSSch_{\HSSchnum}}
    \end{bmatrix}
    =
    \begin{bmatrix}
      \HSStR_{\HSSch_{1}} \\
      \vdots \\
      \HSStR_{\HSSch_{\HSSchnum}}
    \end{bmatrix}
    \HSSF_{\HSSnode}, \quad
    \begin{bmatrix}
      \HSSG_{\HSSch_{1}} \HSSW_{\HSSch_{1}} \\
      \vdots \\
      \HSSG_{\HSSch_{\HSSchnum}} \HSSW_{\HSSch_{\HSSchnum}}
    \end{bmatrix}
    =
    \begin{bmatrix}
      \HSStW_{\HSSch_{1}} \\
      \vdots \\
      \HSStW_{\HSSch_{\HSSchnum}}
    \end{bmatrix}
    \HSSG_{\HSSnode}.
\end{equation*}
This gives the updated transfer matrices \(\HSStR_{\HSSch_{\HSSchind}}\) and \(\HSStW_{\HSSch_{\HSSchind}}\) for each child, as well as the factors \(\HSSF_{\HSSnode}\) and \(\HSSG_{\HSSnode}\) for the current node.
This stage preserves the represented matrix in exact arithmetic and does not perform tolerance-based truncation.

The second stage is a top-down traversal of the HSS tree to compress the basis matrices and the interaction matrices.
For convenience, we still use \(\HSSU_{\HSSnode}\), \(\HSSV_{\HSSnode}\), \(\HSSR_{\HSSnode}\), \(\HSSW_{\HSSnode}\) and \(\HSSB_{\HSSch, \HSSchsib}\) to denote the generators after orthogonalization.
At each nonleaf node \(\HSSnode\), let \(\HSSS_{\HSSnode}^{U}\) and \(\HSSS_{\HSSnode}^{V}\) be the diagonal matrices of retained singular values passed from the preceding level; both are empty at the root.
For each child \(\HSSch \in \ch(\HSSnode)\), form the two matrices
\begin{equation*}
\HSSC_{\HSSch}^{U}
=
\begin{bmatrix}
\bigl(\HSSB_{\HSSch, \HSSchsib}\bigr)_{\HSSchsib \in \sib(\HSSch)}
& \HSStR_{\HSSch} \HSSS_{\HSSnode}^{U}
\end{bmatrix},
\quad
\HSSC_{\HSSch}^{V}
=
\begin{bmatrix}
  \bigl(\HSSB_{\HSSchsib, \HSSch}^{\herm}\bigr)_{\HSSchsib \in \sib(\HSSch)}
  & \HSStW_{\HSSch} \HSSS_{\HSSnode}^{V}
  \end{bmatrix}.
\end{equation*}
Here the blocks indexed by \(\HSSchsib\) are concatenated horizontally, and the last block is omitted when \(\HSSnode\) is the root.
Let the truncated SVDs be
\begin{equation*}
\HSSC_{\HSSch}^{U}
\approx
\HSSP_{\HSSch} \HSSS_{\HSSch}^{U} \HSSX_{\HSSch}^{\herm}, \quad
\HSSC_{\HSSch}^{V}
\approx
\HSSQ_{\HSSch} \HSSS_{\HSSch}^{V} \HSSY_{\HSSch}^{\herm}.
\end{equation*}
The interaction matrices are updated by
\begin{equation*}
\HSStB_{\HSSch_{\HSSchrowind}, \HSSch_{\HSSchcolind}}
= \HSSP_{\HSSch_{\HSSchrowind}}^{\herm} \HSSB_{\HSSch_{\HSSchrowind}, \HSSch_{\HSSchcolind}} \HSSQ_{\HSSch_{\HSSchcolind}}.
\end{equation*}
If \(\HSSnode\) is also nonroot, the transfer matrices are updated by
\begin{equation*}
\HSShR_{\HSSch} = \HSSP_{\HSSch}^{\herm} \HSStR_{\HSSch},
\quad
\HSShW_{\HSSch} = \HSSQ_{\HSSch}^{\herm} \HSStW_{\HSSch},
\end{equation*}
If \(\HSSch\) is a leaf, the basis matrices are updated by
\begin{equation*}
\HSShU_{\HSSch} = \HSSU_{\HSSch} \HSSP_{\HSSch},
\quad
\HSShV_{\HSSch} = \HSSV_{\HSSch} \HSSQ_{\HSSch}.
\end{equation*}
Otherwise, the transfer matrices of its children are updated by
\begin{equation*}
\HSStR_{\HSSchch} = \HSSR_{\HSSchch} \HSSP_{\HSSch},
\quad
\HSStW_{\HSSchch} = \HSSW_{\HSSchch} \HSSQ_{\HSSch},
\quad
\HSSchch \in \ch(\HSSch).
\end{equation*}
The same process is then applied recursively to each nonleaf child using \(\Mat{S}_{\HSSch}^{U}\) and \(\Mat{S}_{\HSSch}^{V}\).

\subsection{The URV Factorization for HSS Matrices} \label{subsec:urv_factorization}

Consider the least-squares problem \(\min_{\HSSc} \|\HSSH \HSSc - \HSSf\|_{2}\), where \(\HSSH\) is an HSS matrix and \(\HSSf\) is a given vector.
In this section, we review the URV factorization for HSS matrices~\cite{Wilber_Epperly_Barnett_2025, Xi_Xia_Cauley_Balakrishnan_2014}, which serves as a direct solver for the least-squares problem.

The URV factorization for HSS matrices is performed in a bottom-up manner on the HSS tree.
Suppose \(\HSSnode\) is a leaf node and \(\HSSD_{\HSSnode} \in \complex^{\HSSnoderowsize_{\HSSnode} \times \HSSnodecolsize_{\HSSnode}}\).
If \(\HSSnoderowsize_{\HSSnode} \geq \HSSnodecolsize_{\HSSnode} + \HSSrank_{\HSSnode}\), a \emph{size reduction} step is performed by a QR factorization:
\begin{equation*}
  \begin{bmatrix}
    \HSSU_{\HSSnode} & \HSSD_{\HSSnode}
  \end{bmatrix}
  = \HSSOm_{\HSSnode}
  \begin{bmatrix}
    \Mat{R}_{1, 1} & \Mat{R}_{1, 2} \\
    \Mat{0} & \Mat{R}_{2, 2} \\
    \hline
    \Mat{0} & \Mat{0}
  \end{bmatrix}
  = \HSSOm_{\HSSnode}
  \begin{bmatrix}
    \HSScU_{\HSSnode} & \HSScD_{\HSSnode} \\
    \Mat{0} & \Mat{0}
  \end{bmatrix},
\end{equation*}
where \(\HSSOm_{\HSSnode} \in \complex^{\HSSnoderowsize_{\HSSnode} \times \HSSnoderowsize_{\HSSnode}}\) is unitary.
The reduced row size, i.e., the number of rows of \(\HSScU_{\HSSnode}\) and \(\HSScD_{\HSSnode}\), is \(\HSSnodereducedrowsize_{\HSSnode} = \HSSnodecolsize_{\HSSnode} + \HSSrank_{\HSSnode}\).
If no size reduction is applied, we set \(\HSSOm_{\HSSnode}\) to be empty and let \(\HSScU_{\HSSnode} = \HSSU_{\HSSnode}\), \(\HSScD_{\HSSnode} = \HSSD_{\HSSnode}\) and \(\HSSnodereducedrowsize_{\HSSnode} = \HSSnoderowsize_{\HSSnode}\).

The next is a \emph{basis elimination} step.
For every leaf node \(\HSSnode\), zeros are introduced into \(\HSSV_{\HSSnode}\) by 
\begin{equation} \label{eq:urv_basis_elimination}
    \HSSV_{\HSSnode} = \HSSQ_{\HSSnode}
    \begin{bmatrix}
        \Mat{0} \\
        \HSShV_{\HSSnode; 2}
    \end{bmatrix},
\end{equation}
where \(\HSSQ_{\HSSnode} \in \complex^{\HSSnodecolsize_{\HSSnode} \times \HSSnodecolsize_{\HSSnode}}\) is unitary and \(\HSShV_{\HSSnode; 2} \in \complex^{\HSSrank_{\HSSnode} \times \HSSrank_{\HSSnode}}\).
This can be done by a reverse QR decomposition.
The diagonal block is then modified by 
\begin{equation*}
\HSStD_{\HSSnode} = \HSScD_{\HSSnode} \HSSQ_{\HSSnode}.
\end{equation*}

Then a \emph{partial decomposition} on the matrix \(\HSStD_{\HSSnode}\) is performed to decouple the unknowns corresponding to the zero and dense parts of~\eqref{eq:urv_basis_elimination}.
First, a QR decomposition is computed on the first block column of \(\HSStD_{\HSSnode}\) as
\begin{equation*}
  \begin{bmatrix}
    \HSStD_{\HSSnode; 1, 1} \\
    \HSStD_{\HSSnode; 2, 1}
  \end{bmatrix}
  = \HSSP_{\HSSnode}
  \begin{bmatrix}
    \HSShD_{\HSSnode; 1, 1} \\
    \Mat{0}
  \end{bmatrix}.
\end{equation*}
Then it follows that
\begin{equation*}
  \HSStD_{\HSSnode}
  = \HSSP_{\HSSnode}
  \begin{bmatrix}
    \HSShD_{\HSSnode; 1, 1} & \HSShD_{\HSSnode; 1, 2} \\
    \Mat{0} & \HSShD_{\HSSnode; 2, 2}
  \end{bmatrix},
\end{equation*}
where
\begin{equation*}
  \begin{bmatrix}
    \HSShD_{\HSSnode; 1, 2} \\
    \HSShD_{\HSSnode; 2, 2}
  \end{bmatrix} = \HSSP_{\HSSnode}^{\herm}
  \begin{bmatrix}
    \HSStD_{\HSSnode; 1, 2} \\
    \HSStD_{\HSSnode; 2, 2}
  \end{bmatrix}.
\end{equation*}
The basis matrix is then modified by
\begin{equation*}
    \HSSP_{\HSSnode}^{\herm} \HSScU_{\HSSnode}
  = \begin{bmatrix}
    \HSShU_{\HSSnode; 1} \\
    \HSShU_{\HSSnode; 2} \\
    \end{bmatrix}.
\end{equation*}

After that, the variables corresponding to the first block column of \(\HSStD_{\HSSnode}\) can be decoupled from the rest variables, and the least-squares problem can be solved by a recursive process, which we describe in the following.
For a nonleaf node \(\HSSnode\) with children \(\{\HSSch_{\HSSchind}\}\), appropriate blocks of the children are merged to form the new generators.
Let \(\HSSD_{\HSSnode}^{+}\), \(\HSSU_{\HSSnode}^{+}\) and \(\HSSV_{\HSSnode}^{+}\) be the block matrices given by
\begin{equation*}
  \begin{gathered}
    \HSSD_{\HSSnode; \HSSchrowind, \HSSchcolind}^{+} =
    \begin{cases}
      \HSShD_{\HSSch_{\HSSchrowind}; 2, 2}, & \HSSchrowind = \HSSchcolind, \\
      \HSShU_{\HSSch_{\HSSchrowind}; 2} \HSSB_{\HSSch_{\HSSchrowind}, \HSSch_{\HSSchcolind}} \HSShV_{\HSSch_{\HSSchcolind}; 2}^{\herm}, & \HSSchrowind \neq \HSSchcolind
    \end{cases} \\
    \HSSU_{\HSSnode; \HSSchind}^{+} = \HSShU_{\HSSch_{\HSSchind}; 2} \HSSR_{\HSSch_{\HSSchind}}, \quad
    \HSSV_{\HSSnode; \HSSchind}^{+} = \HSShV_{\HSSch_{\HSSchind}; 2} \HSSW_{\HSSch_{\HSSchind}}.
  \end{gathered}
\end{equation*}
Then the children of \(\HSSnode\) can be ``removed'' and the node \(\HSSnode\) can be treated as a leaf node.
The generators \(\{\HSSD_{\HSSnode}^{+}\}\), \(\{\HSSU_{\HSSnode}^{+}\}\), \(\{\HSSV_{\HSSnode}^{+}\}\) \(\{\HSSR_{\HSSnode}\}\), \(\{\HSSW_{\HSSnode}\}\) and \(\{\HSSB_{\HSSrownode, \HSScolnode}\}\) define a new HSS matrix of maximum level \(\HSSmaxlevel - 1\).
This HSS can be further factorized by the same discussion above~(with \(\HSSD_{\HSSnode}\), \(\HSSU_{\HSSnode}\) and \(\HSSV_{\HSSnode}\) replaced by \(\HSSD_{\HSSnode}^{+}\), \(\HSSU_{\HSSnode}^{+}\) and \(\HSSV_{\HSSnode}^{+}\) ) until the root node is reached.
At the root node \(\HSSroot\), the QR factorization is performed, i.e., \(\HSSD_{\HSSroot}^{+} = \HSSP_{\HSSroot} \HSShD_{\HSSroot}^{+}\).

Once the URV factors are computed, the solution is obtained by a bottom-up pass and a top-down pass.
The bottom-up pass is to apply the unitary transformations \(\HSSOm_{\HSSnode}\) and \(\HSSP_{\HSSnode}\) to the right-hand side \(\HSSf\).
Starting from every leaf node \(\HSSnode\), if size reduction is performed, the vector \(\HSSf_{\HSSnode}\) is updated by
\begin{equation*}
  \HSSOm_{\HSSnode}^{\herm} \HSSf_{\HSSnode} =
  \begin{bmatrix}
    \HSScf_{\HSSnode} \\
    \times
  \end{bmatrix}, \quad 
  \HSShf_{\HSSnode}
  = \HSSP_{\HSSnode}^{\herm} \HSScf_{\HSSnode} =
  \begin{bmatrix}
    \HSShf_{\HSSnode; 1} \\
    \HSShf_{\HSSnode; 2}
  \end{bmatrix},
\end{equation*}
where the \(\times\) means the corresponding entries are not used in the following steps and can be ignored.
If no size reduction is performed, the first equation is omitted and \(\HSScf_{\HSSnode}\) is replaced by \(\HSSf_{\HSSnode}\) in the second equation.
Then for a nonleaf node \(\HSSnode\) with children \(\{\HSSch_{\HSSchind}\}\), \(\HSSf_{\HSSnode}^{+}\) is obtained by merging its children blocks, i.e., \(\HSSf_{\HSSnode; \HSSchind}^{+} = \HSShf_{\HSSch_{\HSSchind}; 2}\).
This process is repeated~(with \(\HSSf_{\HSSnode}\) replaced by \(\HSSf_{\HSSnode}^{+}\) in the above equations) until the root node is reached.

The top-down pass recovers the solution by back substitution.
At the root node \(\HSSroot\), solve the reduced least-squares problem
\begin{equation*}
  \min_{\HSSc_{\HSSroot}^{+}}
  \|\HSShD_{\HSSroot}^{+} \HSSc_{\HSSroot}^{+} - \HSShf_{\HSSroot}\|_{2}.
\end{equation*}
For each nonleaf node \(\HSSnode\), partition \(\HSSc_{\HSSnode}^{+}\) into the vectors \(\HSShc_{\HSSch; 2}\) associated with its children \(\HSSch \in \ch(\HSSnode)\).
The compressed contribution from columns outside each child is
\begin{equation*}
  \HSShb_{\HSSch}
  =
  \HSSR_{\HSSch} \HSShb_{\HSSnode}
  +
  \sum_{\HSSchsib \in \sib(\HSSch)}
  \HSSB_{\HSSch, \HSSchsib}
  \HSShV_{\HSSchsib; 2}^{\herm}
  \HSShc_{\HSSchsib; 2},
\end{equation*}
where the first term is omitted when \(\HSSnode = \HSSroot\).
The coefficients at each child are recovered by
\begin{equation*}
  \begin{aligned}
    \HSShc_{\HSSch; 1}
    & =
    \HSShD_{\HSSch; 1, 1}^{-1}
    \bigl(
      \HSShf_{\HSSch; 1}
      -
      \HSShD_{\HSSch; 1, 2} \HSShc_{\HSSch; 2}
      -
      \HSShU_{\HSSch; 1} \HSShb_{\HSSch}
    \bigr), \\
    \HSSc_{\HSSch}^{+}
    & =
    \HSSQ_{\HSSch}
    \begin{bmatrix}
      \HSShc_{\HSSch; 1} \\
      \HSShc_{\HSSch; 2}
    \end{bmatrix}.
  \end{aligned}
\end{equation*}
Repeat this process down the tree.
At each leaf node \(\HSSnode\), the corresponding solution block is \(\HSSc_{\HSSnode} = \HSSc_{\HSSnode}^{+}\).

\subsection{A Review of the Direct Inverse Method in 1D Case} \label{subsec:inverse_typeII_1d_nudft}

Consider the 1D type-II NUDFT matrix given by \(\nudftmat(\rowind, \colind) = \e^{-2 \pi \imag \colind \pt_{\rowind}}\) where \(\pt_{\rowind} \in [0, 1)\) for \(\rowind = 0, 1, \dotsc, \numrow - 1\) and \(\colind = 0, 1, \dotsc, \numcol - 1\).
If we define \(\rowptS_{\rowind} = \e^{-2 \pi \imag \pt_{\rowind}} \in \unitcircle\), then \(\nudftmat(\rowind, \colind) = \rowptS_{\rowind}^{\colind}\) and it can be directly verified that \(\nudftmat\) satisfies the following Sylvester equation
\begin{equation} \label{eq:sylvester_equation_1d_nudft_A}
  \rowptSmat \nudftmat - \nudftmat \circulantC = \nudftumat \Mat{e}_{\numcol - 1}^{\trans},
\end{equation}
where \(\rowptSmat = \diag(\rowptS_{0}, \rowptS_{1}, \dotsc, \rowptS_{\numrow - 1})\) is a diagonal matrix, and
\begin{equation*}
  \circulantC =
  \begin{bmatrix}
     &  &  &  & 1 \\
    1 &  &  &  &  \\
     & 1 &  &  &  \\
     &  & \ddots &  &  \\
     &  &  & 1 &
  \end{bmatrix}
\end{equation*}
is a circulant shift matrix and \(\nudftusca_{\rowind} = \rowptS_{\rowind}^{\numcol} - 1\) for \(\rowind = 0, 1, \dotsc, \numrow - 1\).
The matrix \(\circulantC\) can be diagonalized by the DFT matrix, i.e., \(\circulantC = \dftmatinv \circulantD \dftmat\) where \(\dftmat\) is the DFT matrix given by \(\dftmat(\colind, \colindpair) = \e^{-2 \pi \imag \colind \colindpair / \numcol}\) and \(\circulantD = \diag(\colptS_{0}, \colptS_{1}, \dotsc, \colptS_{\numcol - 1})\) is a diagonal matrix.
Here \(\colptS_{\colind} = \unitroot_{\numcol}^{\colind} \in \unitcircle\) and \(\unitroot_{\numcol} = \e^{-2 \pi \imag / \numcol}\) is the \(\numcol\)th root of unity.
Substituting the diagonalization of \(\circulantC\) into~\eqref{eq:sylvester_equation_1d_nudft_A} and multiplying both sides by \(\dftmatinv\) on the right, it follows that
\begin{equation} \label{eq:sylvester_equation_1d_nudft_H}
  \rowptSmat \tnudftmat - \tnudftmat \circulantD = \nudftumat \Mat{v}^{\trans},
\end{equation}
where \(\tnudftmat = \nudftmat \dftmatinv\) and \(\nudftvmat = \dftmatinv \Mat{e}_{\numcol - 1}\).
The entries of \(\nudftvmat\) are given by \(\nudftvsca_{\colind} = \colptS_{\colind} / \numcol\) for \(\colind = 0, 1, \dotsc, \numcol - 1\).
Therefore, \(\tnudftmat\) can be expressed as a Cauchy-like matrix with the entries given by
\begin{equation} \label{eq:1d_nudft_tildeA_entries}
  \tnudftmat(\rowind, \colind)
  = \frac{\nudftusca_{\rowind} \nudftvsca_{\colind}}{\rowptS_{\rowind} - \colptS_{\colind}}
  = \frac{1}{\numcol} \frac{(\rowptS_{\rowind} / \colptS_{\colind})^{\numcol} - 1}{(\rowptS_{\rowind} / \colptS_{\colind}) - 1}
  \eqcolon \nudftkernel(\rowptS_{\rowind}, \colptS_{\colind}).
\end{equation}
When \(\rowptS_{\rowind} = \colptS_{\colind}\), the quotients are interpreted by continuity, giving \(\nudftkernel(\rowptS_{\rowind}, \colptS_{\colind}) = 1\).

In~\cite{Wilber_Epperly_Barnett_2025}, the authors utilize the \emph{displacement structure}~\cite{Beckermann_Townsend_2017} of \(\tnudftmat\) in~\eqref{eq:sylvester_equation_1d_nudft_H} to derive the \emph{low-rank property} of \(\tnudftmat\).
They showed that \(\tnudftmat\) can be approximated by an HSS matrix \(\nudfthss\) with numerical HSS rank \(\HSSrank = \bigO\bigl(\mylog{(1 / \ranktol)} \mylog{\numcol}\bigr)\), where \(\ranktol\) is the accuracy parameter for the approximation.
Specifically, we have the following Theorem.

\begin{theorem}[Numerical rank of HSS blocks of \(\tnudftmat\), Theorem~3.2 in \cite{Wilber_Epperly_Barnett_2025}] \label{thm:1d_nudft_hss_block_rank}
  Let \(\tnudftmat\) be the matrix defined by~\eqref{eq:1d_nudft_tildeA_entries} and \(0 < \ranktol < 1\) and \(\rowindset = \{0, 1, \dotsc, \numrow - 1\}\) and \(\colindset = \{0, 1, \dotsc, \numcol - 1\}\) be the row and column index sets of \(\tnudftmat\) respectively.
  Suppose that \(\nudftintervalx \subset [0, 1)\) is an interval and \(\rowsubindset = \{\rowind : \pt_{\rowind} \in \nudftintervalx\} \subset \rowindset\) and \(\colsubindset = \{\colind: \colind / \numcol \in \nudftintervalx\} \subset \colindset\) are the row and column index sets corresponding to this interval.
  Then for the \(2\)-norm or the Frobenius norm, the HSS row or column block of \(\tnudftmat\) with respect to \(\rowsubindset\) or \(\colsubindset\) has a numerical rank at most \(\HSSrank = \bigO\bigl(\mylog{(1 / \ranktol)} \mylog{\numcoldir_{\col}}\bigr)\) where \(\numcoldir_{\col} = |\colsubindset|\).
\end{theorem}

There are two main differences between Theorem~\ref{thm:1d_nudft_hss_block_rank} and the original result in~\cite{Wilber_Epperly_Barnett_2025}.
First, the original result in~\cite{Wilber_Epperly_Barnett_2025} is for the \(2\)-norm while we add the Frobenius norm in Theorem~\ref{thm:1d_nudft_hss_block_rank}.
This is because the singular values of the corresponding submatrix decay exponentially.
The second is that the rank bound here is more refined by carefully tracking the inequality in the proof of~\cite[Theorem~3.2]{Wilber_Epperly_Barnett_2025}.

From the Sylvester equation~\eqref{eq:sylvester_equation_1d_nudft_H}, the HSS approximation \(\nudfthss\) can be constructed by combining the \emph{factored alternating direction implicit}~(fADI) method~\cite{Benner_Li_Truhar_2009} and interpolative decomposition~\cite{Cheng_Gimbutas_Martinsson_Rokhlin_2005}, at a complexity of \(\bigO\bigl((\numrow + \numcol) \mylog[2]{\numcol}\bigr)\).
Once the HSS approximation \(\nudfthss\) is constructed, the URV factorization described in Section~\ref{subsec:urv_factorization} is performed; together, these form the offline steps of the direct inverse method.
For a given right-hand side \(\targetvvec \in \complex^{\numrow}\), the online step computes an approximate least-squares solution
\(\coeffvec = \dftmatinv \nudfthss^{\dagger} \targetvvec\)
by one application of the URV solver followed by an iFFT.
Therefore, the online complexity is \(\bigO\bigl((\numrow + \numcol) \mylog{\numcol}\bigr)\).

\section{HSS Approximation via Face-Splitting Products} \label{sec:low_rank_property_nudft_matrix}

Section~\ref{subsec:kernel_matrix_perspective} derives a separable kernel representation of the transformed 2D NUDFT matrix and motivates its off-diagonal low-rank structure.
Section~\ref{subsec:face_splitting_product_hss_tree_matrix} develops the face-splitting construction of HSS trees and matrices, with bounds on the ranks at each level.
Section~\ref{subsec:error_estimate_hss_approximation} bounds the matrix approximation error in terms of the errors in the two 1D factors.

\subsection{Kernel Matrix Perspective} \label{subsec:kernel_matrix_perspective}

Consider the 2D type-II NUDFT matrix \(\nudftmat \in \complex^{\numrow \times \numcol}\) given by~\eqref{eq:typeII_2d_nudft_matrix}.
Let \(\nudftmatx \in \complex^{\numrow \times \numcolx}\) and \(\nudftmaty \in \complex^{\numrow \times \numcoly}\) be the type-II NUDFT in \(\dirx\)- and \(\diry\)- direction respectively, i.e., \(\nudftmatx(\rowind, \colindx) = \e^{-2 \pi \imag \colindx \ptx_{\rowind}}\) and \(\nudftmaty(\rowind, \colindy) = \e^{-2 \pi \imag \colindy \pty_{\rowind}}\).
Then \(\nudftmat\) can be written as the face-splitting product \(\nudftmat = \nudftmatx \faceprod \nudftmaty\).

Let \(\dftmatx \in \complex^{\numcolx \times \numcolx}\) and \(\dftmaty \in \complex^{\numcoly \times \numcoly}\) be the DFT matrices in \(\dirx\)- and \(\diry\)- direction respectively and \(\dftmat = \dftmatx \otimes \dftmaty \in \complex^{\numcol \times \numcol}\) be the DFT matrix in 2D.
Following the discussions in Section~\ref{subsec:inverse_typeII_1d_nudft}, if we define \(\tnudftmatx = \nudftmatx \dftmatxinv\),  \(\tnudftmaty = \nudftmaty \dftmatyinv\) and \(\tnudftmat = \nudftmat \dftmatinv\), then it can also be verified that \(\tnudftmat = \tnudftmatx \faceprod \tnudftmaty\).
Using~\eqref{eq:1d_nudft_tildeA_entries}, the entries of \(\tnudftmat\) can be written as
\begin{equation} \label{eq:2d_nudft_tildeA_entries}
  \begin{aligned}
    \tnudftmat\bigl(\rowind, (\colindx, \colindy)\bigr)
    & = 
    \frac{1}{\numcolx} \frac{\bigl(\rowptSx_{\rowind} / \colptSx_{\colindx}\bigr)^{\numcolx} - 1}{\bigl(\rowptSx_{\rowind} / \colptSx_{\colindx}\bigr) - 1}
    \frac{1}{\numcoly} \frac{\bigl(\rowptSy_{\rowind} / \colptSy_{\colindy}\bigr)^{\numcoly} - 1}{\bigl(\rowptSy_{\rowind} / \colptSy_{\colindy}\bigr) - 1} \\
    & = \nudftkernelx\bigl(\rowptSx_{\rowind}, \colptSx_{\colindx}\bigr) \nudftkernely\bigl(\rowptSy_{\rowind}, \colptSy_{\colindy}\bigr) \\
    & \eqcolon \nudftkernel\bigl(\nudftrowpt_{\rowind}, \nudftcolpt_{(\colindx, \colindy)}\bigr),
  \end{aligned}
\end{equation}
where \(\rowptSx_{\rowind} = \e^{-2 \pi \imag \ptx_{\rowind}}\), \(\rowptSy_{\rowind} = \e^{-2 \pi \imag \pty_{\rowind}}\), \(\colptSx_{\colindx} = \e^{-2 \pi \imag \colindx / \numcolx}\) and \(\colptSy_{\colindy} = \e^{-2 \pi \imag \colindy / \numcoly}\) belong to the unit circle \(\unitcircle\), \(\nudftkernelx\) and \(\nudftkernely\) are corresponding kernel functions in \(\dirx\)- and \(\diry\)- direction respectively, and \(\nudftrowpt_{\rowind} = (\rowptSx_{\rowind}, \rowptSy_{\rowind})\) and \(\nudftcolpt_{(\colindx, \colindy)} = (\colptSx_{\colindx}, \colptSy_{\colindy})\) are points on \(\unitcircle^{2}\).
Consequently, \(\tnudftmat\) can be viewed as a kernel matrix with the kernel function
\begin{equation} \label{eq:2d_nudft_kernel_function}
  \nudftkernel(\GENzmat, \GENwmat) = \nudftkernelx(\GENzsca_{1}, \GENwsca_{1}) \nudftkernely(\GENzsca_{2}, \GENwsca_{2}), \quad \GENzmat = (\GENzsca_{1}, \GENzsca_{2}) \in \unitcircle^{2}, \quad \GENwmat = (\GENwsca_{1}, \GENwsca_{2}) \in \unitcircle^{2}.
\end{equation}
We next use the separable kernel representation~\eqref{eq:2d_nudft_tildeA_entries} to motivate the HSS approximation and estimate its off-diagonal ranks.
The following lemma bounds the rank of a face-splitting product by the product of the ranks of its factors.

\begin{lemma}[Rank of the matrix face-splitting product] \label{lem:face_splitting_exact_rank}
  Let \(\GENmatA \in \complex^{\GENrowsize \times \GENcolsizeA}\) and \(\GENmatB \in \complex^{\GENrowsize \times \GENcolsizeB}\) be two matrices and \(\GENmat = \GENmatA \faceprod \GENmatB\), then \(\rank(\GENmat) \leq \rank(\GENmatA) \rank(\GENmatB)\).
\end{lemma}
\begin{proof}
  For \(\GENindAB = 1, 2\), let
  \(\GENmat_{\GENindAB}
  = \GENLRleft_{\GENindAB} \GENLRright_{\GENindAB}^{\herm}\)
  be a rank factorization, where
  \(\GENrank_{\GENindAB} = \rank(\GENmat_{\GENindAB})\),
  \(\GENLRleft_{\GENindAB} \in \complex^{\GENrowsize \times \GENrank_{\GENindAB}}\),
  and
  \(\GENLRright_{\GENindAB} \in \complex^{\GENcolsize_{\GENindAB} \times \GENrank_{\GENindAB}}\).
  By the rowwise definition of the face-splitting product,
  \begin{equation*}
    \GENmatA \faceprod \GENmatB
    =
    (\GENLRleftA \faceprod \GENLRleftB)
    (\GENLRrightA \otimes \GENLRrightB)^{\herm}.
  \end{equation*}
  Therefore, \(\rank(\GENmat)\) is at most \(\GENrankA \GENrankB\).
\end{proof}

\begin{figure}
  \centering
  \includegraphics[width=0.3\textwidth]{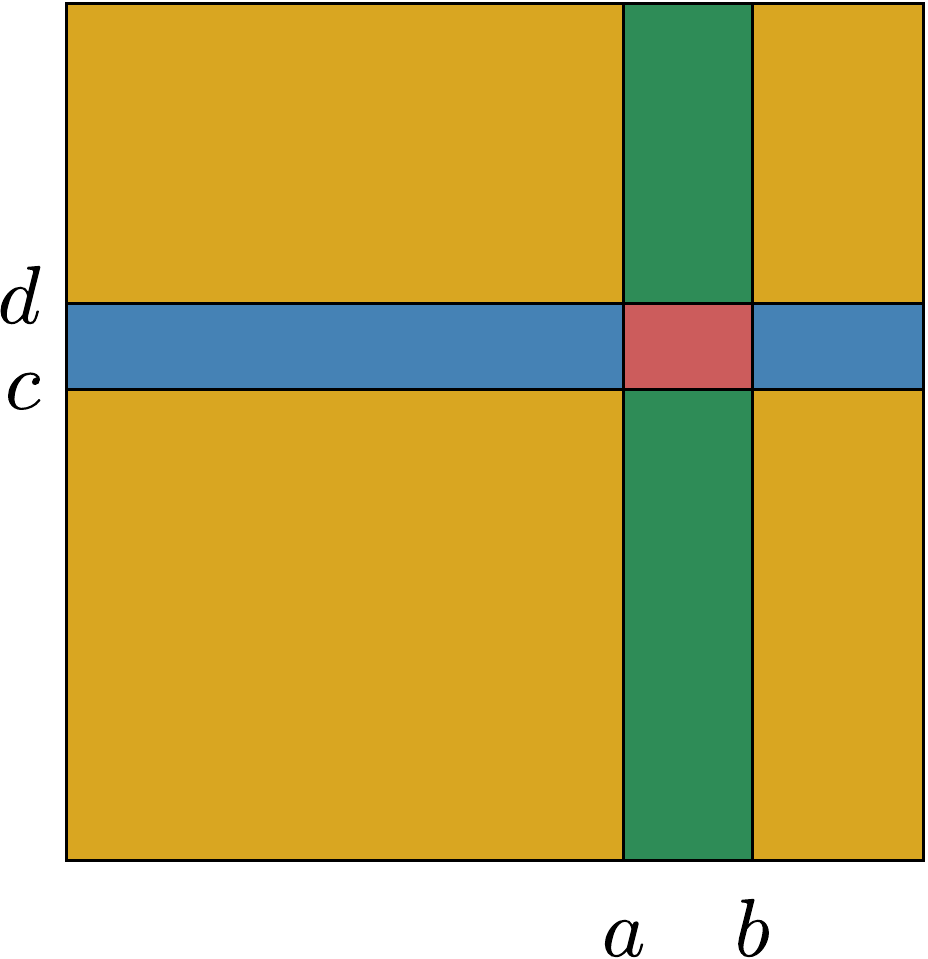}
  \caption{
    The partition of \([0, 1)^2\) according to a rectangular domain.
    The red box represents the rectangular domain \(\nudftintervalx \times \nudftintervaly\).
    The blue boxes represent the points whose \(\dirx\)-coordinate not in \(\nudftintervalx\) but \(\diry\)-coordinate in \(\nudftintervaly\).
    The green boxes represent the points whose \(\dirx\)-coordinate in \(\nudftintervalx\) but \(\diry\)-coordinate not in \(\nudftintervaly\).
    The yellow boxes represent the points whose \(\dirx\)-coordinate not in \(\nudftintervalx\) and \(\diry\)-coordinate not in \(\nudftintervaly\).
    } \label{fig:low_rank_intuition}
\end{figure}

Using the parametrization \(\GENzsca = \e^{-2 \pi \imag \tsca}\) with \(\tsca \in [0, 1)\), we identify each row and column point in \(\unitcircle^{2}\) with a point in the periodic domain \([0, 1)^{2}\).
The first class is the rectangle \(\nudftintervalx \times \nudftintervaly\), shown in red in Figure~\ref{fig:low_rank_intuition}.
The second, third, and fourth classes consist of points outside the corresponding intervals in the \(\dirx\)-coordinate only, the \(\diry\)-coordinate only, and both coordinates, respectively; these are shown in blue, green, and yellow.

Let \(\rowsubindset\) and \(\colsubindset = \colsubindsetx \times \colsubindsety\) be row and column index sets of \(\tnudftmat\) respectively, such that the corresponding row and column points belong to the first class.
Define \(\rowsubindset^{\compl} = \rowindset \setminus \rowsubindset\) and \(\colsubindset^{\compl} = \colindset \setminus \colsubindset\).
Suppose \(|\colsubindsetx| = \numcolx_{\col}\) and \(|\colsubindsety| = \numcoly_{\col}\) and \(\numcoldir_{\col} = \max\{\numcolx_{\col}, \numcoly_{\col}\}\).

Consider the HSS row block \(\tnudftmat(\rowsubindset, \colsubindset^{\compl})\).
We can partition it as
\begin{equation} \label{eq:typeII_2d_hss_row_block_partition}
  \tnudftmat(\rowsubindset, \colsubindset^{\compl})
  = 
  \begin{bmatrix}
    \tnudftmat(\rowsubindset, \colsubindset_{\myblue}) & \tnudftmat(\rowsubindset, \colsubindset_{\mygreen}) & \tnudftmat(\rowsubindset, \colsubindset_{\myyellow})
  \end{bmatrix},
\end{equation}
where \(\colsubindset_{\myblue}\), \(\colsubindset_{\mygreen}\) and \(\colsubindset_{\myyellow}\) denote the column index sets corresponding to the second, third and fourth classes respectively.
Note that these submatrices can be further factorized as
\begin{equation} \label{eq:typeII_2d_hss_row_block_submatrix_factorization}
  \begin{aligned}
    \tnudftmat\bigl(\rowsubindset, \colsubindset_{\myblue}\bigr)
    & = \tnudftmatx\bigl(\rowsubindset, \colsubindsetxcompl\bigr) \faceprod \tnudftmaty\bigl(\rowsubindset, \colsubindsety\bigr), \\
    \tnudftmat\bigl(\rowsubindset, \colsubindset_{\mygreen}\bigr)
    & = \tnudftmatx\bigl(\rowsubindset, \colsubindsetx\bigr) \faceprod \tnudftmaty\bigl(\rowsubindset, \colsubindsetycompl\bigr), \\
    \tnudftmat\bigl(\rowsubindset, \colsubindset_{\myyellow}\bigr)
    & = \tnudftmatx\bigl(\rowsubindset, \colsubindsetxcompl\bigr) \faceprod \tnudftmaty\bigl(\rowsubindset, \colsubindsetycompl\bigr).
  \end{aligned}
\end{equation}

By Theorem~\ref{thm:1d_nudft_hss_block_rank}, the factors
\(\tnudftmatx(\rowsubindset, \colsubindsetxcompl)\) and
\(\tnudftmaty(\rowsubindset, \colsubindsetycompl)\)
admit rank-\(\HSSrank\) approximations, where
\(\HSSrank = \bigO\bigl(\mylog{(1 / \ranktol)} \mylog{\numcoldir_{\col}}\bigr)\).
Substituting these approximations into~\eqref{eq:typeII_2d_hss_row_block_submatrix_factorization} and applying Lemma~\ref{lem:face_splitting_exact_rank} gives approximations of the blue, green, and yellow submatrices with ranks at most
\(\HSSrank \numcoldir_{\col}\),
\(\HSSrank \numcoldir_{\col}\), and
\(\HSSrank^{2}\), respectively.
The resulting approximation of the HSS row block therefore has rank at most
\(2 \HSSrank \numcoldir_{\col} + \HSSrank^{2}\),
which is \(\bigO(\numcoldir_{\col} \mylog{\numcoldir_{\col}})\) for fixed \(\ranktol\).
The same argument applies to the HSS column block.
For a square frequency grid, this motivates an HSS rank model of
\(\bigO(\sqrt{\numcol} \mylog{\numcol})\) at fixed accuracy.
The propagation of approximation errors through the face-splitting product is considered in Section~\ref{subsec:error_estimate_hss_approximation}.

\subsection{Face-Splitting Product of HSS Trees and HSS Matrices} \label{subsec:face_splitting_product_hss_tree_matrix}

Let \(\HSStreex\) be an HSS tree with row and column index sets \(\HSSrowindset\) and \(\HSScolindsetx\) and \(\HSStreey\) be an HSS tree with row and column index sets \(\HSSrowindset\) and \(\HSScolindsety\).
Suppose \(\HSScolindset = \HSScolindsetx \times \HSScolindsety\).
We construct a new HSS tree \(\HSStree\) with row and column index sets \(\HSSrowindset\) and \(\HSScolindset\) as follows.

\begin{definition}[Face-splitting product of HSS trees] \label{def:face_splitting_product_hss_tree}
  Let \(\HSStreex\) and \(\HSStreey\) be two HSS trees of row and column index sets \(\HSSrowindset\) and \(\HSScolindsetx\) and \(\HSSrowindset\) and \(\HSScolindsety\) respectively, then the face-splitting product of \(\HSStreex\) and \(\HSStreey\), denoted by \(\HSStree = \HSStreex \faceprod \HSStreey\), is defined as follows.
  \begin{enumerate}[(1)]
    \item The root of \(\HSStree\) is \(\HSSroot = (\HSSrootx, \HSSrooty)\), where \(\HSSrootx\) and \(\HSSrooty\) are the roots of \(\HSStreex\) and \(\HSStreey\) respectively.
    \item If \(\HSSnodex\) is a leaf and \(\HSSnodey\) is a nonleaf, then \(\HSSnode = (\HSSnodex, \HSSnodey)\) is a nonleaf node of \(\HSStree\) and its children are given by \(\ch(\HSSnode) = \{(\HSSnodex, \HSSchy): \HSSchy \in \ch(\HSSnodey)\}\).
    \item If \(\HSSnodex\) is a nonleaf and \(\HSSnodey\) is a leaf, then \(\HSSnode = (\HSSnodex, \HSSnodey)\) is a nonleaf node of \(\HSStree\) and its children are given by \(\ch(\HSSnode) = \{(\HSSchx, \HSSnodey): \HSSchx \in \ch(\HSSnodex)\}\).
    \item If \(\HSSnodex\) and \(\HSSnodey\) are nonleaf, then \(\HSSnode = (\HSSnodex, \HSSnodey)\) is a nonleaf node of \(\HSStree\) and its children are given by \(\ch(\HSSnode) = \{(\HSSchx, \HSSchy): \HSSchx \in \ch(\HSSnodex), \HSSchy \in \ch(\HSSnodey)\}\).
    \item If \(\HSSnodex\) and \(\HSSnodey\) are both leaves, then \(\HSSnode = (\HSSnodex, \HSSnodey)\) is a leaf node of \(\HSStree\).
    \item Each node \(\HSSnode = (\HSSnodex, \HSSnodey)\) is associated with the row index set \(\HSSrowindset_{\HSSnode} = \HSSrowindset_{\HSSnodex} \cap \HSSrowindset_{\HSSnodey}\) and the column index set \(\HSScolindset_{\HSSnode} = \HSScolindsetx_{\HSSnodex} \times \HSScolindsety_{\HSSnodey}\).
  \end{enumerate}
\end{definition}

Typically, from the discussions in Section~\ref{subsec:inverse_typeII_1d_nudft}, the HSS tree \(\HSStreex\) and \(\HSStreey\) are constructed by the bisection of the periodic interval \([0, 1)\).
Therefore, the HSS tree \(\HSStree\) can be viewed as a tree constructed by the quadrisection of \([0, 1)^{2}\).
Note that each node of \(\HSStree\) is a ``tensor product of'' a node of \(\HSStreex\) and a node of \(\HSStreey\), corresponding to a rectangular domain in \([0, 1)^{2}\).
We first prove the face-splitting product of two HSS trees is still an HSS tree and then prove the face-splitting product of two HSS matrices is still an HSS matrix.

\begin{theorem}[Face-splitting product of HSS trees is an HSS tree] \label{thm:face_splitting_product_hss_tree}
  Let \(\HSStree\) be a face-splitting product of two HSS trees \(\HSStreex\) and \(\HSStreey\) of row and column index sets \(\HSSrowindset\) and \(\HSScolindsetx\) and \(\HSSrowindset\) and \(\HSScolindsety\) respectively, then \(\HSStree\) is an HSS tree of row and column index sets \(\HSSrowindset\) and \(\HSScolindset\).
\end{theorem}
\begin{proof}
  It suffices to verify that the conditions in Definition~\ref{def:hss_tree} for \(\HSStree\).
  \begin{enumerate}[(1)]
    \item For the root node \(\HSSroot = (\HSSrootx, \HSSrooty)\), we have \(\HSSrowindset_{\HSSroot} = \HSSrowindset_{\HSSrootx} \cap \HSSrowindset_{\HSSrooty} = \HSSrowindset \cap \HSSrowindset = \HSSrowindset\) and \(\HSScolindset_{\HSSroot} = \HSScolindsetx_{\HSSrootx} \times \HSScolindsety_{\HSSrooty} = \HSScolindsetx \times \HSScolindsety = \HSScolindset\).
    \item If \(\HSSnode = (\HSSnodex, \HSSnodey)\) is a nonleaf node, then there are three cases.
    We only give the proof for the first case and the proof for the other two cases can be done by similar arguments.
    If \(\HSSnodex\) is a leaf and \(\HSSnodey\) is a nonleaf, then the children of \(\HSSnode\) are given by \(\ch(\HSSnode) = \{(\HSSnodex, \HSSchy): \HSSchy \in \ch(\HSSnodey)\}\).
    Therefore,
    \begin{equation*}
    \begin{aligned}
      \HSSrowindset_{\HSSnode}
      &=
      \HSSrowindset_{\HSSnodex}
      \bigcap
      \HSSrowindset_{\HSSnodey}
      =
      \HSSrowindset_{\HSSnodex}
      \bigcap
      \biggl(
        \bigsqcup_{\HSSchy\in\ch(\HSSnodey)}
        \HSSrowindset_{\HSSchy}
      \biggr)
      \\
      &=
      \bigsqcup_{\HSSchy\in\ch(\HSSnodey)}
      \biggl(
        \HSSrowindset_{\HSSnodex}
        \bigcap
        \HSSrowindset_{\HSSchy}
      \biggr)
      =
      \bigsqcup_{\HSSch\in\ch(\HSSnode)}
      \HSSrowindset_{\HSSch},
      \\
      \HSScolindset_{\HSSnode}
      &=
      \HSScolindsetx_{\HSSnodex}
      \times
      \HSScolindsety_{\HSSnodey}
      =
      \HSScolindsetx_{\HSSnodex}
      \times
      \biggl(
        \bigsqcup_{\HSSchy\in\ch(\HSSnodey)}
        \HSScolindsety_{\HSSchy}
      \biggr)
      \\
      &=
      \bigsqcup_{\HSSchy\in\ch(\HSSnodey)}
      \biggl(
        \HSScolindsetx_{\HSSnodex}
        \times
        \HSScolindsety_{\HSSchy}
      \biggr)
      =
      \bigsqcup_{\HSSch\in\ch(\HSSnode)}
      \HSScolindset_{\HSSch}.
    \end{aligned}
    \end{equation*}
  \end{enumerate}
  This completes the proof.
\end{proof}

\begin{theorem}[Face-splitting product of HSS matrices is an HSS matrix] \label{thm:face_splitting_product_hss_matrix}
  Let \(\HSStree\) be a face-splitting product of two HSS trees \(\HSStreex\) and \(\HSStreey\) of row and column index sets \(\rowindset\) and \(\colindsetx\) and \(\rowindset\) and \(\colindsety\) respectively.
  If \(\HSSHx\) and \(\HSSHy\) are two HSS matrices corresponding to \(\HSStreex\) and \(\HSStreey\), then \(\HSSH = \HSSHx \faceprod \HSSHy\) is an HSS matrix for the HSS tree \(\HSStree\).

  Furthermore, suppose that all leaves of \(\HSStreex\) and \(\HSStreey\) are at the same level \(\HSSmaxlevel\).
  Let \(\HSSrankx_{\HSSmylevel}\) and \(\HSSranky_{\HSSmylevel}\) bound the HSS block ranks of the two factor matrices at level \(\HSSmylevel\), and let \(\HSSdiagsizex_{\HSSmylevel}\) and \(\HSSdiagsizey_{\HSSmylevel}\) be the maximum numbers of columns of their diagonal blocks at that level.
  Then, for any node \(\HSSnode = (\HSSnodex, \HSSnodey)\) of \(\HSStree\) at level \(\HSSmylevel\), the corresponding HSS block ranks are at most
  \(\HSSrankx_{\HSSmylevel} \HSSdiagsizey_{\HSSmylevel}
  + \HSSranky_{\HSSmylevel} \HSSdiagsizex_{\HSSmylevel}
  + \HSSrankx_{\HSSmylevel} \HSSranky_{\HSSmylevel}\).
\end{theorem}
\begin{proof}
  Without loss of generality, assume that \(\myTx\) and \(\myTy\) are full binary HSS trees of maximum level \(\HSSmaxlevel\).
  This means that the cases (2) and (3) in Definition~\ref{def:face_splitting_product_hss_tree} will not happen.
  The maximum level of \(\myT\) is \(\HSSmaxlevel\), and every nonleaf node of \(\myT\) has exactly four children.
  General cases can be proved by similar arguments.

  Let \(\{\myDx, \myUx, \myVx, \myRx, \myWx, \myBx\}\) and \(\{\myDy, \myUy, \myVy, \myRy, \myWy, \myBy\}\) be the generators of \(\myHx\) and \(\myHy\), respectively.
  For a node \(\myt = (\mytx, \myty) \in \myT\), we define
  \begin{equation*}
    \myD_{\myt}
    \coloneq \myH(\myJ_{\myt}, \myK_{\myt})
    = \myH\bigl(\myJ_{\myt}, \myKx_{\mytx} \times \myKy_{\myty}\bigr)
    = \myHx(\myJ_{\myt}, \myKx_{\mytx})
    \faceprod \myHy(\myJ_{\myt}, \myKy_{\myty}).
  \end{equation*}
  By \(\myDx_{\mytx} = \myHx(\myJ_{\mytx}, \myKx_{\mytx})\) and \(\myJ_{\myt} = \myJ_{\mytx} \cap \myJ_{\myty}\), \(\myHx(\myJ_{\myt}, \myKx_{\mytx})\) is a submatrix corresponding to some rows of \(\myDx_{\mytx}\).
  We denote \(\myHx(\myJ_{\myt}, \myKx_{\mytx}) = \myDx_{\myt}\).
  Similarly, we denote \(\myHy(\myJ_{\myt}, \myKy_{\myty}) = \myDy_{\myt}\).
  Therefore,
  \begin{equation*}
    \myD_{\myt} = \myDx_{\myt} \faceprod \myDy_{\myt}.
  \end{equation*}

  If \(\myt = (\mytx, \myty)\) is a leaf, then both \(\mytx\) and \(\myty\) are leaves, meaning that both \(\myDx_{\mytx}\) and \(\myDy_{\myty}\) are dense matrices, thus \(\myD_{\myt}\) defined above is a dense matrix as well.
  
  If \(\myt = (\mytx, \myty)\) is a nonleaf, then by our assumption both \(\mytx\) and \(\myty\) are nonleaf and \(\myDx_{\mytx}\) and \(\myDy_{\myty}\) are \(2 \times 2\) block matrices.
  Let \(\ch(\mytx) = \{\mychx_{1}, \mychx_{2}\}\) and \(\ch(\myty) = \{\mychy_{1}, \mychy_{2}\}\), then
  \begin{equation*}
    \ch(\myt) = \{(\mychx_{1}, \mychy_{1}), (\mychx_{1}, \mychy_{2}), (\mychx_{2}, \mychy_{1}), (\mychx_{2}, \mychy_{2})\}.
  \end{equation*}
  The columns of \(\myD_{\myt}\) can be partitioned into \(4\) column blocks corresponding to \(\ch(\myt)\), with the \((\mycolchindx, \mycolchindy)\)-th column block given by
  \begin{equation*}
    \myD_{\myt; :, (\mycolchindx, \mycolchindy)}
    =
    \myDx_{\myt; :, \mycolchindx}
    \faceprod
    \myDy_{\myt; :, \mycolchindy}
    =
    \myHx(\myJ_{\myt}, \myKx_{\mychx_{\mycolchindx}})
    \faceprod
    \myHy(\myJ_{\myt}, \myKy_{\mychy_{\mycolchindy}}).
  \end{equation*}
  We partition the rows of \(\myD_{\myt}\) into \(4\) blocks according to its children, where the \((\myrowchindx, \myrowchindy)\)-th row block of \(\myD_{\myt}\) is
  \begin{equation*}
    \myD_{\myt; (\myrowchindx, \myrowchindy), :}
    = \myDx_{\myt; \myrowchindx, :}
    \faceprod
    \myDy_{\myt; \myrowchindy, :}
    = \myHx(\myJ_{(\mychx_{\myrowchindx}, \mychy_{\myrowchindy})}, \myKx_{\mytx})
    \faceprod
    \myHy(\myJ_{(\mychx_{\myrowchindx}, \mychy_{\myrowchindy})}, \myKy_{\myty}).
  \end{equation*}
  Consequently, the \(((\myrowchindx, \myrowchindy), (\mycolchindx, \mycolchindy))\)-th block of \(\myD_{\myt}\) is
  \begin{equation*}
    \myD_{\myt; (\myrowchindx, \myrowchindy), (\mycolchindx, \mycolchindy)}
    = \myDx_{\myt; \myrowchindx, \mycolchindx}
    \faceprod
    \myDy_{\myt; \myrowchindy, \mycolchindy}
    = \myHx(\myJ_{(\mychx_{\myrowchindx}, \mychy_{\myrowchindy})}, \myKx_{\mychx_{\mycolchindx}})
    \faceprod\
    \myHy(\myJ_{(\mychx_{\myrowchindx}, \mychy_{\myrowchindy})}, \myKy_{\mychy_{\mycolchindy}}).
  \end{equation*}

  In the following, we use the one-to-one mapping of the indices
  \begin{equation*}
  \begin{aligned}
    (1,1) &\leftrightarrow (\mychx_{1},\mychy_{1}) \leftrightarrow 1,
    &\qquad
    (1,2) &\leftrightarrow (\mychx_{1},\mychy_{2}) \leftrightarrow 2,
    \\
    (2,1) &\leftrightarrow (\mychx_{2},\mychy_{1}) \leftrightarrow 3,
    &\qquad
    (2,2) &\leftrightarrow (\mychx_{2},\mychy_{2}) \leftrightarrow 4.
  \end{aligned}
\end{equation*}
  We fix the row child node to be \(\mych_{1} = (\mychx_{1}, \mychy_{1})\) and discuss the block structure of \(\myD_{\myt}\).
  Note that for the off-diagonal blocks, we have
  \begin{equation*}
    \myDx_{\mytx; 1, 2}
    = \myUx_{\mychx_{1}} \myBx_{\mychx_{1}, \mychx_{2}} \myVxh_{\mychx_{2}}, \quad
    \myDy_{\myty; 1, 2}
    = \myUy_{\mychy_{1}} \myBy_{\mychy_{1}, \mychy_{2}} \myVyh_{\mychy_{2}}.
  \end{equation*}
  
  \begin{itemize}
    \item
    Case: The column child node is \(\mych_{1} = (\mychx_{1}, \mychy_{1})\), corresponding to block \((1, 1) = ((1, 1), (1, 1))\).
    
    In this case,
    \begin{equation*}
    \myD_{\myt; 1, 1}
    = \myHx(\myJ_{\mych_{1}}, \myKx_{\mychx_{1}})
    \faceprod \myHy(\myJ_{\mych_{1}}, \myKy_{\mychy_{1}})
    = \myD_{\mych_{1}}.
    \end{equation*}
    \item
    Case: The column child node is \(\mych_{2} = (\mychx_{1}, \mychy_{2})\), corresponding to block \((1, 2) = ((1, 1), (1, 2))\).

    Using the proof of Lemma~\ref{lem:face_splitting_exact_rank}, we have
    \begin{equation} \label{eq:D_lambda_1_2}
    \begin{aligned}
      \myD_{\myt; 1, 2}
      & = \myDx_{\myt; 1, 1}
      \faceprod
      \myDy_{\myt; 1, 2} = \myDx_{\mych_{1}}
      \faceprod \bigl(\myUy_{\mych_{1}} \myBy_{\mychy_{1}, \mychy_{2}} \myVyh_{\mychy_{2}}\bigr) \\
      & = \bigl(
        \myDx_{\mych_{1}}
        \faceprod
        \myUy_{\mych_{1}}\bigr)
      \bigl(
        \idmat
        \otimes
        \myBy_{\mychy_{1}, \mychy_{2}}\bigr)
      \bigl(\idmat \otimes \myVy_{\mychy_{2}}\bigr)^{\herm},
    \end{aligned}
  \end{equation}
    where \(\myUy_{\mych_{1}}\) is the submatrix of \(\myUy_{\mychy_{1}}\) corresponding to the row index set \(\myJ_{\mych_{1}}\).
    \item
    Case: The column child node is \(\mych_{3} = (\mychx_{2}, \mychy_{1})\), corresponding to block \((1, 3) = ((1, 1), (2, 1))\).

    In this case, by a similar argument as in~\eqref{eq:D_lambda_1_2}, we have
    \begin{equation} \label{eq:D_lambda_1_3}
      \myD_{\myt; 1, 3}
      = \bigl(\myUx_{\mych_{1}} \myBx_{\mychx_{1}, \mychx_{2}} \myVxh_{\mychx_{2}}\bigr)
      \faceprod
      \myDy_{\mych_{1}}
      = \bigl(\myUx_{\mych_{1}} \faceprod \myDy_{\mych_{1}}\bigr)
      \bigl(\myBx_{\mychx_{1}, \mychx_{2}} \otimes \idmat\bigr)
      \bigl(\myVx_{\mychx_{2}} \otimes \idmat\bigr)^{\herm}.
    \end{equation}
    \item
    Case: The column child node is \(\mych_{4} = (\mychx_{2}, \mychy_{2})\), corresponding to block \((1, 4) = ((1, 1), (2, 2))\).

    In this case, we have
    \begin{equation} \label{eq:D_lambda_1_4}
      \begin{aligned}
        \myD_{\myt; 1, 4}
        & = \bigl(\myUx_{\mych_{1}} \myBx_{\mychx_{1}, \mychx_{2}} \myVxh_{\mychx_{2}}\bigr)
        \faceprod
        \bigl(\myUy_{\mych_{1}} \myBy_{\mychy_{1}, \mychy_{2}} \myVyh_{\mychy_{2}}\bigr) \\
        & = \bigl(\myUx_{\mych_{1}} \faceprod \myUy_{\mych_{1}}\bigr)
        \bigl(\myBx_{\mychx_{1}, \mychx_{2}} \otimes \myBy_{\mychy_{1}, \mychy_{2}}\bigr)
        \bigl(\myVx_{\mychx_{2}} \otimes \myVy_{\mychy_{2}}\bigr)^{\herm}.
      \end{aligned}
    \end{equation}
  \end{itemize}

  Therefore, for a nonroot node \(\xi = (\alpha, \beta)\), if we define
  \begin{equation}
  \label{eq:U_V_xi}
  \begin{aligned}
    \myU_{\mych}
    &=
    \begin{bmatrix}
      \myDx_{\mych} \faceprod \myUy_{\mych}
      &
      \myUx_{\mych} \faceprod \myDy_{\mych}
      &
      \myUx_{\mych} \faceprod \myUy_{\mych}
    \end{bmatrix},
    \\
    \myV_{\mych}
    &=
    \begin{bmatrix}
      \idmat \otimes \myVy_{\mychy}
      &
      \myVx_{\mychx} \otimes \idmat
      &
      \myVx_{\mychx} \otimes \myVy_{\mychy}
    \end{bmatrix}.
  \end{aligned}
  \end{equation}
  then the off-diagonal blocks~\eqref{eq:D_lambda_1_2}, \eqref{eq:D_lambda_1_3} and \eqref{eq:D_lambda_1_4} can be written as
  \begin{equation*}
    \begin{aligned}
      & \myD_{\myt; 1, 2} = \myU_{\mych_{1};} \myB_{\mych_{1}, \mych_{2}} \myVh_{\mych_{2};},
      \quad
      \myB_{\mych_{1}, \mych_{2}} = \diag\bigl(\idmat \otimes \myBy_{\mychy_{1}, \mychy_{2}}, \Mat{0}, \Mat{0}\bigr) \\
      & \myD_{\myt; 1, 3} = \myU_{\mych_{1};} \myB_{\mych_{1}, \mych_{3}} \myVh_{\mych_{3};}, 
      \quad
      \myB_{\mych_{1}, \mych_{3}} = \diag\bigl(\Mat{0}, \myBx_{\mychx_{1}, \mychx_{2}} \otimes \idmat, \Mat{0}\bigr)\\
      & \myD_{\myt; 1, 4} = \myU_{\mych_{1};} \myB_{\mych_{1}, \mych_{4}} \myVh_{\mych_{4};},
      \quad
      \myB_{\mych_{1}, \mych_{4}} = \diag\bigl(\Mat{0}, \Mat{0}, \myBx_{\mychx_{1}, \mychx_{2}} \otimes \myBy_{\mychy_{1}, \mychy_{2}}\bigr).
    \end{aligned}
  \end{equation*}
  
    It remains to verify the nested basis relations.
  Suppose \(\myt = (\mytx, \myty)\) is a nonroot and nonleaf node.
  We only check the first row block corresponding to \(\mych_{1} = (\mychx_{1}, \mychy_{1})\).
  With the parent column indices ordered by their children, define the row-selection matrices
  \begin{equation*}
    \Mat{E}_{x}
    =
    \begin{bmatrix}
      \idmat & \Mat{0}
    \end{bmatrix}
    \in \complex^{|\myKx_{\mychx_{1}}| \times |\myKx_{\mytx}|},
    \qquad
    \Mat{E}_{y}
    =
    \begin{bmatrix}
      \idmat & \Mat{0}
    \end{bmatrix}
    \in \complex^{|\myKy_{\mychy_{1}}| \times |\myKy_{\myty}|}.
  \end{equation*}
  All identity and zero matrices below have the corresponding dimensions.
  Consider \(\myV_{\myt}\) in~\eqref{eq:U_V_xi} first.
  Using the nested basis relations in the two one-dimensional HSS trees, we have
  \begin{equation*}
    \begin{aligned}
      \myV_{\myt; 1}
      & =
      \begin{bmatrix}
        \Mat{E}_{x} \otimes \myVy_{\myty; 1}
        & \myVx_{\mytx; 1} \otimes \Mat{E}_{y}
        & \myVx_{\mytx; 1} \otimes \myVy_{\myty; 1}
      \end{bmatrix} \\
      & =
      \begin{bmatrix}
        \Mat{E}_{x} \otimes
        \bigl(\myVy_{\mychy_{1}} \myWy_{\mychy_{1}}\bigr)
        &
        \bigl(\myVx_{\mychx_{1}} \myWx_{\mychx_{1}}\bigr)
        \otimes \Mat{E}_{y}
        &
        \bigl(\myVx_{\mychx_{1}} \myWx_{\mychx_{1}}\bigr)
        \otimes
        \bigl(\myVy_{\mychy_{1}} \myWy_{\mychy_{1}}\bigr)
      \end{bmatrix} \\
      & =
      \begin{bmatrix}
        \idmat \otimes \myVy_{\mychy_{1}}
        & \myVx_{\mychx_{1}} \otimes \idmat
        & \myVx_{\mychx_{1}} \otimes \myVy_{\mychy_{1}}
      \end{bmatrix} \\
      &\quad {}\times
      \diag\bigl(
        \Mat{E}_{x} \otimes \myWy_{\mychy_{1}},
        \myWx_{\mychx_{1}} \otimes \Mat{E}_{y},
        \myWx_{\mychx_{1}} \otimes \myWy_{\mychy_{1}}
      \bigr) \\
      & = \myV_{\mych_{1}} \myW_{\mych_{1}},
    \end{aligned}
  \end{equation*}
  where \(\myW_{\mych_{1}}\) is the block diagonal matrix in the above equation.

  Similarly, for \(\myU_{\myt}\) in~\eqref{eq:U_V_xi}, we have
  \begin{equation*}
    \begin{aligned}
      \myU_{\myt; 1}
      & =
      \begin{bmatrix}
        \myDx_{\myt; 1, :} \faceprod \myUy_{\myt; 1}
        & \myUx_{\myt; 1} \faceprod \myDy_{\myt; 1, :}
        & \myUx_{\myt; 1} \faceprod \myUy_{\myt; 1}
      \end{bmatrix}.
    \end{aligned}
  \end{equation*}
  The first column block can be further factorized as
  \begin{equation*}
    \begin{aligned}
      \myU_{\myt; 1}^{[1]}
      & =
      \begin{bmatrix}
        \myDx_{\mych_{1}}
        &
        \myUx_{\mych_{1}}
        \myBx_{\mychx_{1}, \mychx_{2}}
        \myVxh_{\mychx_{2}}
      \end{bmatrix}
      \faceprod
      \bigl(
        \myUy_{\mych_{1}}
        \myRy_{\mychy_{1}}
      \bigr) \\
      & =
      \begin{bmatrix}
        \myDx_{\mych_{1}}
        \faceprod
        \bigl(
          \myUy_{\mych_{1}}
          \myRy_{\mychy_{1}}
        \bigr)
        &
        \bigl(
          \myUx_{\mych_{1}}
          \myBx_{\mychx_{1}, \mychx_{2}}
          \myVxh_{\mychx_{2}}
        \bigr)
        \faceprod
        \bigl(
          \myUy_{\mych_{1}}
          \myRy_{\mychy_{1}}
        \bigr)
      \end{bmatrix} \\
      & =
      \begin{bmatrix}
        \myDx_{\mych_{1}} \faceprod \myUy_{\mych_{1}}
        &
        \myUx_{\mych_{1}} \faceprod \myDy_{\mych_{1}}
        &
        \myUx_{\mych_{1}} \faceprod \myUy_{\mych_{1}}
      \end{bmatrix} \\
      &\quad {}\times
      \begin{bmatrix}
        \idmat \otimes \myRy_{\mychy_{1}}
        & \Mat{0} \\
        \Mat{0}
        & \Mat{0} \\
        \Mat{0}
        &
        \bigl(
          \myBx_{\mychx_{1}, \mychx_{2}}
          \myVxh_{\mychx_{2}}
        \bigr)
        \otimes \myRy_{\mychy_{1}}
      \end{bmatrix} \\
      & =
      \myU_{\mych_{1}}
      \myR_{\mych_{1}}^{[1]},
    \end{aligned}
  \end{equation*}
  where \(\myR_{\mych_{1}}^{[1]}\) is the block matrix in the above equation.
  Similarly, the second and third column blocks of \(\myU_{\myt; 1}\) can be factorized as
  \begin{equation*}
    \myU_{\myt; 1}^{[2]}
    =
    \myU_{\mych_{1}} \myR_{\mych_{1}}^{[2]},
    \qquad
    \myU_{\myt; 1}^{[3]}
    =
    \myU_{\mych_{1}} \myR_{\mych_{1}}^{[3]},
  \end{equation*}
  where
  \begin{equation*}
    \myR_{\mych_{1}}^{[2]}
    =
    \begin{bmatrix}
      \Mat{0} \\
      \myRx_{\mychx_{1}} \otimes \Mat{E}_{y} \\
      \myRx_{\mychx_{1}} \otimes
      \begin{bmatrix}
        \Mat{0}
        &
        \myBy_{\mychy_{1}, \mychy_{2}}
        \myVyh_{\mychy_{2}}
      \end{bmatrix}
    \end{bmatrix},
    \qquad
    \myR_{\mych_{1}}^{[3]}
    =
    \begin{bmatrix}
      \Mat{0} \\
      \Mat{0} \\
      \myRx_{\mychx_{1}} \otimes \myRy_{\mychy_{1}}
    \end{bmatrix}.
  \end{equation*}
  Therefore,
  \begin{equation*}
    \begin{aligned}
      \myU_{\myt; 1}
      & =
      \begin{bmatrix}
        \myU_{\myt; 1}^{[1]}
        & \myU_{\myt; 1}^{[2]}
        & \myU_{\myt; 1}^{[3]}
      \end{bmatrix} \\
      & =
      \myU_{\mych_{1}}
      \begin{bmatrix}
        \myR_{\mych_{1}}^{[1]}
        & \myR_{\mych_{1}}^{[2]}
        & \myR_{\mych_{1}}^{[3]}
      \end{bmatrix}
      =
      \myU_{\mych_{1}} \myR_{\mych_{1}}.
    \end{aligned}
  \end{equation*}
  
  Combining all the above results, we conclude that \(\myH\) is an HSS matrix about \(\tree{T}\).
  The results on the rank of HSS blocks can be easily verified by the construction process.
\end{proof}

Since the proof of Theorem~\ref{thm:face_splitting_product_hss_matrix} is constructive, it can be used to compute the HSS representation of \(\HSSH\) from the HSS representations of \(\HSSHx\) and \(\HSSHy\).
We will use it as part of our algorithm for constructing the HSS representation of the transformed 2D NUDFT matrix in Section~\ref{subsec:typeII_2d_nudft_algorithm}.
For \(\numcolx = \numcoly = \numcoldir\), assume that the two 1D HSS approximations have rank bounds
\(\HSSrank_{\HSSmylevel}
= \bigO\bigl(\mylog{(1 / \ranktol)} \mylog{(\numcoldir / 2^{\HSSmylevel})}\bigr)\)
at level \(\HSSmylevel\).
Since each diagonal block at this level has
\(\bigO(\numcoldir / 2^{\HSSmylevel})\) columns,
Theorem~\ref{thm:face_splitting_product_hss_matrix} bounds the ranks in their face-splitting product by
\(\bigO\bigl(\HSSrank_{\HSSmylevel} \numcoldir / 2^{\HSSmylevel} + \HSSrank_{\HSSmylevel}^{2}\bigr)\).
For fixed \(\ranktol\), this gives the level-dependent bound
\(\bigO\bigl((\numcoldir / 2^{\HSSmylevel}) \mylog{(\numcoldir / 2^{\HSSmylevel})}\bigr)\)
and an overall HSS rank bound of
\(\bigO(\sqrt{\numcol} \mylog{\numcol})\).

\subsection{Error estimate of the HSS approximation} \label{subsec:error_estimate_hss_approximation}

We next bound the error in approximating \(\tnudftmat\) by the face-splitting product of two 1D HSS approximations.
The argument uses the following equal-row-norm property, which follows from the constant magnitudes of the NUDFT entries and the unitarity of the DFT matrices up to scaling.

\begin{lemma}[Equal row norm property] \label{lem:typeII_2d_same_row_norm}
  Consider the 2D type-II NUDFT matrix \(\nudftmat = \nudftmatx \faceprod \nudftmaty\) defined by~\eqref{eq:typeII_2d_nudft_matrix}, where \(\nudftmatx\) and \(\nudftmaty\) are the type-II NUDFT matrices in \(\dirx\)- and \(\diry\)- direction respectively.
  Let \(\dftmat = \dftmatx \otimes \dftmaty\) be the DFT matrix in 2D and \(\tnudftmat = \nudftmat \dftmatinv = \tnudftmatx \faceprod \tnudftmaty\) where \(\tnudftmatx = \nudftmatx \dftmatxinv\) and \(\tnudftmaty = \nudftmaty \dftmatyinv\).
  Then every row of \(\tnudftmatx = \nudftmatx \dftmatxinv\) has the same norm, denoted by \(\nudftrownormx\) and every row of \(\tnudftmaty = \nudftmaty \dftmatyinv\) has the same norm, denoted by \(\nudftrownormy\).
  Consequently, every row of \(\tnudftmat = \nudftmat \dftmatinv = \tnudftmatx \faceprod \tnudftmaty\) has the same norm \(\nudftrownormx \nudftrownormy\).
  In particular, \(\|\tnudftmatx\|_{\fro} = \sqrt{\numrow} \nudftrownormx\), \(\|\tnudftmaty\|_{\fro} = \sqrt{\numrow} \nudftrownormy\) and \(\|\tnudftmat\|_{\fro} = \sqrt{\numrow} \nudftrownormx \nudftrownormy\).
\end{lemma}

\begin{theorem} \label{thm:nudft_hss_approximation}
  Using the same notations as in Lemma~\ref{lem:typeII_2d_same_row_norm}, for a given \(0 < \ranktol < 1\), suppose \(\nudfthssx\) and \(\nudfthssy\) are two HSS approximations of \(\tnudftmatx\) and \(\tnudftmaty\) such that \(\|\tnudftmatx - \nudfthssx\|_{\fro} \leq \ranktol \|\tnudftmatx\|_{\fro}\) and \(\|\tnudftmaty - \nudfthssy\|_{\fro} \leq \ranktol \|\tnudftmaty\|_{\fro}\).
  Assume that there exists a constant \(\nudftconst\) such that \(\max_{\rowind} \{\|\nudfthssx(\rowind, :)\|_{2}\} \leq \nudftconst \max_{\rowind} \{\|\tnudftmatx(\rowind, :)\|_{2}\}\).
  Then \(\nudfthss = \nudfthssx \faceprod \nudfthssy\) is an HSS approximation of \(\tnudftmat = \tnudftmatx \faceprod \tnudftmaty\) and \(\|\nudfthss - \tnudftmat\|_{\fro} \leq (1 + \nudftconst) \ranktol \|\tnudftmat\|_{\fro}\).
\end{theorem}
\begin{proof}
  From Lemma~\ref{lem:faceprod_F_norm} and Lemma~\ref{lem:typeII_2d_same_row_norm}, we have
  \begin{equation*}
    \begin{aligned}
       \|\tnudftmat - \nudfthss\|_{\fro}
        & = \|(\tnudftmatx - \nudfthssx) \faceprod \tnudftmaty
        + \nudfthssx \faceprod (\tnudftmaty - \nudfthssy)\|_{\fro} \\
        & \leq \|(\tnudftmatx - \nudfthssx) \faceprod \tnudftmaty\|_{\fro}
        + \| \nudfthssx \faceprod (\tnudftmaty - \nudfthssy)\|_{\fro} \\
        & \leq \ranktol \|\tnudftmatx\|_{\fro} \nudftrownormy + \nudftconst \nudftrownormx \ranktol \|\tnudftmaty\|_{\fro}\\
        & = \ranktol \sqrt{\numrow} \nudftrownormx \nudftrownormy + \nudftconst \ranktol \sqrt{\numrow} \nudftrownormx \nudftrownormy = (1 + \nudftconst) \ranktol \|\tnudftmat\|_{\fro}.
    \end{aligned}
  \end{equation*}
  This completes the proof.
\end{proof}

The Frobenius-norm assumption implies that
\(\nudftconst = 1 + \sqrt{\numrow} \ranktol\)
is admissible, giving the relative error bound
\((2 + \sqrt{\numrow} \ranktol)\ranktol\).
Therefore, one can expect the approximation error in Theorem~\ref{thm:nudft_hss_approximation} is not large when \(\sqrt{\numrow} \ranktol\) is small.
However, this estimate is quite pessimistic.
If every row of \(\tnudftmatx - \nudfthssx\) has approximately the same norm~(and this could be the case in practice), then one could expect that \(\nudftconst\) is a small constant independent of \(\numrow\) and \(\numcol\), meaning that the approximation error in Theorem~\ref{thm:nudft_hss_approximation} is small.

\section{A Direct Solver for the Inverse NUDFT} \label{sec:typeII_2d_nudft_solver}

In this section, we present a superfast solver for the inverse NUDFT problem~\eqref{eq:typeII_2d_nudft} based on the HSS approximation of \(\tnudftmat\).
Section~\ref{subsec:typeII_2d_nudft_algorithm} describes the algorithm in detail and Section~\ref{subsec:complexity_analysis} provides a complexity analysis of it.

\subsection{An Algorithm for the Direct Inversion of Type-II 2D NUDFT} \label{subsec:typeII_2d_nudft_algorithm}

Based on the discussion in the previous sections, we design an algorithm for solving the 2D type-II inverse NUDFT using an HSS approximation of \(\tnudftmat\).
The algorithm consists of two stages.

The \emph{offline} stage constructs \(\nudfthss\) from the face-splitting product of two 1D HSS approximations, recompresses it as described in Section~\ref{subsec:hss_recompression}, and computes its URV factorization.
The factors can be reused for different target vectors with the same sample points, frequency grid, and approximation tolerance.
Writing
\begin{equation*}
  \nudftmat = \tnudftmat \dftmat
  \approx \nudftmat_{\fast} \coloneq \nudfthss \dftmat,
\end{equation*}
the \emph{online} stage computes
\begin{equation*}
  \coeffvec = \nudftmat_{\fast}^{\pinv} \targetvvec
  = \dftmatinv \nudfthss^{\pinv} \targetvvec
\end{equation*}
using the URV solver followed by a two-dimensional iFFT.
Algorithm~\ref{alg:typeII_2d_nudft_solver} summarizes these steps.

The direct solver can also be used as a preconditioner for lsqr.
The iterations solve the original NUDFT least-squares problem, using the NUFFT to apply \(\nudftmat\) and its adjoint.
The preconditioner is applied by the URV solver followed by a two-dimensional iFFT.
Its adjoint,
\(\nudftmat_{\fast}^{\pinv, \herm}
= \nudfthss^{\pinv \herm} \dftmat / \numcol\),
is applied by a two-dimensional FFT followed by the adjoint URV solver.

\begin{algorithm}[htbp]
  \caption{A direct solver for the type-II 2D NUDFT problem} \label{alg:typeII_2d_nudft_solver}
  \begin{algorithmic}[1]
    \REQUIRE{Sample points \(\{(\ptx_{\rowind}, \pty_{\rowind})\}_{\rowind = 0}^{\numrow - 1}\) and number of frequencies \(\numcolx\) and \(\numcoly\), target values \(\{\targetv_{\rowind}\}_{\rowind = 0}^{\numrow - 1}\), accuracy parameter \(\ranktol\).}
    \ENSURE{Coefficients \(\coeffvec\).}
    \STAGE{Offline stage: Construction.}
    \STATE{Construct the HSS approximation \(\nudfthssx\) of \(\tnudftmatx\) and \(\nudfthssy\) of \(\tnudftmaty\) using the fADI-based method with accuracy parameter \(\ranktol / 10 \) in Section~\ref{subsec:inverse_typeII_1d_nudft}.}
    \STATE{Compute the HSS approximation \(\nudfthss = \nudfthssx \faceprod \nudfthssy\) of \(\tnudftmat = \nudftmat \dftmatinv\) using the proof of Theorem~\ref{thm:face_splitting_product_hss_matrix}.}
    \STATE{Recompress the HSS approximation \(\nudfthss\) by the algorithm in Section~\ref{subsec:hss_recompression} with accuracy parameter \(\ranktol\).}
    \STAGE{Offline stage: Factorization.}
    \STATE{Compute the URV factorization of \(\nudfthss\).}
    \STAGE{Online stage: Solving.}
    \STATE{Compute \(\coeffvec = \dftmatinv \nudfthss^{\pinv} \targetvvec\) by the URV solution and two-dimensional iFFT.}
  \end{algorithmic}
\end{algorithm}

\subsection{Complexity Analysis} \label{subsec:complexity_analysis}

We analyze the complexity of Algorithm~\ref{alg:typeII_2d_nudft_solver} under a simplified setting.
Suppose \(\HSStreex\) and \(\HSStreey\) are both full binary HSS trees with maximum level \(\HSSmaxlevel\) constructed by bisecting the column index sets and all leaf nodes have exact column size \(\numcoldir_{\HSSmaxlevel}\), where \(\numcoldir_{\HSSmaxlevel}\) is a small constant independent of \(\numcolx\) and \(\numcoly\).
That is, \(\numcolx = \numcoly = \numcoldir = 2^{\HSSmaxlevel} \numcoldir_{\HSSmaxlevel}\) and \(\numcol = 4^{\HSSmaxlevel} \numcol_{\HSSmaxlevel}\) where \(\numcol_{\HSSmaxlevel} = \numcoldir_{\HSSmaxlevel}^{2}\) is fixed independently of \(\numrow\) and \(\numcol\).
Let \(\HSStree = \HSStreex \faceprod \HSStreey\) and assume further that each leaf node contains \(\numrow_{\HSSmaxlevel}\) sample points, so that \(\numrow = 4^{\HSSmaxlevel} \numrow_{\HSSmaxlevel}\).
Then \(\numrow_{\HSSmaxlevel} = (\numrow / \numcol) \numcol_{\HSSmaxlevel}\), while \(\numrow_{\HSSmaxlevel}\) need not remain bounded.
At level \(\HSSmylevel\), the product tree has \(4^{\HSSmylevel}\) nodes; let \(\HSSrank_{\HSSmylevel}\) bound the ranks of their HSS blocks.
For fixed \(\ranktol\), write \(\numcoldir_{\HSSmylevel} = \numcoldir / 2^{\HSSmylevel}\) and assume
\(\HSSrank_{\HSSmylevel} = \bigO(\numcoldir_{\HSSmylevel} \mylog{\numcoldir_{\HSSmylevel}})\),
as motivated in Sections~\ref{subsec:kernel_matrix_perspective} and~\ref{subsec:face_splitting_product_hss_tree_matrix}.

Consider the HSS construction of Algorithm~\ref{alg:typeII_2d_nudft_solver}.
First, the two one-dimensional HSS matrices \(\nudfthssx\) and \(\nudfthssy\) both have size \(\numrow \times \numcoldir\).
Applying the fADI-based construction in~\cite[Section~3.5]{Wilber_Epperly_Barnett_2025}, as reviewed in Section~\ref{subsec:inverse_typeII_1d_nudft}, gives a total cost of \(\bigO\bigl((\numrow + \numcoldir) \mylog[2]{\numcoldir}\bigr) = \bigO\bigl((\numrow + \sqrt{\numcol}) \mylog[2]{\numcol}\bigr)\).

Next, the face-splitting product is formed by assembling the generators in the proof of Theorem~\ref{thm:face_splitting_product_hss_matrix}.
With the corresponding one-dimensional rank model \(\bigO\bigl(\mylog{\numcoldir_{\HSSmylevel}}\bigr)\), the column sizes of the resulting basis matrices are bounded by \(\HSSrank_{\HSSmylevel} = \bigO\bigl(\numcoldir_{\HSSmylevel} \mylog{\numcoldir_{\HSSmylevel}}\bigr)\) before recompression.
The interaction and transfer matrices are assembled directly from Kronecker products, requiring \(\bigO(\HSSrank_{\HSSmylevel}^{2})\) operations per node.
The sample points are grouped once according to their pairs of one-dimensional leaves, and the product tree and its index sets can be organized in \(\bigO(\numrow + \numcol)\) operations.
Summing over the \(4^{\HSSmylevel}\) nodes at each level gives a total face-splitting cost of \(\bigO\bigl(\numrow + \numcol \mylog[3]{\numcol}\bigr)\).

The recompression cost follows the analysis in~\cite[Section~5]{Xia_2012}.
The main differences are that our HSS tree is a quad tree and the generator sizes depend on the level.
The QR factorizations, SVDs and basis transformations require \(\bigO(\HSSrank_{\HSSmylevel}^{3})\) operations per nonleaf node.
Since the leaf basis matrices have a bounded number of columns, their thin QR factorizations and basis updates cost \(\bigO(\numrow_{\HSSmaxlevel} + 1)\) operations per leaf, while the SVDs act only on small weight matrices.
Summing over all \(4^{\HSSmaxlevel}\) leaves gives a total leaf cost of \(\bigO(\numrow + \numcol)\).
Consequently, the recompression complexity is \(\bigO\bigl(\numrow + \numcol^{3 / 2} \mylog[3]{\numcol}\bigr)\).

Similarly, the URV factorization and solution costs follow from~\cite{Xi_Xia_Cauley_Balakrishnan_2014, Wilber_Epperly_Barnett_2025}, with the binary tree replaced by a quad tree and the fixed rank replaced by the level-dependent rank \(\HSSrank_{\HSSmylevel}\).
The thin QR factorization in the size reduction step keeps the leaf cost linear in the number of sample points.
The remaining factorization and solution operations cost \(\bigO(\HSSrank_{\HSSmylevel}^{3})\) and \(\bigO(\HSSrank_{\HSSmylevel}^{2})\) per node, respectively.
Thus, the URV factorization costs \(\factorcomplexity\), and the URV solution for one right-hand side costs \(\solvecomplexity\).
Applying the adjoint of the URV solution operator has the same arithmetic complexity and uses the same factors.

Combining the four steps gives the bound \(\offlinecomplexity\) for the offline stage.
The two-dimensional iFFT adds \(\bigO(\numcol \mylog{\numcol})\) operations to each solve, so the online arithmetic complexity remains \(\onlinecomplexity\).

\section{Numerical Results} \label{sec:numerical_results}

In this section, we evaluate the computational cost and reconstruction accuracy of the proposed direct solver for the 2D type-II NUDFT.
We also examine its use as a preconditioner for lsqr.
Two types of nonuniform grids are considered: a random grid and a modified polar grid~\cite{Fenn_Kunis_Potts_2007}.
The random grid consists of \(\numrow\) points sampled independently and uniformly from \([0, 1)^{2}\).
The polar grid consists of the center point \((1/2, 1/2)\) together with points of the form
\begin{equation*}
  \begin{gathered}
    x_{(p, q)} = \frac{1}{2} + r_{p} \cos (2 \pi t_{q}), \quad
    y_{(p, q)} = \frac{1}{2} + r_{p} \sin (2 \pi t_{q}), \\
    r_{p} = \frac{\sqrt{2}}{2} \frac{p}{n_{r}}, \quad
    t_{q} = \frac{q}{n_{t}}, \quad
    p = 1, \ldots, n_{r} - 1, \quad
    q = 0, 1, \ldots, n_{t} - 1,
  \end{gathered}
\end{equation*}
where points outside \([0, 1)^{2}\) are excluded.
Examples of the two grids are shown in Figure~\ref{fig:nonuniform_grids}.
The polar sampling geometry is motivated by Fourier reconstruction in CT, as discussed in Section~\ref{sec:introduction}.
Together with the random grid, it allows us to examine the proposed solver for different distributions of nonuniform Fourier samples.

\begin{figure}[htbp]
    \centering
    \begin{subfigure}{0.3\textwidth}
        \centering
        \includegraphics[width=\linewidth]{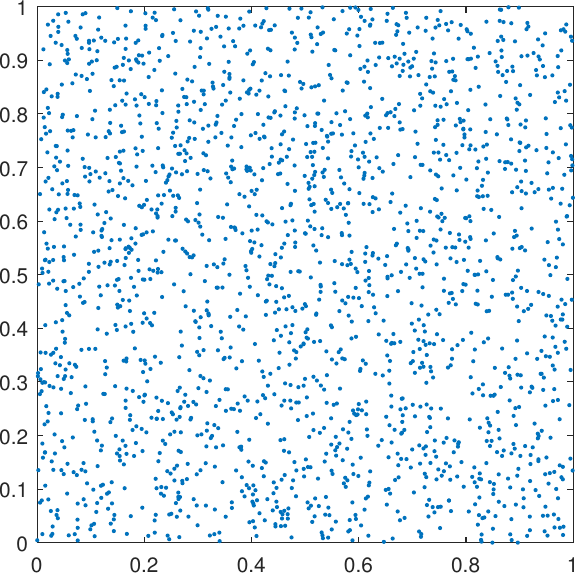}
    \end{subfigure}
    \hspace{0.01\textwidth}
    \begin{subfigure}{0.3\textwidth}
        \centering
        \includegraphics[width=\linewidth]{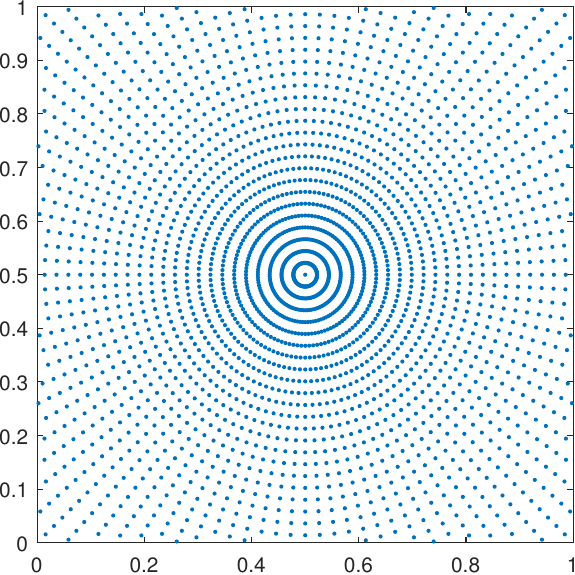}
    \end{subfigure}
    \caption{Examples of nonuniform grids with \(\numcol = 32^{2}\).
    Left: a random grid.
    Right: a polar grid.}
    \label{fig:nonuniform_grids}
\end{figure}

We let \(\numcoldir\) range over \(\{32, 64, 128, 256, 512\}\) and set
\(\numcolx = \numcoly = \numcoldir\), so that \(\numcol = \numcoldir^{2}\).
For the random grid, we use \(\numrow = 1.5 \numcol\).
For the polar grid, we set \(n_{r} = \numcoldir\) and
\(n_{t} = \lceil \beta \numcoldir \mylog_{2}{\numcoldir} \rceil\) with \(\beta = 0.6\).
For the resulting grids, the number of sample points is approximately
\(\numrow = 0.47 \numcol \mylog_{4}{\numcol}\).

Following the experimental setup in~\cite[Example~5.4]{Kircheis_Potts_2023}, we use the modified Shepp-Logan phantom as the exact solution.
Its pixel values on a \(\numcoldir \times \numcoldir\) grid form the coefficient vector \(\coeffvec_{\exact}\).
The target values \(\targetvvec\) are generated by applying the type-II 2D NUFFT to \(\coeffvec_{\exact}\), with no additional noise.
This setup isolates the error and computational cost of solving the discrete Fourier inverse problem.
In addition to reporting residuals and errors, we display the reconstructed phantom and a horizontal row profile for \(\numcoldir = 512\).
The images and profiles show the real parts of the reconstructed coefficients.
All phantom images use the same grayscale range \([0, 1]\).
The profiles correspond to the \(416\)th row of the coefficient array, which intersects three small interior features of the phantom.

The HSS tolerance \(\ranktol\) is set to \(10^{-2}\) or \(10^{-4}\).
The leaf column size of each one-dimensional HSS tree is fixed at \(16\), giving a leaf column size of \(256\) for the product tree.
For the iterative solution, we use MATLAB's \texttt{lsqr}~\cite{Paige_Saunders_1982} with a tolerance of \(10^{-12}\) and a maximum of \(500\) iterations.
We use lsqr to denote the unpreconditioned method and plsqr to denote the method preconditioned by the proposed direct solver.
All algorithms are implemented in MATLAB R2023b, and the FINUFFT library~\cite{Barnett_Magland_Klinteberg_2019} is used for the NUFFT evaluations.
The experiments are carried out on a server with two Intel Gold 6226R CPUs at 2.90 GHz and 1000.6 GB of RAM.

We report the following quantities.
The HSS rank after recompression is denoted by \(\HSSrank_{\h}\).
The construction time \(t_{\construct}\) includes the construction of the two one-dimensional HSS matrices, the face-splitting product, and recompression.
The factorization time \(t_{\factor}\) measures the URV factorization, and the solution time \(t_{\solve}\) measures one application of the URV solver followed by the two-dimensional iFFT.
For a computed solution \(\coeffvec\), the relative residual and relative error are
\begin{equation*}
    r_{\solve}
    = \frac{\|\nudftmat \coeffvec - \targetvvec\|_{2}}
           {\|\targetvvec\|_{2}},
    \quad
    e_{\solve}
    = \frac{\|\coeffvec - \coeffvec_{\exact}\|_{2}}
           {\|\coeffvec_{\exact}\|_{2}}.
\end{equation*}
These quantities are evaluated using the full computed solution before taking its real part for visualization.
For iterative solvers, \(t_{\pre} = t_{\construct} + t_{\factor}\) denotes the preprocessing time, \(t_{\iter}\) denotes the time spent in the iterative solution, and \(n_{\iter}\) denotes the number of iterations.
For a single right-hand side, the total time is \(t_{\construct} + t_{\factor} + t_{\solve}\) for the direct solver and \(t_{\pre} + t_{\iter}\) for plsqr.
When the sampling points, coefficient grid, and approximation tolerance remain fixed, the offline factors can be reused for different target vectors.
For \(n_{\mathrm{rhs}}\) right-hand sides, the total direct solution time is therefore approximately \(t_{\construct} + t_{\factor} + n_{\mathrm{rhs}} t_{\solve}\).
The practical benefit of this reuse depends on the number of right-hand sides and the solution cost of competing methods at a comparable accuracy.
All times in the tables are reported in seconds.

\subsection{Random Grid}

Figure~\ref{fig:rand_rank_construct_factor_solve_alpha_1.5} shows the HSS rank and the construction, factorization, and solution times for random grids with \(\numrow = 1.5 \numcol\).
The rank growth agrees well with the estimate
\(\bigO\bigl(\sqrt{\numcol} \mylog{\numcol}\bigr)\).
For this sampling ratio, the complexity bounds in Section~\ref{subsec:complexity_analysis} reduce to
\(\bigO\bigl(\numcol^{3 / 2} \mylog[3]{\numcol}\bigr)\) for construction and factorization, and
\(\bigO\bigl(\numcol \mylog[3]{\numcol}\bigr)\) for solution.
The dashed curves indicate these reference scalings.
The measured solution times are much smaller than the construction and factorization times, consistent with the complexity estimates in Section~\ref{subsec:complexity_analysis}.

\begin{figure}[tbhp]
    \centering
    \begin{subfigure}{0.37\textwidth}
        \centering
        \includegraphics[width=\linewidth]{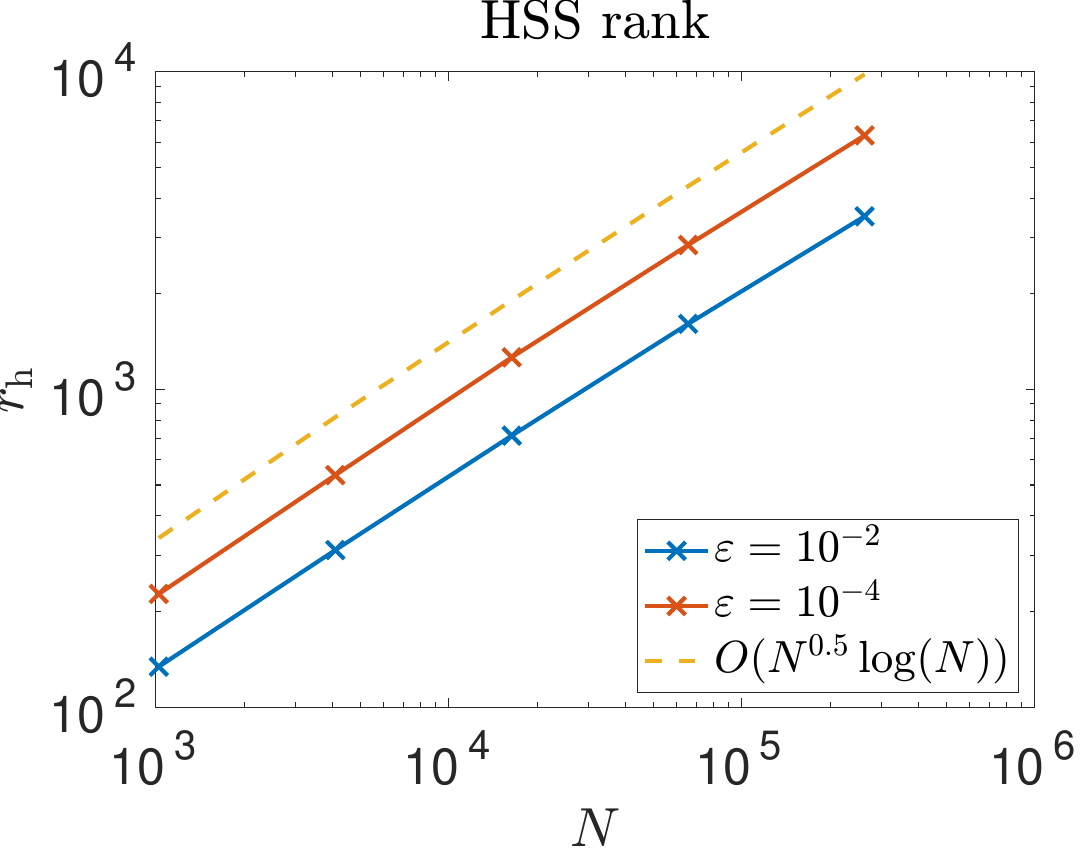}
    \end{subfigure}
    \hspace{0.01\textwidth}
    \begin{subfigure}{0.37\textwidth}
        \centering
        \includegraphics[width=\linewidth]{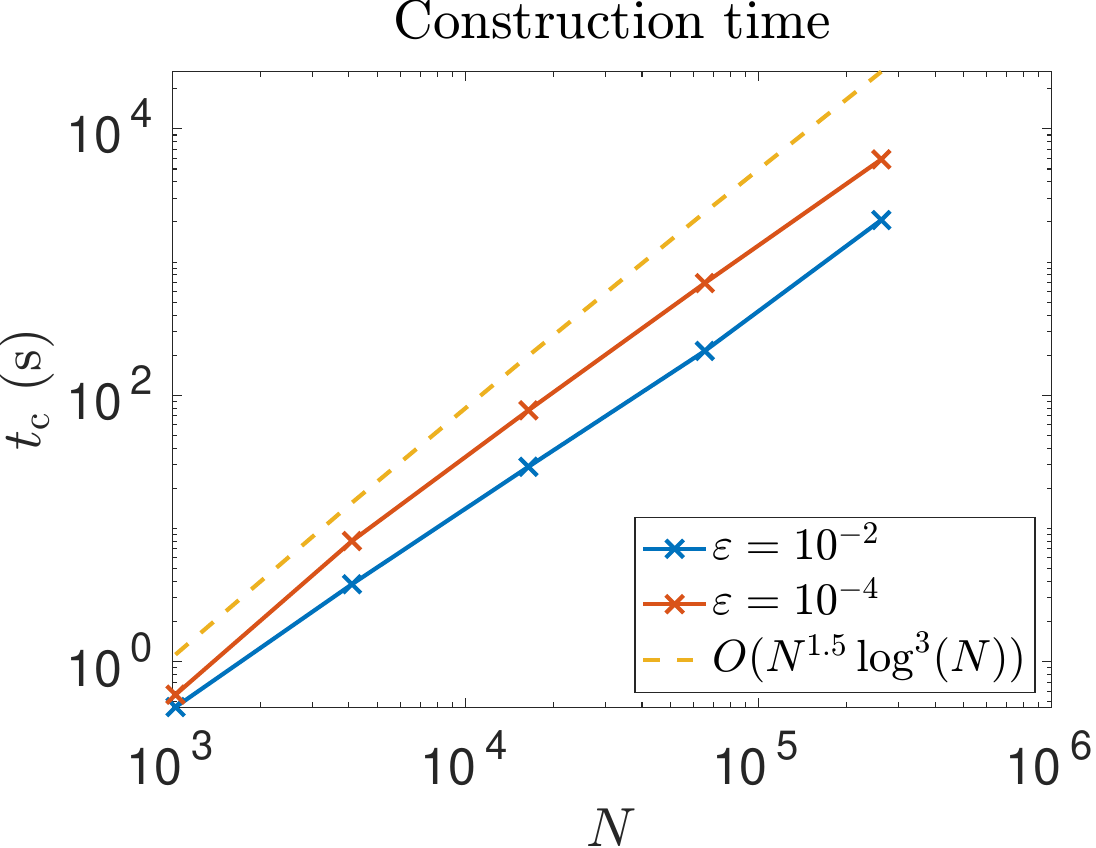}
    \end{subfigure}
    \begin{subfigure}{0.37\textwidth}
        \centering
        \includegraphics[width=\linewidth]{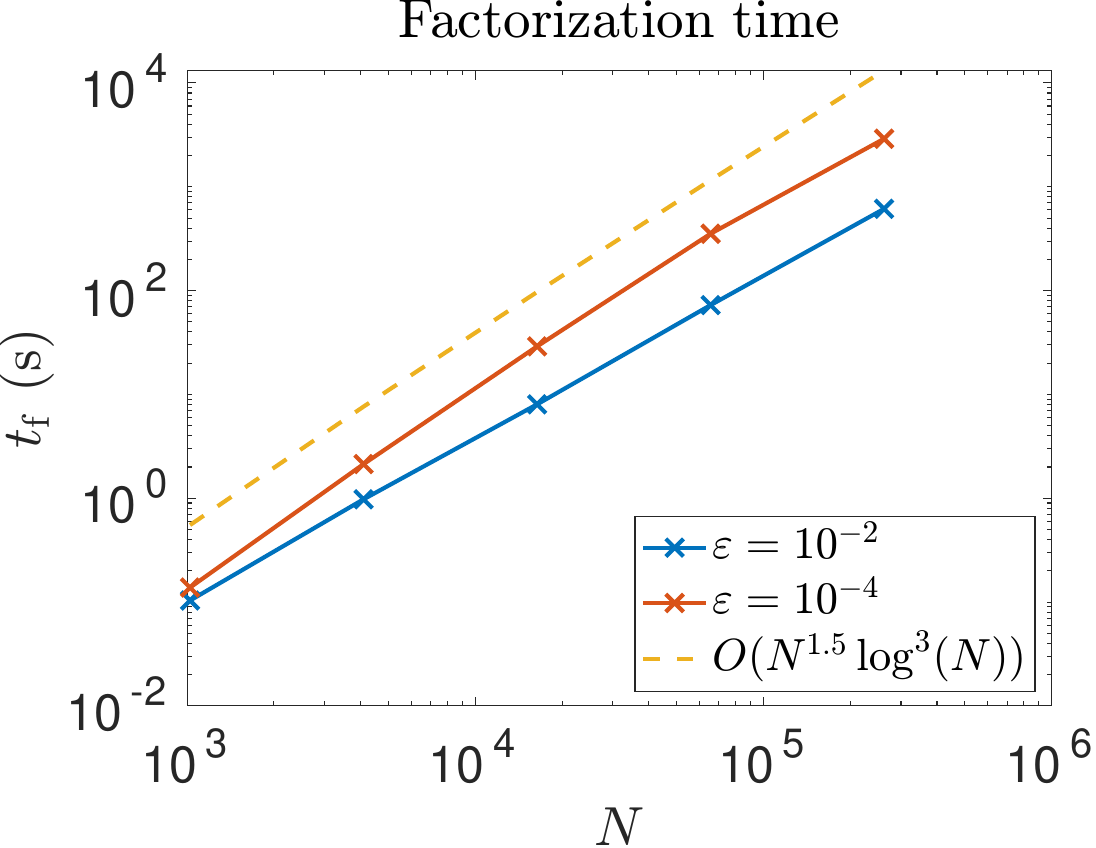}
    \end{subfigure}
    \hspace{0.01\textwidth}
    \begin{subfigure}{0.37\textwidth}
        \centering
        \includegraphics[width=\linewidth]{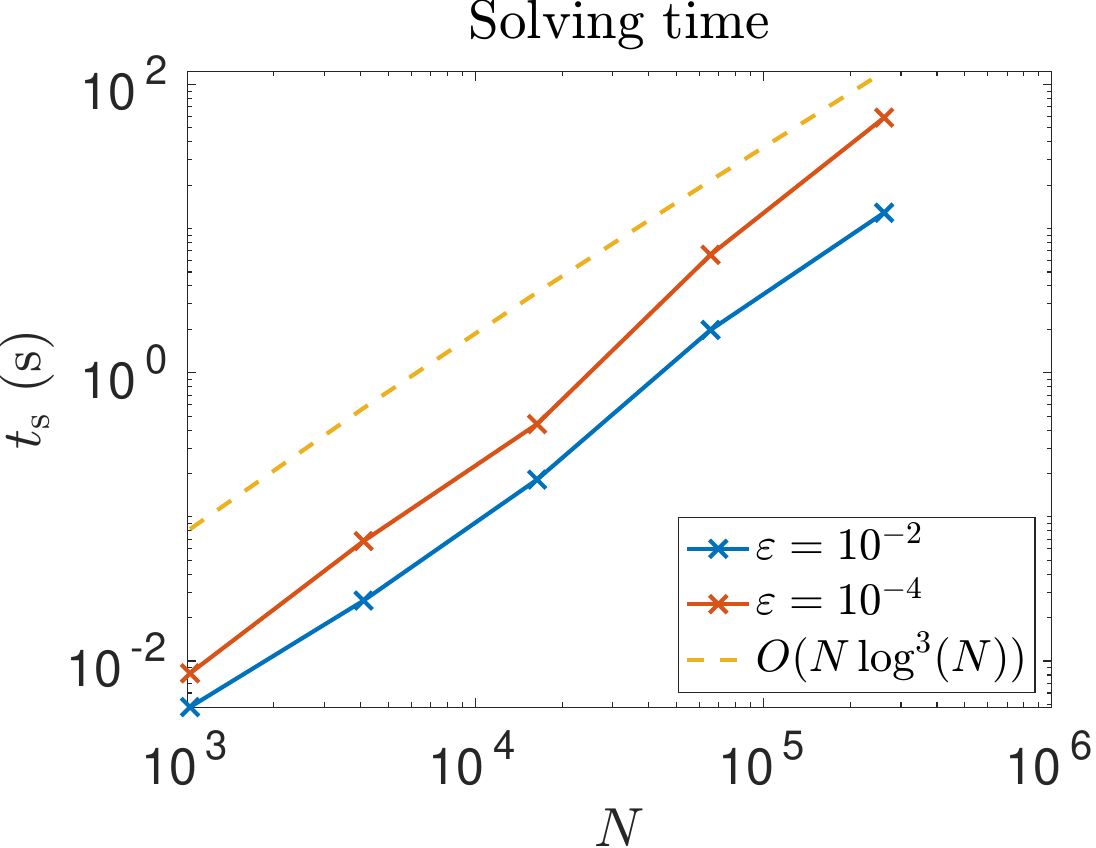}
    \end{subfigure}
    \caption{HSS rank, construction time, factorization time and solution time of the direct solver on random grids with \(\numrow = 1.5 \numcol\).
    Dashed curves show the theoretical reference scalings.}
    \label{fig:rand_rank_construct_factor_solve_alpha_1.5}
\end{figure}

Table~\ref{tab:typeII_2d_rand_direct_1.5} shows that the relative reconstruction errors remain comparable to the HSS tolerance over the tested problem sizes.
They range from \(6.4 \times 10^{-3}\) to \(1.4 \times 10^{-2}\) for \(\ranktol = 10^{-2}\), and from \(7.4 \times 10^{-5}\) to \(1.4 \times 10^{-4}\) for \(\ranktol = 10^{-4}\).
The tighter tolerance increases both preprocessing and solution costs.
At \(\numcol = 512^{2}\), the combined construction and factorization times are approximately \(45\) minutes and \(2.4\) hours for the two tolerances, while one solution takes \(13\) and \(59\) seconds, respectively.

\begin{table}[tbhp]
\centering
\begin{tabular}{cc c ccccc}
\toprule
\(N\) & \(M\) & \(\varepsilon\) & \(t_{\mathrm{c}}\) & \(t_{\mathrm{f}}\) & \(t_{\mathrm{s}}\) & \(r_{\mathrm{s}}\) & \(e_{\mathrm{s}}\) \\
\midrule
\multirow{2}{*}{\(32^2\)} & \multirow{2}{*}{1536} & \(10^{-2}\) & 4.5e-01 & 1.0e-01 & 4.8e-03 & 2.2e-03 & 6.4e-03 \\
 & & \(10^{-4}\) & 5.6e-01 & 1.4e-01 & 8.2e-03 & 2.1e-05 & 7.9e-05 \\
\midrule
\multirow{2}{*}{\(64^2\)} & \multirow{2}{*}{6144} & \(10^{-2}\) & 3.8e+00 & 9.8e-01 & 2.6e-02 & 3.5e-03 & 1.3e-02 \\
 & & \(10^{-4}\) & 8.0e+00 & 2.1e+00 & 6.8e-02 & 2.8e-05 & 7.4e-05 \\
\midrule
\multirow{2}{*}{\(128^2\)} & \multirow{2}{*}{24576} & \(10^{-2}\) & 2.9e+01 & 8.0e+00 & 1.8e-01 & 3.2e-03 & 9.7e-03 \\
 & & \(10^{-4}\) & 7.7e+01 & 2.9e+01 & 4.4e-01 & 2.5e-05 & 1.0e-04 \\
\midrule
\multirow{2}{*}{\(256^2\)} & \multirow{2}{*}{98304} & \(10^{-2}\) & 2.1e+02 & 7.2e+01 & 2.0e+00 & 4.8e-03 & 1.1e-02 \\
 & & \(10^{-4}\) & 7.0e+02 & 3.5e+02 & 6.6e+00 & 3.8e-05 & 9.2e-05 \\
\midrule
\multirow{2}{*}{\(512^2\)} & \multirow{2}{*}{393216} & \(10^{-2}\) & 2.1e+03 & 6.1e+02 & 1.3e+01 & 4.5e-03 & 1.4e-02 \\
 & & \(10^{-4}\) & 5.9e+03 & 2.9e+03 & 5.9e+01 & 4.2e-05 & 1.4e-04 \\
\bottomrule
\end{tabular}
\caption{Construction time, factorization time, solution time, relative residual and relative error of the direct solver on random grids with \(\numrow = 1.5 \numcol\).}
\label{tab:typeII_2d_rand_direct_1.5}
\end{table}

Figure~\ref{fig:rand_phantom_alpha_1.5} compares the exact phantom with the direct reconstructions at \(\numcol = 512^{2}\).
Both reconstructed images are visually close to the exact phantom at the displayed scale.
The row profiles show that the three small interior features are recovered at both tolerances.
For \(\ranktol = 10^{-2}\), small fluctuations remain in the constant regions of the profile.
These deviations are substantially reduced when \(\ranktol = 10^{-4}\), consistent with the smaller relative error reported in Table~\ref{tab:typeII_2d_rand_direct_1.5}.

\begin{figure}[tbhp]
    \centering
    \begin{subfigure}{0.3\textwidth}
        \centering
        \includegraphics[width=\linewidth]{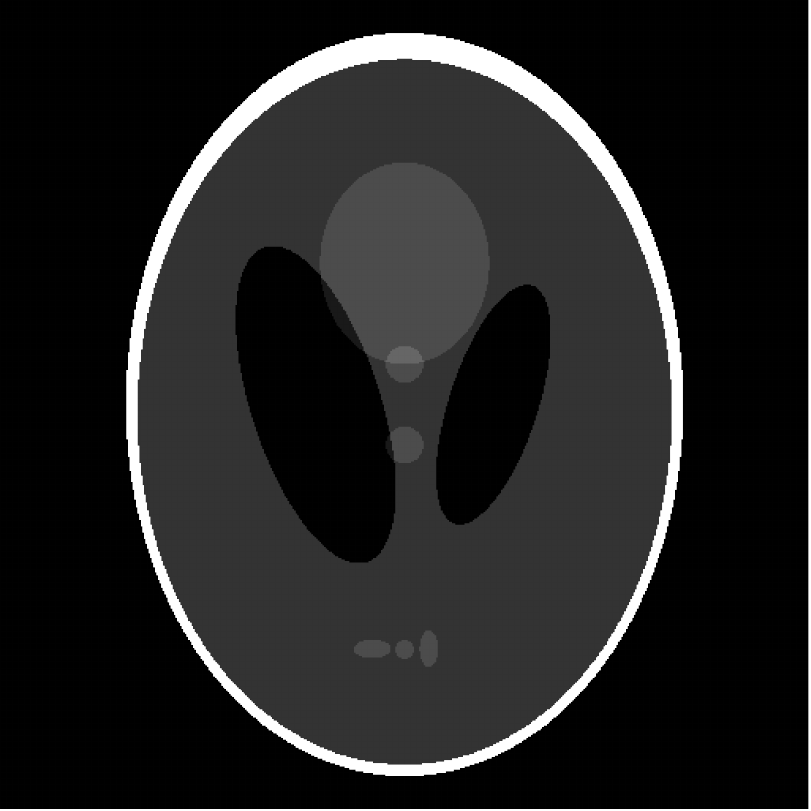}
    \end{subfigure}
    \hspace{0.01\textwidth}
    \begin{subfigure}{0.3\textwidth}
        \centering
        \includegraphics[width=\linewidth]{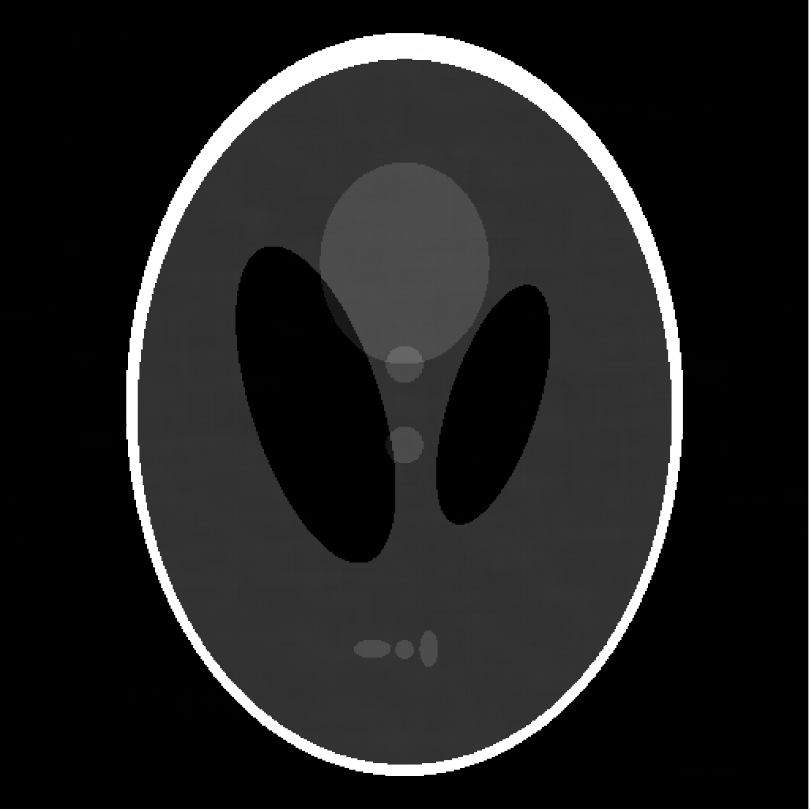}
    \end{subfigure}
    \hspace{0.01\textwidth}
    \begin{subfigure}{0.3\textwidth}
        \centering
        \includegraphics[width=\linewidth]{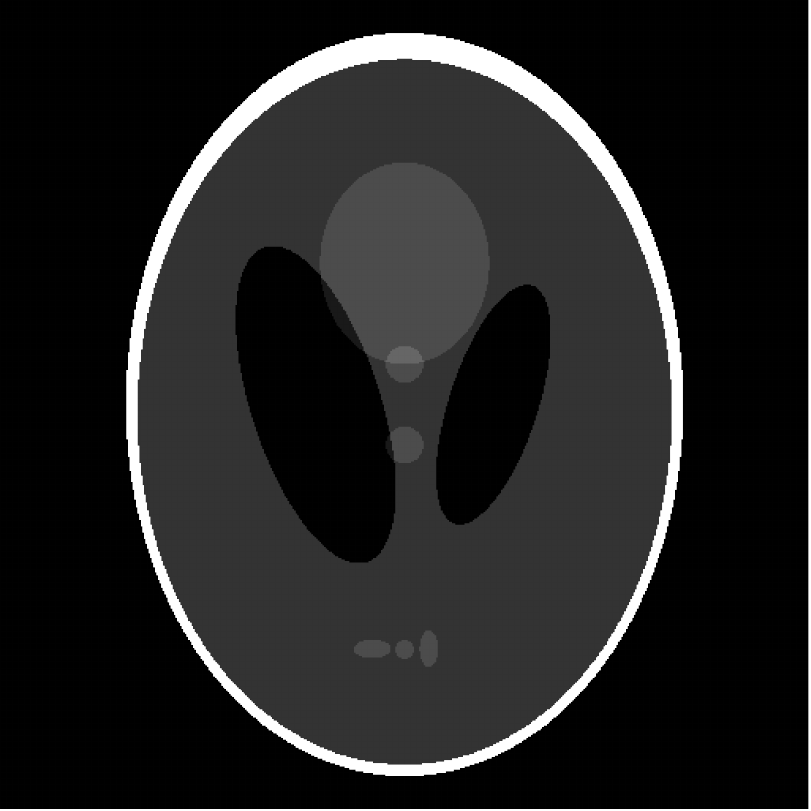}
    \end{subfigure}
    \par\medskip
    \begin{subfigure}{0.3\textwidth}
        \centering
        \includegraphics[width=\linewidth]{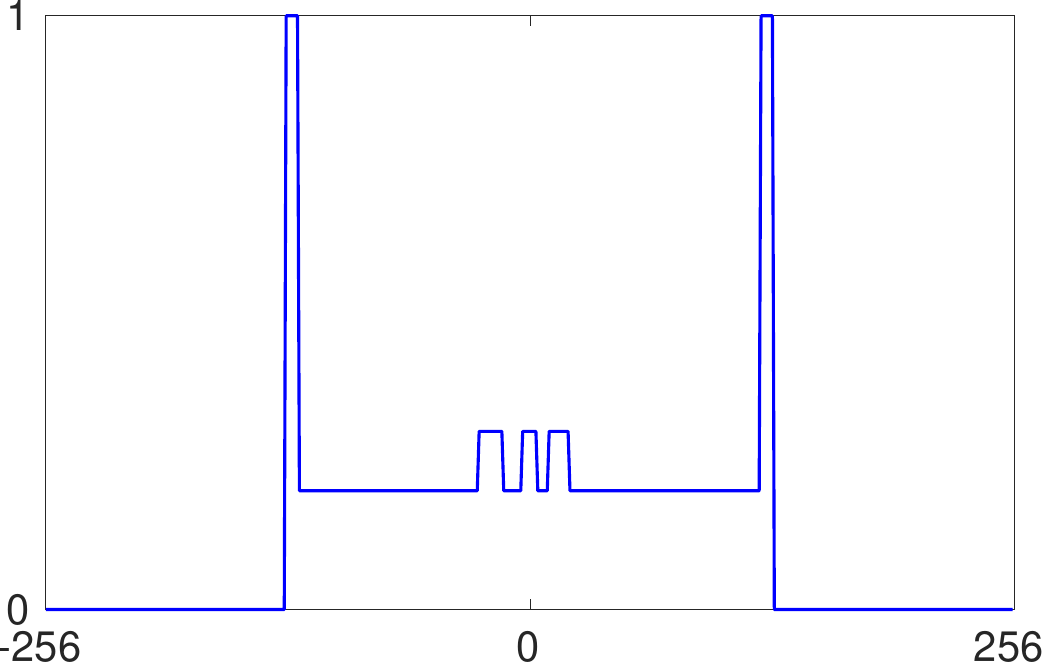}
    \end{subfigure}
    \hspace{0.01\textwidth}
    \begin{subfigure}{0.3\textwidth}
        \centering
        \includegraphics[width=\linewidth]{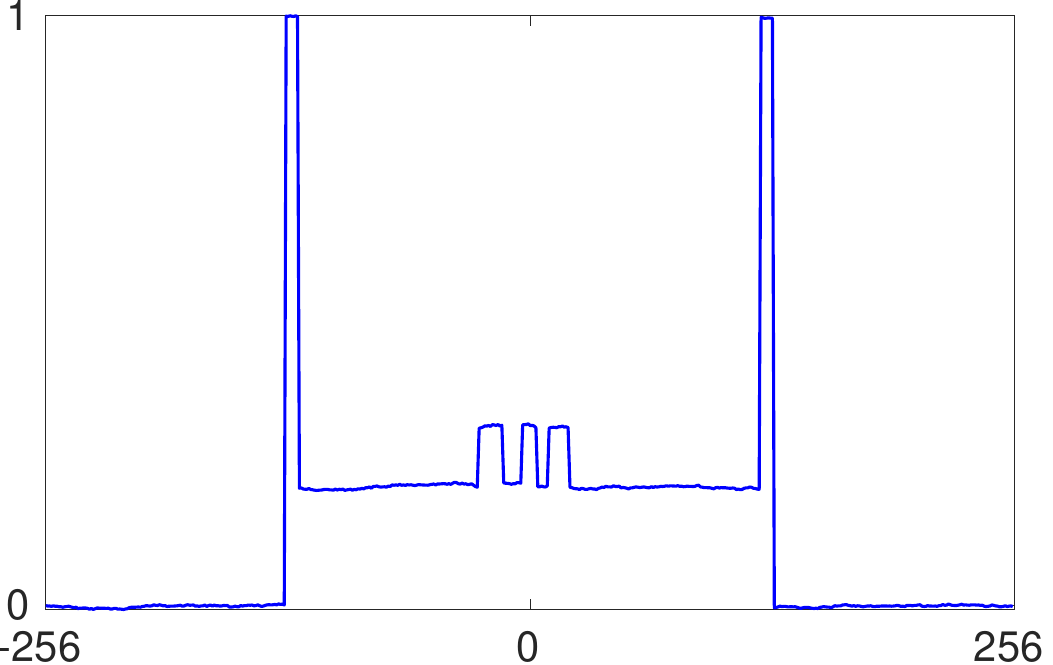}
    \end{subfigure}
    \hspace{0.01\textwidth}
    \begin{subfigure}{0.3\textwidth}
        \centering
        \includegraphics[width=\linewidth]{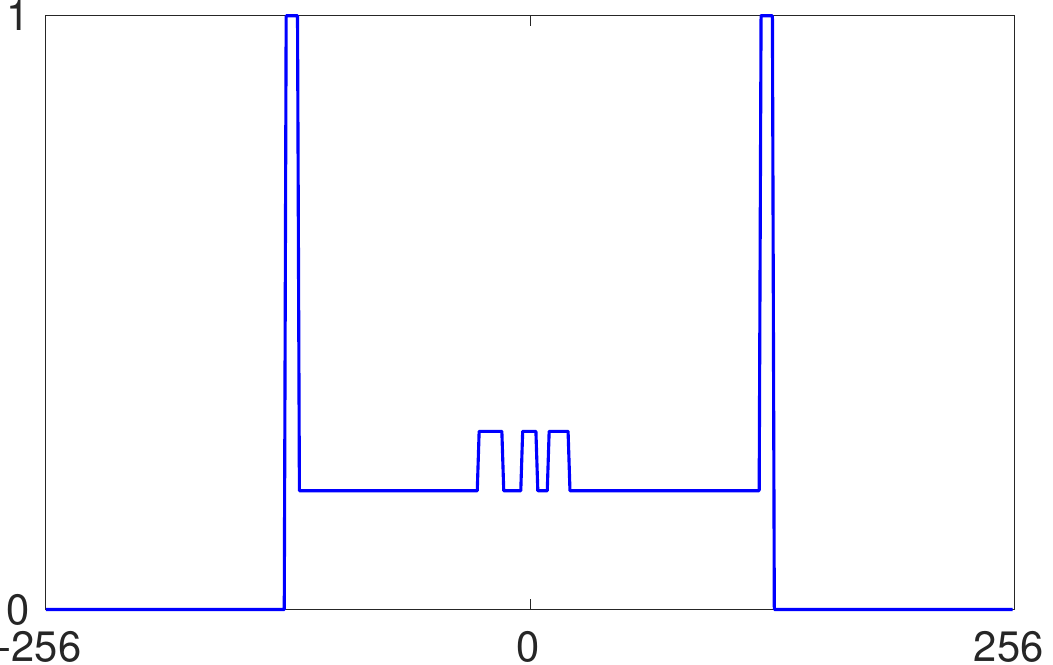}
    \end{subfigure}
    \caption{Phantom reconstruction on a random grid with \(\numcol = 512^{2}\) and \(\numrow = 1.5 \numcol\).
    Top: the real parts of the exact phantom and the direct reconstructions with \(\ranktol = 10^{-2}\) and \(10^{-4}\), from left to right.
    Bottom: the corresponding \(416\)th rows.
    The relative errors of the two reconstructions are \(1.4 \times 10^{-2}\) and \(1.4 \times 10^{-4}\), respectively.}
    \label{fig:rand_phantom_alpha_1.5}
\end{figure}

Table~\ref{tab:typeII_2d_rand_iterative_1.5} compares lsqr and plsqr.
Unpreconditioned lsqr reaches the limit of \(500\) iterations for every tested problem size.
The proposed preconditioner reduces the iteration count to between \(8\) and \(15\) for \(\ranktol = 10^{-2}\), and to \(4\) for \(\ranktol = 10^{-4}\).
All preconditioned runs achieve relative residuals below \(6 \times 10^{-13}\) and relative errors below \(2 \times 10^{-11}\).
At \(\numcol = 512^{2}\), the iterative solution times are approximately \(150\) and \(130\) seconds for the two tolerances.
Thus, the tighter tolerance reduces the iteration count, but the increased cost per iteration limits the reduction in iterative solution time.

\begin{table}[tbhp]
\centering
\begin{tabular}{cc c ccccc}
\toprule
\(N\) & \(M\) & method & \(t_{\pre}\) & \(t_{\iter}\) & \(n_{\iter}\) & \(r_{\solve}\) & \(e_{\solve}\) \\
\midrule
\multirow{3}{*}{\(32^2\)} & \multirow{3}{*}{1536} & lsqr & - & 1.4e+00 & 500 & 1.1e-08 & 1.4e-07 \\
 & & plsqr (\(\varepsilon = 10^{-2}\)) & 5.6e-01 & 1.0e-01 & 8 & 1.7e-13 & 1.6e-12 \\
 & & plsqr (\(\varepsilon = 10^{-4}\)) & 7.0e-01 & 7.5e-02 & 4 & 1.2e-15 & 1.0e-14 \\
\midrule
\multirow{3}{*}{\(64^2\)} & \multirow{3}{*}{6144} & lsqr & - & 2.2e+00 & 500 & 3.6e-09 & 7.2e-08 \\
 & & plsqr (\(\varepsilon = 10^{-2}\)) & 4.8e+00 & 6.6e-01 & 10 & 1.4e-13 & 1.4e-12 \\
 & & plsqr (\(\varepsilon = 10^{-4}\)) & 1.0e+01 & 5.3e-01 & 4 & 9.7e-15 & 9.4e-14 \\
\midrule
\multirow{3}{*}{\(128^2\)} & \multirow{3}{*}{24576} & lsqr & - & 5.2e+00 & 500 & 8.3e-07 & 2.8e-05 \\
 & & plsqr (\(\varepsilon = 10^{-2}\)) & 3.7e+01 & 4.2e+00 & 10 & 4.0e-13 & 7.7e-12 \\
 & & plsqr (\(\varepsilon = 10^{-4}\)) & 1.1e+02 & 4.3e+00 & 4 & 2.2e-14 & 3.5e-13 \\
\midrule
\multirow{3}{*}{\(256^2\)} & \multirow{3}{*}{98304} & lsqr & - & 2.1e+01 & 500 & 2.2e-06 & 1.2e-04 \\
 & & plsqr (\(\varepsilon = 10^{-2}\)) & 2.9e+02 & 2.5e+01 & 12 & 2.6e-13 & 6.2e-12 \\
 & & plsqr (\(\varepsilon = 10^{-4}\)) & 1.0e+03 & 2.4e+01 & 4 & 2.7e-14 & 3.5e-13 \\
\midrule
\multirow{3}{*}{\(512^2\)} & \multirow{3}{*}{393216} & lsqr & - & 8.8e+01 & 500 & 6.6e-06 & 4.3e-04 \\
 & & plsqr (\(\varepsilon = 10^{-2}\)) & 2.7e+03 & 1.5e+02 & 15 & 5.4e-13 & 1.5e-11 \\
 & & plsqr (\(\varepsilon = 10^{-4}\)) & 8.8e+03 & 1.3e+02 & 4 & 2.2e-13 & 4.5e-12 \\
\bottomrule
\end{tabular}
\caption{Preprocessing time, iterative solution time, iteration count, relative residual and relative error of lsqr and plsqr on random grids with \(\numrow = 1.5 \numcol\).}
\label{tab:typeII_2d_rand_iterative_1.5}
\end{table}

\subsection{Polar Grid}

We next consider the polar grids with \(\beta = 0.6\), for which
\(\numrow \approx 0.47 \numcol \mylog_{4}{\numcol}\).
Figure~\ref{fig:polar_rank_construct_factor_solve_beta_0.6} shows the HSS rank and the timings of the direct solver.
As in the random-grid case, the observed rank growth agrees well with
\(\bigO\bigl(\sqrt{\numcol} \mylog{\numcol}\bigr)\).
Although the oversampling ratio now increases with \(\numcol\), the leading complexity terms remain the same.
Indeed, the term \(\numrow \mylog[2]{\numcol}\) in the construction bound is
\(\bigO\bigl(\numcol \mylog[3]{\numcol}\bigr)\).
Thus, the construction and factorization bounds are dominated by
\(\bigO\bigl(\numcol^{3 / 2} \mylog[3]{\numcol}\bigr)\), while the solution bound remains
\(\bigO\bigl(\numcol \mylog[3]{\numcol}\bigr)\).

\begin{figure}[tbhp]
    \centering
    \begin{subfigure}{0.37\textwidth}
        \centering
        \includegraphics[width=\linewidth]{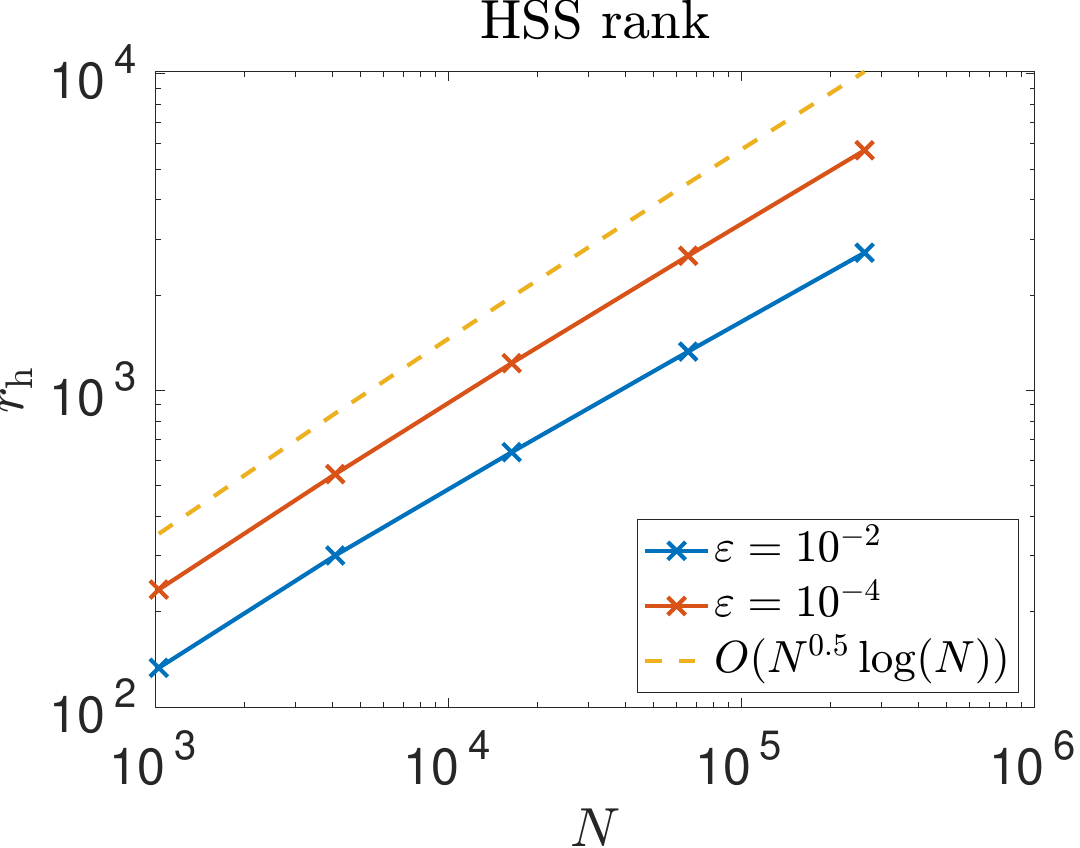}
    \end{subfigure}
    \hspace{0.01\textwidth}
    \begin{subfigure}{0.37\textwidth}
        \centering
        \includegraphics[width=\linewidth]{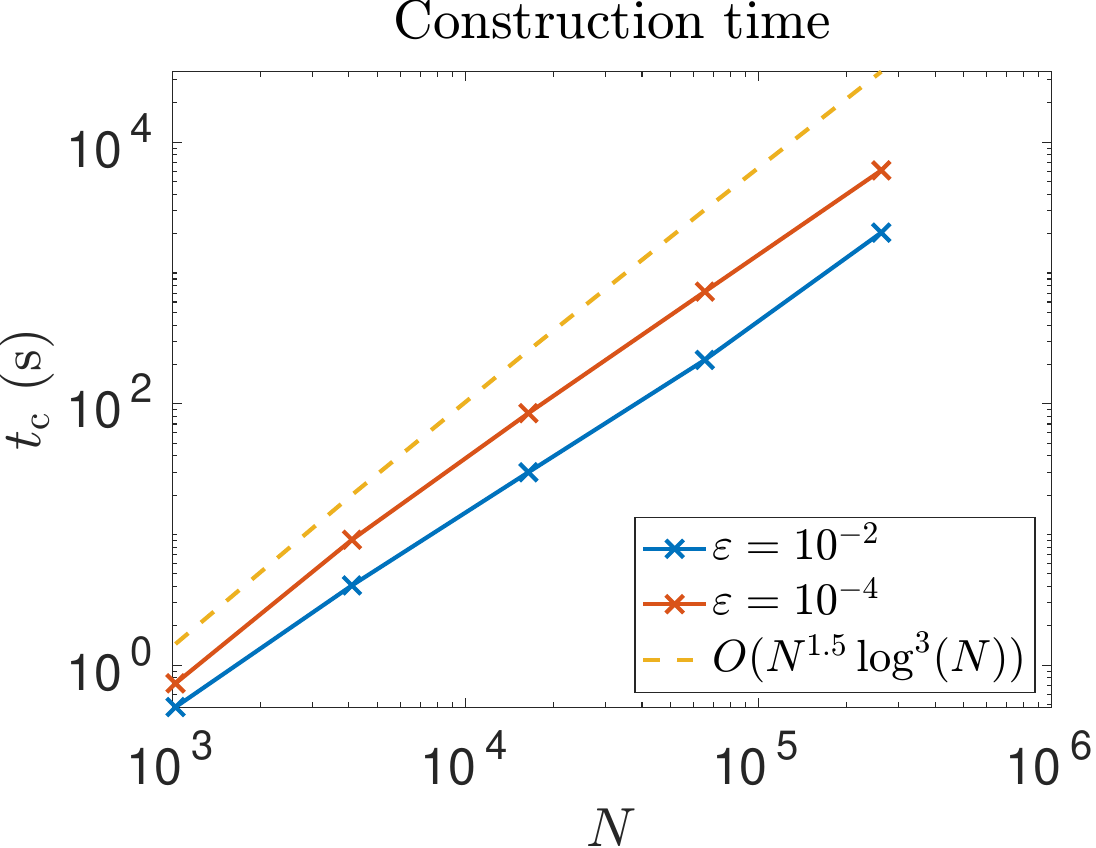}
    \end{subfigure}
    \begin{subfigure}{0.37\textwidth}
        \centering
        \includegraphics[width=\linewidth]{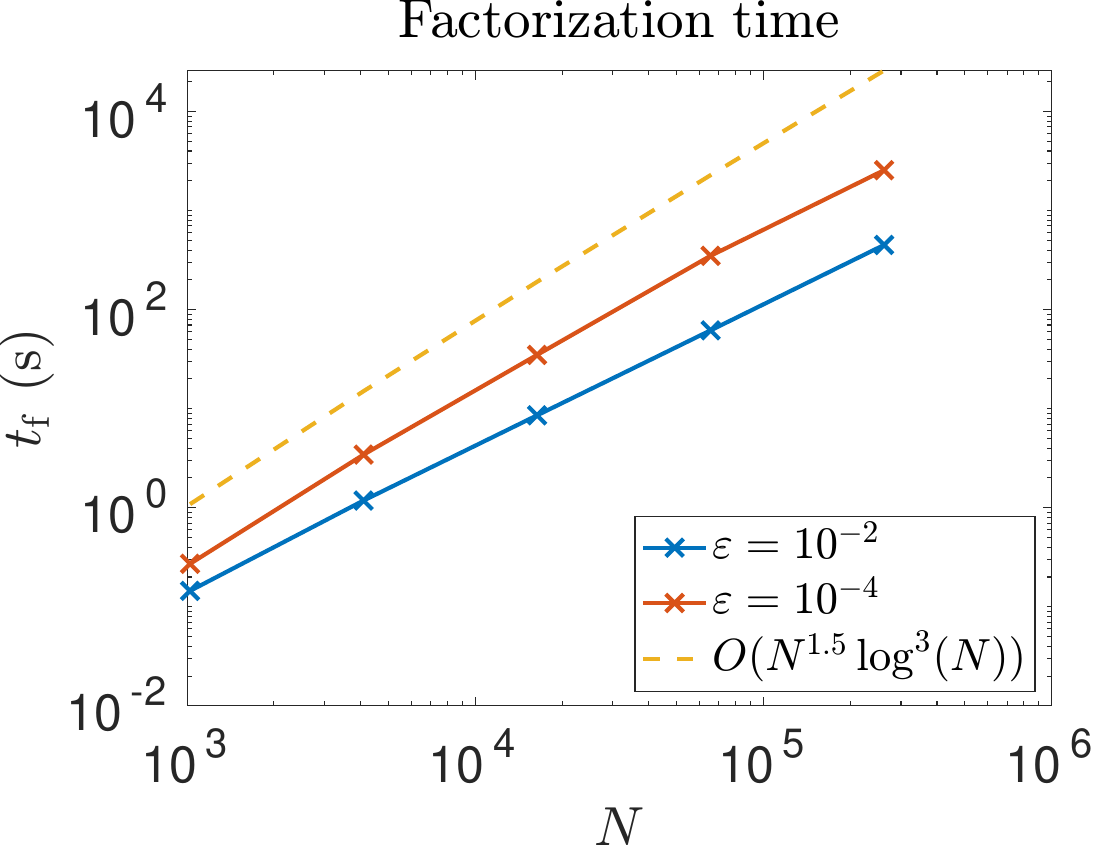}
    \end{subfigure}
    \hspace{0.01\textwidth}
    \begin{subfigure}{0.37\textwidth}
        \centering
        \includegraphics[width=\linewidth]{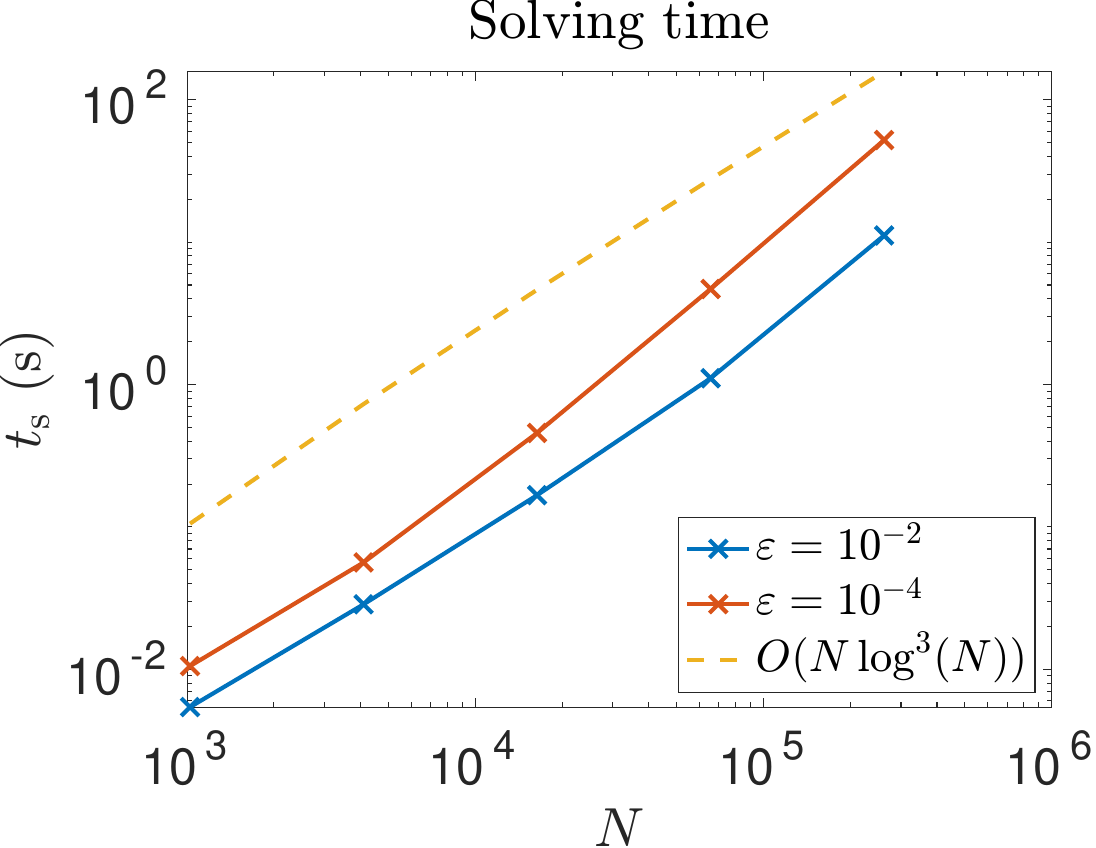}
    \end{subfigure}
    \caption{HSS rank, construction time, factorization time and solution time of the direct solver on polar grids with \(\beta = 0.6\) and \(\numrow \approx 0.47 \numcol \mylog_{4}{\numcol}\).
    Dashed curves show the theoretical reference scalings.}
    \label{fig:polar_rank_construct_factor_solve_beta_0.6}
\end{figure}

Table~\ref{tab:typeII_2d_polar_direct_0.6} reports the timings and reconstruction errors.
At \(\numcol = 512^{2}\), the combined construction and factorization times are approximately \(41\) minutes and \(2.4\) hours for \(\ranktol = 10^{-2}\) and \(10^{-4}\), while one solution takes \(11\) and \(52\) seconds, respectively.
The online solution is therefore substantially less expensive than the offline computation for both tolerances.

The direct reconstruction errors are larger than in the random-grid experiments.
For \(\ranktol = 10^{-2}\), the relative errors range from \(7.5 \times 10^{-2}\) to \(2.1 \times 10^{-1}\).
Tightening the tolerance to \(10^{-4}\) reduces this range to between \(2.3 \times 10^{-4}\) and \(2.0 \times 10^{-3}\).

\begin{table}[tbhp]
\centering
\begin{tabular}{cc c ccccc}
\toprule
\(N\) & \(M\) & \(\varepsilon\) & \(t_{\mathrm{c}}\) & \(t_{\mathrm{f}}\) & \(t_{\mathrm{s}}\) & \(r_{\mathrm{s}}\) & \(e_{\mathrm{s}}\) \\
\midrule
\multirow{2}{*}{\(32^2\)} & \multirow{2}{*}{2397} & \(10^{-2}\) & 4.8e-01 & 1.4e-01 & 5.4e-03 & 5.3e-03 & 8.5e-02 \\
 & & \(10^{-4}\) & 7.3e-01 & 2.7e-01 & 1.1e-02 & 4.2e-05 & 2.3e-04 \\
\midrule
\multirow{2}{*}{\(64^2\)} & \multirow{2}{*}{11620} & \(10^{-2}\) & 4.1e+00 & 1.2e+00 & 2.9e-02 & 6.2e-03 & 1.0e-01 \\
 & & \(10^{-4}\) & 9.1e+00 & 3.4e+00 & 5.6e-02 & 4.9e-05 & 7.4e-04 \\
\midrule
\multirow{2}{*}{\(128^2\)} & \multirow{2}{*}{54373} & \(10^{-2}\) & 3.0e+01 & 8.6e+00 & 1.7e-01 & 9.4e-03 & 2.1e-01 \\
 & & \(10^{-4}\) & 8.5e+01 & 3.5e+01 & 4.6e-01 & 7.0e-05 & 2.0e-03 \\
\midrule
\multirow{2}{*}{\(256^2\)} & \multirow{2}{*}{249074} & \(10^{-2}\) & 2.2e+02 & 6.2e+01 & 1.1e+00 & 1.7e-02 & 1.4e-01 \\
 & & \(10^{-4}\) & 7.2e+02 & 3.5e+02 & 4.7e+00 & 1.1e-04 & 2.0e-03 \\
\midrule
\multirow{2}{*}{\(512^2\)} & \multirow{2}{*}{1122019} & \(10^{-2}\) & 2.0e+03 & 4.5e+02 & 1.1e+01 & 2.3e-02 & 7.5e-02 \\
 & & \(10^{-4}\) & 6.1e+03 & 2.6e+03 & 5.2e+01 & 1.4e-04 & 2.7e-04 \\
\bottomrule
\end{tabular}
\caption{Construction time, factorization time, solution time, relative residual and relative error of the direct solver on polar grids with \(\numrow \approx 0.47 \numcol \mylog_{4}{\numcol}\).}
\label{tab:typeII_2d_polar_direct_0.6}
\end{table}

Figure~\ref{fig:polar_phantom_beta_0.6} shows the reconstructions and row profiles for \(\numcol = 512^{2}\).
As in the random-grid case, both reconstructed images are visually close to the exact phantom at the displayed scale, and the row profiles recover the three small interior features at both tolerances.
For \(\ranktol = 10^{-2}\), small deviations remain in the constant regions.
With \(\ranktol = 10^{-4}\), the reconstructed profile closely follows the exact one at the displayed scale.

\begin{figure}[tbhp]
    \centering
    \begin{subfigure}{0.3\textwidth}
        \centering
        \includegraphics[width=\linewidth]{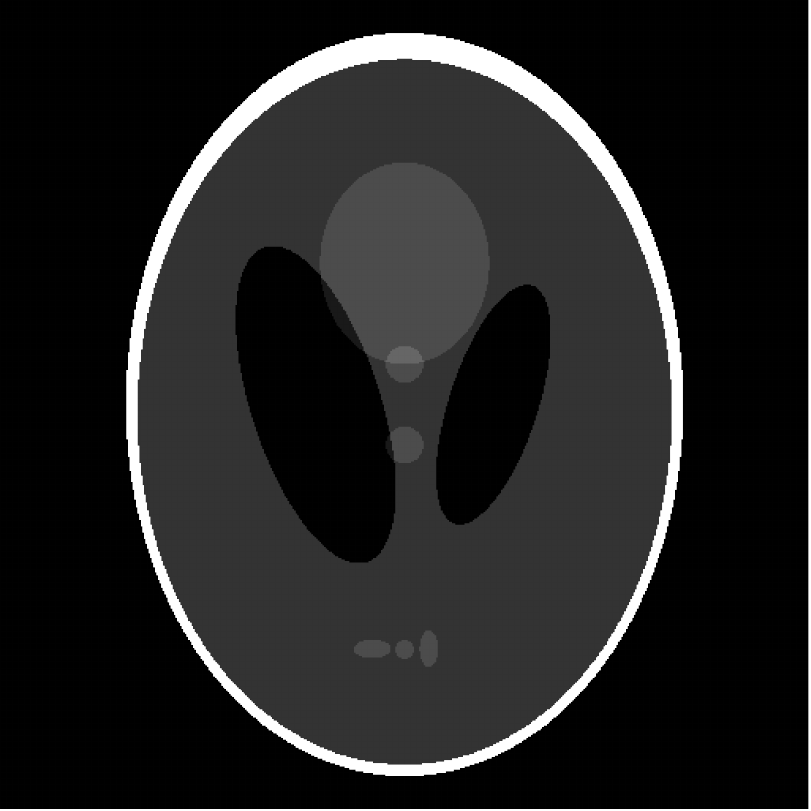}
    \end{subfigure}
    \hspace{0.01\textwidth}
    \begin{subfigure}{0.3\textwidth}
        \centering
        \includegraphics[width=\linewidth]{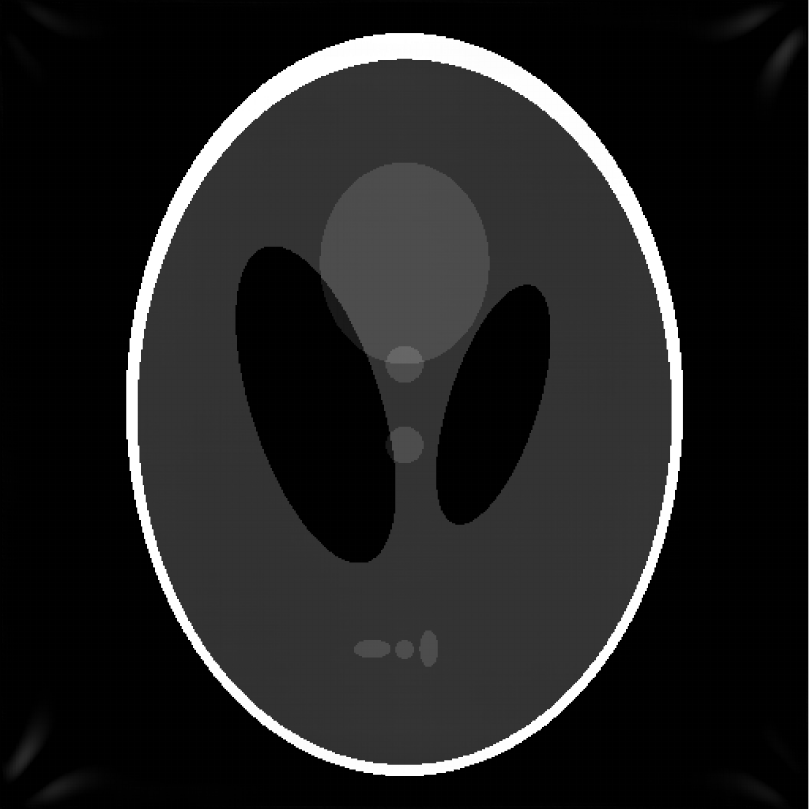}
    \end{subfigure}
    \hspace{0.01\textwidth}
    \begin{subfigure}{0.3\textwidth}
        \centering
        \includegraphics[width=\linewidth]{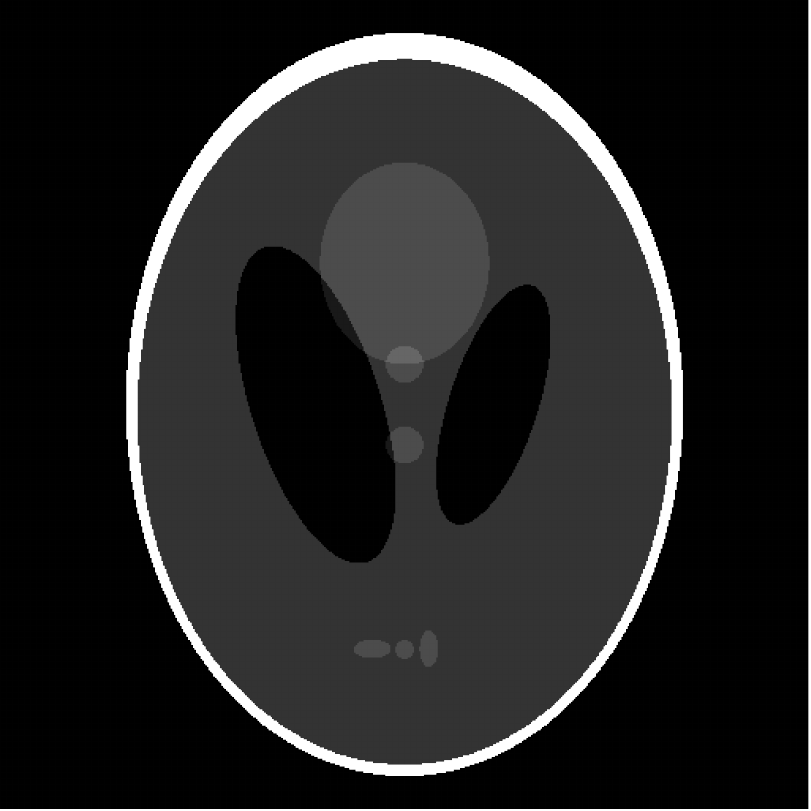}
    \end{subfigure}
    \par\medskip
    \begin{subfigure}{0.3\textwidth}
        \centering
        \includegraphics[width=\linewidth]{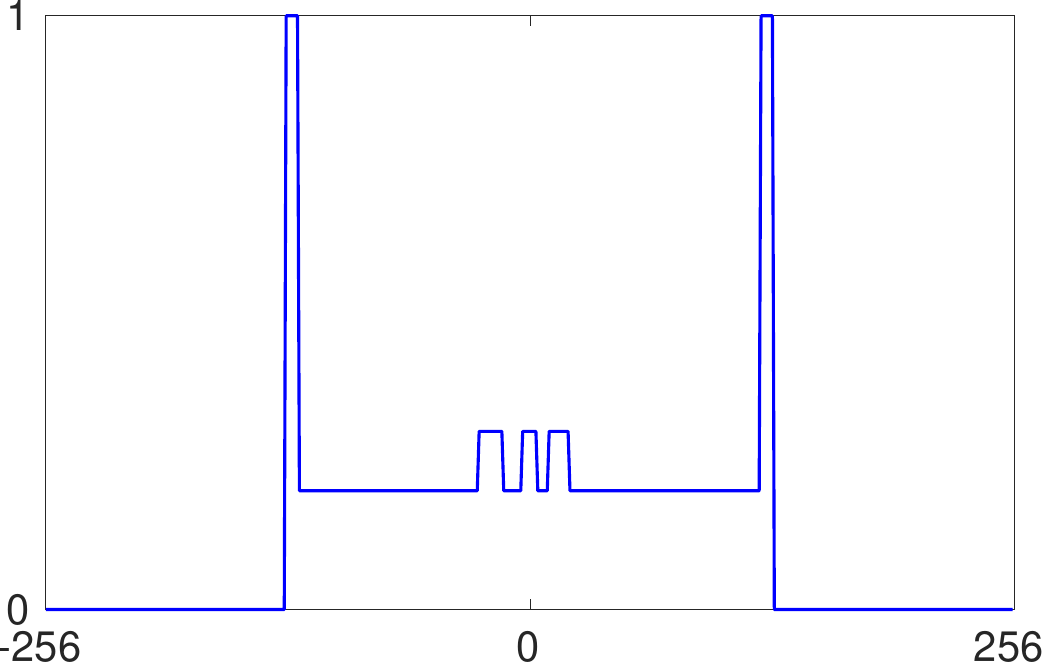}
    \end{subfigure}
    \hspace{0.01\textwidth}
    \begin{subfigure}{0.3\textwidth}
        \centering
        \includegraphics[width=\linewidth]{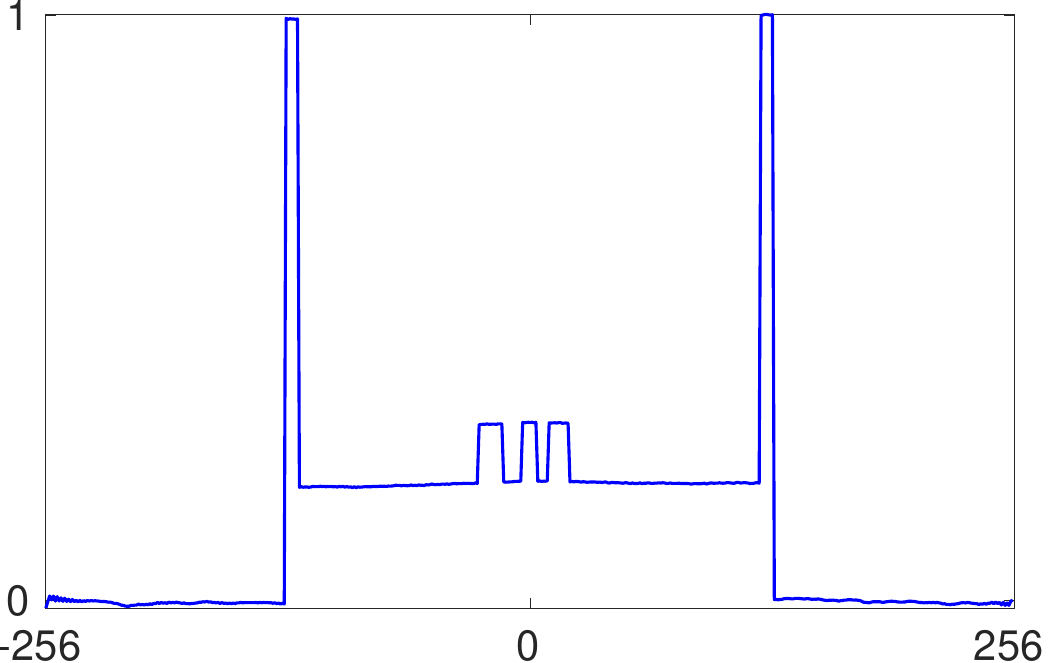}
    \end{subfigure}
    \hspace{0.01\textwidth}
    \begin{subfigure}{0.3\textwidth}
        \centering
        \includegraphics[width=\linewidth]{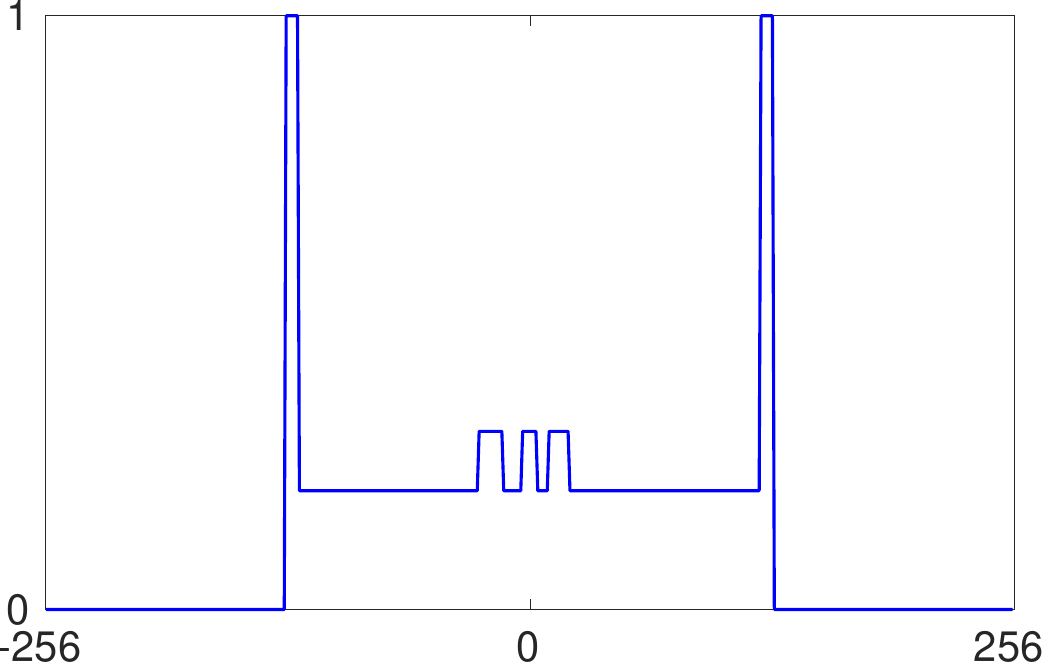}
    \end{subfigure}
    \caption{Phantom reconstruction on a polar grid with \(\numcol = 512^{2}\), \(\beta = 0.6\), and \(\numrow = 1122019\).
    Top: the real parts of the exact phantom and the direct reconstructions with \(\ranktol = 10^{-2}\) and \(10^{-4}\), from left to right.
    Bottom: the corresponding \(416\)th rows.
    The relative errors of the two reconstructions are \(7.5 \times 10^{-2}\) and \(2.7 \times 10^{-4}\), respectively.}
    \label{fig:polar_phantom_beta_0.6}
\end{figure}

The iterative results are reported in Table~\ref{tab:typeII_2d_polar_iterative_0.6}.
Unpreconditioned lsqr terminates after \(470\) iterations for the smallest problem and reaches the limit of \(500\) iterations for all larger problems.
For the largest problem, its relative residual and relative error are \(2.0 \times 10^{-5}\) and \(6.5 \times 10^{-4}\), respectively.
With the proposed preconditioner, the iteration counts are between \(16\) and \(32\) for \(\ranktol = 10^{-2}\), and between \(5\) and \(7\) for \(\ranktol = 10^{-4}\).
All preconditioned runs achieve relative residuals below \(10^{-12}\) and relative errors below \(10^{-10}\).

At \(\numcol = 512^{2}\) and \(\ranktol = 10^{-2}\), the direct solution has relative error \(7.5 \times 10^{-2}\), whereas plsqr reduces it to \(3.4 \times 10^{-12}\) in \(16\) iterations.
Tightening the HSS tolerance to \(10^{-4}\) reduces the iteration count to \(5\), but the iterative times remain similar, at approximately \(150\) and \(140\) seconds.
For this problem, the coarser HSS approximation therefore provides accurate iterative reconstruction with substantially less preprocessing and a similar iterative solution time.

\begin{table}[tbhp]
\centering
\begin{tabular}{cc c ccccc}
\toprule
\(N\) & \(M\) & method & \(t_{\pre}\) & \(t_{\iter}\) & \(n_{\iter}\) & \(r_{\solve}\) & \(e_{\solve}\) \\
\midrule
\multirow{3}{*}{\(32^2\)} & \multirow{3}{*}{2397} & lsqr & - & 1.5e+00 & 470 & 9.4e-13 & 3.9e-12 \\
 & & plsqr (\(\varepsilon = 10^{-2}\)) & 6.2e-01 & 2.0e-01 & 16 & 1.6e-13 & 4.3e-12 \\
 & & plsqr (\(\varepsilon = 10^{-4}\)) & 1.0e+00 & 1.3e-01 & 5 & 1.9e-13 & 1.3e-11 \\
\midrule
\multirow{3}{*}{\(64^2\)} & \multirow{3}{*}{11620} & lsqr & - & 3.0e+00 & 500 & 2.6e-05 & 5.7e-03 \\
 & & plsqr (\(\varepsilon = 10^{-2}\)) & 5.3e+00 & 1.5e+00 & 22 & 1.6e-13 & 7.0e-12 \\
 & & plsqr (\(\varepsilon = 10^{-4}\)) & 1.3e+01 & 8.2e-01 & 6 & 1.3e-13 & 5.2e-12 \\
\midrule
\multirow{3}{*}{\(128^2\)} & \multirow{3}{*}{54373} & lsqr & - & 1.0e+01 & 500 & 2.3e-05 & 1.6e-03 \\
 & & plsqr (\(\varepsilon = 10^{-2}\)) & 3.8e+01 & 1.2e+01 & 32 & 8.4e-13 & 9.5e-11 \\
 & & plsqr (\(\varepsilon = 10^{-4}\)) & 1.2e+02 & 7.0e+00 & 7 & 1.8e-13 & 5.3e-11 \\
\midrule
\multirow{3}{*}{\(256^2\)} & \multirow{3}{*}{249074} & lsqr & - & 4.2e+01 & 500 & 5.1e-05 & 2.6e-03 \\
 & & plsqr (\(\varepsilon = 10^{-2}\)) & 2.8e+02 & 5.4e+01 & 29 & 3.8e-13 & 1.7e-11 \\
 & & plsqr (\(\varepsilon = 10^{-4}\)) & 1.1e+03 & 3.9e+01 & 7 & 7.2e-14 & 2.2e-11 \\
\midrule
\multirow{3}{*}{\(512^2\)} & \multirow{3}{*}{1122019} & lsqr & - & 2.0e+02 & 500 & 2.0e-05 & 6.5e-04 \\
 & & plsqr (\(\varepsilon = 10^{-2}\)) & 2.5e+03 & 1.5e+02 & 16 & 2.4e-13 & 3.4e-12 \\
 & & plsqr (\(\varepsilon = 10^{-4}\)) & 8.7e+03 & 1.4e+02 & 5 & 3.0e-13 & 8.7e-12 \\
\bottomrule
\end{tabular}
\caption{Preprocessing time, iterative solution time, iteration count, relative residual and relative error of lsqr and plsqr on polar grids with \(\numrow \approx 0.47 \numcol \mylog_{4}{\numcol}\).}
\label{tab:typeII_2d_polar_iterative_0.6}
\end{table}

\section{Conclusions and Future Directions} \label{sec:conclusion_future_directions}

This paper develops an algebraic HSS construction based on the face-splitting product and applies it to direct inversion of the 2D type-II NUDFT.
We define the face-splitting product of two HSS trees and show that the product of two HSS matrices admits an HSS representation on the resulting tree.
The constructive proof provides explicit formulas for the HSS generators and bounds on the off-diagonal ranks.
For the transformed NUDFT matrix \(\tnudftmat = \nudftmat \dftmatinv\), this construction combines two 1D HSS approximations, and we establish a Frobenius-norm approximation error bound.

Recompression and URV factorization, followed by a two-dimensional iFFT in the solution stage, give a direct solver for the associated least-squares problem.
For fixed accuracy, square frequency grids, and balanced HSS trees with fixed leaf column sizes, the offline and online complexities are \(\offlinecomplexity\) and \(\onlinecomplexity\), respectively.
The offline factors can be reused for different right-hand sides with the same sampling geometry and coefficient grid.
Numerical experiments on random and polar grids examine the computational scaling and the accuracy of image reconstruction from nonuniform Fourier data.
The polar-grid experiments are motivated by Fourier reconstruction in CT.
The results also demonstrate that the HSS approximation substantially reduces lsqr iteration counts and improves the solution accuracy attained within the prescribed iteration limit.

Future work includes extending the solver to the 2D type-III NUDFT following~\cite{Li_Liu_2025}, and establishing rigorous numerical rank bounds for the HSS blocks of \(\tnudftmat\).
An analogue of Lemma~\ref{lem:face_splitting_exact_rank} for numerical rank could help address the latter.
Higher-dimensional extensions also merit investigation, since the rank model suggests that the cost of the present HSS construction increases rapidly with dimension.
Alternative representations, such as butterfly factorizations~\cite{Candes_Demanet_Ying_2009, Li_Yang_Martin_Ho_Ying_2015}, may provide a more economical representation of the transformed matrix.
Developing efficient inverse algorithms or preconditioners from such representations remains an open direction.

\section*{Acknowledgments}

YL was supported in part by the National Natural Science Foundation of China under Grant Nos. 12271109 and 12526211; by the Shanghai Pilot Program for Basic Research-Fudan University under Grant No. 21TQ1400100 (22TQ017); by the Scientific Research Innovation Capability Support Project for Young Faculty under Grant No. SRICSPYF-ZY2025159; and by the Xuemin Institute of Advanced Studies, Fudan University.

\bibliographystyle{siamplain}
\bibliography{ref}

@article{Cooley_Tukey_1965,
   author = {Cooley, James W. and Tukey, John W.},
   title = {An algorithm for the machine calculation of complex {Fourier} series},
   journal = {Mathematics of Computation},
   volume = {19},
   number = {90},
   pages = {297-301},
   ISSN = {00255718, 10886842},
   year = {1965},
   type = {Journal Article}
}

@article{Dutt_Rokhlin_1993,
   author = {Dutt, Alok and Rokhlin, Vladimir},
   title = {Fast {Fourier} transforms for nonequispaced data},
   journal = {SIAM Journal on Scientific Computing},
   volume = {14},
   number = {6},
   pages = {1368-1393},
   ISSN = {1064-8275},
   year = {1993},
   type = {Journal Article}
}

@article{Dutt_Rokhlin_1995,
   author = {Dutt, Alok and Rokhlin, Vladimir},
   title = {Fast {Fourier} transforms for nonequispaced data, {II}},
   journal = {Applied and Computational Harmonic Analysis},
   volume = {2},
   number = {1},
   pages = {85-100},
   ISSN = {1063-5203},
   year = {1995},
   type = {Journal Article}
}

@article{Anderson_Dahleh_1996,
   author = {Anderson, Chris and Dahleh, Marie Dillon},
   title = {Rapid computation of the discrete {Fourier} transform},
   journal = {SIAM Journal on Scientific Computing},
   volume = {17},
   number = {4},
   pages = {913-919},
   year = {1996},
   type = {Journal Article}
}

@inbook{Bourgeois_Wajer_vanOrmondt_GraveronDemilly_2001,
   author = {Bourgeois, Marc and Wajer, Frank T. A. W. and van Ormondt, Dirk and Graveron-Demilly, Danielle},
   title = {Reconstruction of {MRI} images from non-uniform sampling and its application to intrascan motion correction in functional {MRI}},
   booktitle = {Modern Sampling Theory: Mathematics and Applications},
   publisher = {Birkhäuser Boston},
   address = {Boston, MA},
   pages = {343-363},
   ISBN = {978-1-4612-0143-4},
   year = {2001},
   type = {Book Section}
}

@inbook{Potts_Steidl_Tasche_2001,
   author = {Potts, Daniel and Steidl, Gabriele and Tasche, Manfred},
   title = {Fast {Fourier} transforms for nonequispaced data: A tutorial},
   booktitle = {Modern Sampling Theory: Mathematics and Applications},
   publisher = {Birkhäuser Boston},
   address = {Boston, MA},
   pages = {247-270},
   ISBN = {978-1-4612-0143-4},
   year = {2001},
   type = {Book Section}
}

@article{Greengard_Lee_2004,
   author = {Greengard, Leslie and Lee, June-Yub},
   title = {Accelerating the nonuniform fast {Fourier} transform},
   journal = {SIAM Review},
   volume = {46},
   number = {3},
   pages = {443-454},
   ISSN = {0036-1445},
   year = {2004},
   type = {Journal Article}
}

@article{Lee_Greengard_2005,
   author = {Lee, June-Yub and Greengard, Leslie},
   title = {The type 3 nonuniform {FFT} and its applications},
   journal = {Journal of Computational Physics},
   volume = {206},
   number = {1},
   pages = {1-5},
   ISSN = {0021-9991},
   year = {2005},
   type = {Journal Article}
}

@article{Fenn_Kunis_Potts_2007,
   author = {Fenn, Markus and Kunis, Stefan and Potts, Daniel},
   title = {On the computation of the polar {FFT}},
   journal = {Appl. Comput. Harmon. Anal.},
   volume = {22},
   number = {2},
   pages = {257-263},
   ISSN = {1063-5203},
   year = {2007},
   type = {Journal Article}
}

@book{Bagchi_Mitra_2012,
  author    = {Bagchi, Sonali and Mitra, Sanjit K.},
  title     = {The nonuniform discrete {Fourier} transform and its applications in signal processing},
  publisher = {Springer},
  address   = {Boston, MA},
  year      = {2012},
  isbn      = {978-1-4615-4925-3}
}

@article{Ruiz-Antolín_Townsend_2018,
   author = {Ruiz-Antolín, Diego and Townsend, Alex},
   title = {A nonuniform fast {Fourier} transform based on low rank approximation},
   journal = {SIAM Journal on Scientific Computing},
   volume = {40},
   number = {1},
   pages = {A529-A547},
   ISSN = {1064-8275},
   year = {2018},
   type = {Journal Article}
}

@article{Kircheis_Potts_2019,
   author = {Kircheis, Melanie and Potts, Daniel},
   title = {Direct inversion of the nonequispaced fast {Fourier} transform},
   journal = {Linear Algebra and its Applications},
   volume = {575},
   pages = {106-140},
   ISSN = {0024-3795},
   year = {2019},
   type = {Journal Article}
}

@article{Barnett_Magland_Klinteberg_2019,
   author = {Barnett, Alexander H. and Magland, Jeremy and af Klinteberg, Ludvig},
   title = {A parallel nonuniform fast {Fourier} transform library based on an “exponential of semicircle" kernel},
   journal = {SIAM Journal on Scientific Computing},
   volume = {41},
   number = {5},
   pages = {C479-C504},
   ISSN = {1064-8275},
   year = {2019},
   type = {Journal Article}
}

@book{Hansen_Jorgensen_Lionheart_2021,
  author    = {Hansen, Per Christian and J{\o}rgensen, Jakob and Lionheart, William R. B.},
  title     = {Computed tomography: Algorithms, insight, and just enough theory},
  publisher = {Society for Industrial and Applied Mathematics},
  address   = {Philadelphia, PA},
  year      = {2021},
  doi       = {10.1137/1.9781611976670}
}

@article{Kircheis_Potts_2023,
   author = {Kircheis, Melanie and Potts, Daniel},
   title = {Fast and direct inversion methods for the multivariate nonequispaced fast {Fourier} transform},
   journal = {Front. Appl. Math. Stat.},
   volume = {Volume 9},
   ISSN = {2297-4687},
   year = {2023},
   type = {Journal Article}
}

@article{Wilber_Epperly_Barnett_2025,
   author = {Wilber, Heather and Epperly, Ethan N. and Barnett, Alex H.},
   title = {Superfast direct inversion of the nonuniform discrete {Fourier} transform via hierarchically semiseparable least squares},
   journal = {SIAM Journal on Scientific Computing},
   pages = {A1702-A1732},
   ISSN = {1064-8275},
   year = {2025},
   type = {Journal Article}
}

@article{Jiang_Greengard_2025,
   author = {Jiang, Shidong and Greengard, Leslie},
   title = {A dual-space multilevel kernel-splitting framework for discrete and continuous convolution},
   journal = {Commun. Pure Appl. Math.},
   volume = {78},
   number = {5},
   pages = {1086-1143},
   ISSN = {0010-3640},
   year = {2025},
   type = {Journal Article}
}

@article{Li_Liu_2025,
  author        = {Li, Yingzhou and Liu, Jingyu},
  title         = {A superfast direct solver for {type-III} inverse nonuniform discrete {Fourier} transform},
  journal       = {arXiv preprint},
  year          = {2025},
  eprint        = {2512.03733},
  archivePrefix = {arXiv},
  primaryClass  = {math.NA}
}

@article{Hackbusch_1999,
   author = {Hackbusch, Wolfgang},
   title = {A sparse matrix arithmetic based on {$\mathcal{H}$}-matrices. {Part I}: Introduction to {$\mathcal{H}$}-matrices},
   journal = {Computing},
   volume = {62},
   number = {2},
   pages = {89-108},
   ISSN = {0010-485X},
   year = {1999},
   type = {Journal Article}
}

@article{Borm_Grasedyck_Hackbusch_2003,
   author = {B{\"o}rm, Steffen and Grasedyck, Lars and Hackbusch, Wolfgang},
   title = {Introduction to hierarchical matrices with applications},
   journal = {Engineering Analysis with Boundary Elements},
   volume = {27},
   number = {5},
   pages = {405-422},
   ISSN = {0955-7997},
   year = {2003},
   type = {Journal Article}
}

@article{Martinsson_Rokhlin_2005,
   author = {Martinsson, Per-Gunnar and Rokhlin, Vladimir},
   title = {A fast direct solver for boundary integral equations in two dimensions},
   journal = {Journal of Computational Physics},
   volume = {205},
   number = {1},
   pages = {1-23},
   ISSN = {0021-9991},
   year = {2005},
   type = {Journal Article}
}

@article{Candes_Demanet_Ying_2009,
   author = {Cand{\`{e}}s, Emmanuel and Demanet, Laurent and Ying, Lexing},
   title = {A fast butterfly algorithm for the computation of {Fourier} integral operators},
   journal = {Multiscale Model. Simul.},
   volume = {7},
   number = {4},
   pages = {1727-1750},
   ISSN = {1540-3459},
   year = {2009},
   type = {Journal Article}
}

@article{Xia_Chandrasekaran_Gu_Li_2010,
   author = {Xia, Jianlin and Chandrasekaran, Shivkumar and Gu, Ming and Li, Xianye S.},
   title = {Fast algorithms for hierarchically semiseparable matrices},
   journal = {Numerical Linear Algebra with Applications},
   volume = {17},
   number = {6},
   pages = {953-976},
   ISSN = {1070-5325},
   year = {2010},
   type = {Journal Article}
}

@article{Xia_2012,
   author = {Xia, Jianlin},
   title = {On the complexity of some hierarchical structured matrix algorithms},
   journal = {SIAM J. Matrix Anal. Appl.},
   volume = {33},
   number = {2},
   pages = {388-410},
   year = {2012},
   type = {Journal Article}
}

@article{Xi_Xia_Cauley_Balakrishnan_2014,
   author = {Xi, Yuanzhe and Xia, Jianlin and Cauley, Stephen and Balakrishnan, Venkataramanan},
   title = {Superfast and stable structured solvers for {Toeplitz} least squares via randomized sampling},
   journal = {SIAM Journal on Matrix Analysis and Applications},
   volume = {35},
   number = {1},
   pages = {44-72},
   ISSN = {0895-4798},
   year = {2014},
   type = {Journal Article}
}

@article{Corona_Martinsson_Zorin_2015,
   author = {Corona, Eduardo and Martinsson, Per-Gunnar and Zorin, Denis},
   title = {An {$O(N)$} direct solver for integral equations on the plane},
   journal = {Applied and Computational Harmonic Analysis},
   volume = {38},
   number = {2},
   pages = {284-317},
   ISSN = {1063-5203},
   year = {2015},
   type = {Journal Article}
}

@article{Li_Yang_Martin_Ho_Ying_2015,
   author = {Li, Yingzhou and Yang, Haizhao and Martin, Eileen R. and Ho, Kenneth L. and Ying, Lexing},
   title = {Butterfly factorization},
   journal = {Multiscale Model. Simul.},
   volume = {13},
   number = {2},
   pages = {714-732},
   ISSN = {1540-3459},
   year = {2015},
   type = {Journal Article}
}

@article{Beckermann_Townsend_2017,
   author = {Beckermann, Bernhard and Townsend, Alex},
   title = {On the singular values of matrices with displacement structure},
   journal = {SIAM Journal on Matrix Analysis and Applications},
   volume = {38},
   number = {4},
   pages = {1227-1248},
   ISSN = {0895-4798},
   year = {2017},
   type = {Journal Article}
}

@article{Paige_Saunders_1982,
   author = {Paige, Christopher C. and Saunders, Michael A.},
   title = {{LSQR}: An algorithm for sparse linear equations and sparse least squares},
   journal = {ACM Trans. Math. Softw.},
   volume = {8},
   number = {1},
   pages = {43-71},
   ISSN = {0098-3500},
   year = {1982},
   type = {Journal Article}
}

@article{Cheng_Gimbutas_Martinsson_Rokhlin_2005,
   author = {Cheng, Hongwei and Gimbutas, Zydrunas and Martinsson, Per-Gunnar and Rokhlin, Vladimir},
   title = {On the compression of low rank matrices},
   journal = {SIAM Journal on Scientific Computing},
   volume = {26},
   number = {4},
   pages = {1389-1404},
   ISSN = {1064-8275},
   year = {2005},
   type = {Journal Article}
}

@article{Benner_Li_Truhar_2009,
   author = {Benner, Peter and Li, Ren-Cang and Truhar, Ninoslav},
   title = {On the {ADI} method for {Sylvester} equations},
   journal = {J. Comput. Appl. Math.},
   volume = {233},
   number = {4},
   pages = {1035-1045},
   ISSN = {0377-0427},
   year = {2009},
   type = {Journal Article}
}

\end{document}